\documentclass[twosided,a4paper]{amsart}

\usepackage{xypic}
\input xy
\xyoption{all}
\usepackage{epsfig}
\usepackage{color}
\usepackage{amsthm}
\usepackage{changepage}
\usepackage{amssymb}
\usepackage{mathdots}
\usepackage{amsmath}
\usepackage{amscd}
\usepackage{amsopn}
\usepackage{enumerate}
\usepackage{mathtools}
\usepackage{url}
\usepackage{graphicx}
\usepackage{hyperref}

\usepackage{hyperref}\hypersetup{colorlinks}

\usepackage{color} 

\definecolor{darkred}{rgb}{1,0,0} 
\definecolor{darkgreen}{rgb}{0,0.8,0}
\definecolor{darkblue}{rgb}{0,0,1}

\hypersetup{colorlinks,
linkcolor=darkblue,
filecolor=darkgreen,
urlcolor=darkred,
citecolor=darkgreen}

 \numberwithin{equation}{section}

\newtheorem {Theorem}{Theorem}

\numberwithin{Theorem}{section}
\newtheorem {Conjecture}[Theorem]  {Conjecture}
\newtheorem {Lemma}[Theorem]    {Lemma}

\newtheorem {Proposition}[Theorem]{Proposition}
\newtheorem {Corollary}[Theorem]{Corollary}
\theoremstyle{definition}
\newtheorem{Definition}[Theorem]{Definition}

\newtheorem{Remark}[Theorem]{Remark}
\newtheorem{Example}[Theorem]{Example}

\newtheoremstyle{MyTheorem}
        {.6em}{.6em}              
        {\itshape}                      
        {}                              
        {\bfseries}                     
        {. }                             
        { }                             
        {\thmname{#1}\thmnumber{\addtocounter{MyTheorem}{-4}#2}\thmnote{ \bfseries #3}}
\theoremstyle{MyTheorem}
\newtheorem{MyTheorem}{Theorem}

\newtheoremstyle{TheoremForIntro} 
        {.6em}{.6em}              
        {\itshape}                      
        {}                              
        {\bfseries}                     
        {.}                             
        { }                             
        {\thmname{#1}\thmnote{ \bfseries #3}}
    \theoremstyle{TheoremForIntro}
    \newtheorem{TheoremIntro}[MyTheorem]{Theorem}

\expandafter\chardef\csname pre amssym.def at\endcsname=\the\catcode`\@
\catcode`\@=11
\def\undefine#1{\let#1\undefined}
\def\newsymbol#1#2#3#4#5{\let\next@\relax
 \ifnum#2=\@ne\let\next@\msafam@\else
 \ifnum#2=\tw@\let\next@\msbfam@\fi\fi
 \mathchardef#1="#3\next@#4#5}
\def\mathhexbox@#1#2#3{\relax
 \ifmmode\mathpalette{}{\m@th\mathchar"#1#2#3}%
 \else\leavevmode\hbox{$\m@th\mathchar"#1#2#3$}\fi}
\def\hexnumber@#1{\ifcase#1 0\or 1\or 2\or 3\or 4\or 5\or 6\or 7\or 8\or
 9\or A\or B\or C\or D\or E\or F\fi}

\font\teneufm=eufm10
\font\seveneufm=eufm7
\font\fiveeufm=eufm5
\newfam\eufmfam
\textfont\eufmfam=\teneufm
\scriptfont\eufmfam=\seveneufm
\scriptscriptfont\eufmfam=\fiveeufm

\catcode`\@=\csname pre amssym.def at\endcsname

\newcommand{\fsl}{{\mathfrak {sl}}}
\newcommand{\fgl}{{\mathfrak {gl}}}
\newcommand{\fso}{{\mathfrak {so}}}

\newcommand{\fg}{{\mathfrak g}}
\newcommand{\fh}{{\mathfrak h}}

\newcommand{\ft}{{\mathfrak t}}

\newcommand{\fp}{{\mathfrak p}}
\newcommand{\fl}{{\mathfrak l}}
\newcommand{\fm}{{\mathfrak m}}

\newcommand{\fu}{{\mathfrak u}}

\newcommand{\sSU}{{\mathsf {SU}}}
\newcommand{\sSL}{{\mathsf {SL}}}
\newcommand{\sSO}{{\mathsf {SO}}}
\newcommand{\sO}{{\mathsf {O}}}
\newcommand{\sSp}{{\mathsf {Sp}}}
\newcommand{\sSpin}{{\mathsf {Spin}}}
\newcommand{\sGL}{{\mathsf {GL}}}

\newcommand{\sG}{{\mathsf G}}

\newcommand{\sH}{{\mathsf H}}
\newcommand{\sS}{{\mathsf S}}

\newcommand{\sT}{{\mathsf T}}

\newcommand{\sP}{{\mathsf P}}
\newcommand{\sL}{{\mathsf L}}

\newcommand{\sU}{{\mathsf U}}

\newcommand{\Sym}{{\mathsf{Sym}}}
\newcommand{\Bun}{{\mathsf{Bun}}}
\newcommand{\Stab}{{\mathsf{Stab}}}

\newcommand{\Fuch}{{\mathsf{Fuch}}}
\newcommand{\Hom}{{\mathsf{Hom}}}

\newcommand{\Pic}{{\mathsf{Pic}}}
\newcommand{\Prym}{{\mathsf{Prym}}}
\newcommand{\End}{{\mathsf{End}}}
\newcommand{\Hit}{{\mathsf{Hit}}}

\newcommand{\Aut}{{\mathsf{Aut}}}

\newcommand{\Xx}{{\mathcal X}}

\newcommand{\Ee}{{\mathcal E}}
\newcommand{\Ff}{{\mathcal F}}

\newcommand{\Gg}{{\mathcal G}}

\newcommand{\Mm}{{\mathcal M}}

\newcommand{\Oo}{{\mathcal O}}
\newcommand{\Pp}{{\mathcal P}}

\newcommand{\Ww}{{\mathcal W}}
\newcommand{\Zz}{{\mathcal Z}}

\newcommand{\ip}[2]{\langle #1, #2 \rangle}
\newcommand{\ipd}{\ip{\cdot}{\cdot}}

\newcommand{\floor}[1]{\left\lfloor #1\right\rfloor}

\newcommand{\mtrx}[1]{\left (\begin{matrix}#1\end{matrix}\right)}

\def    \HH     {{\mathbb H}}

\def    \C      {{\mathbb C}}
\def    \R      {{\mathbb R}}
\def    \Z      {{\mathbb Z}}

\def    \Q      {{\mathbb Q}}

\def    \P    {{\mathbb P}}

\def    \ra     {{\rightarrow}}

\def    \p      {\partial}

\def    \H  {\operatorname{\scriptscriptstyle{H}}}

\newcommand{\An}{\xymatrix{ *{\circ} \ar@{-}[r]|*\dir{ } & *{\circ}\ar@{-}[r]&{\cdots}&*{\circ}\ar@{-}[l]|*\dir{ }\ar@{-}[r]|*\dir{ }&*{\circ} }}
\newcommand{\Anlabel}{\xymatrix@R=.25em{ *{\circ}\ar@<-1ex>@{}[d]^{\alpha_{1}} \ar@{-}[r]|*\dir{ } & *{\circ}\ar@<-1ex>@{}[d]^{\alpha_{2}}\ar@{-}[r]&{\cdots}&*{\circ}\ar@{-}[l]|*\dir{ }\ar@{-}[r]|*\dir{ }\ar@<-2ex>@{}[d]^{\alpha_{n-1}}&*{\circ}\ar@<-1ex>@{}[d]^{\alpha_{n}}\\&&&& }}
\newcommand{\Bn}{\xymatrix{ *{\circ} \ar@{-}[r]|*\dir{ } & *{\circ}\ar@{-}[r]&{\cdots}&*{\circ}\ar@{-}[l]|*\dir{ }\ar@{=}[r]|*\dir{>}&*{\circ} }}
\newcommand{\Bnlabel}{\xymatrix@R=.25em{ *{\circ}\ar@<-1ex>@{}[d]^{\alpha_{1}} \ar@{-}[r]|*\dir{ } & *{\circ}\ar@<-1ex>@{}[d]^{\alpha_{2}}\ar@{-}[r]&{\cdots}&*{\circ}\ar@{-}[l]|*\dir{ }\ar@{=}[r]|*\dir{>}\ar@<-2ex>@{}[d]^{\alpha_{n-1}}&*{\circ}\ar@<-1ex>@{}[d]^{\alpha_{n}}\\&&&& }}
\newcommand{\Cn}{\xymatrix{ *{\circ} \ar@{-}[r]|*\dir{ } & *{\circ}\ar@{-}[r]&{\cdots}&*{\circ}\ar@{-}[l]|*\dir{ }\ar@{=}[r]|*\dir{<}&*{\circ} }}
\newcommand{\Cnlabel}{\xymatrix@R=.25em{ *{\circ}\ar@<-1ex>@{}[d]^{\alpha_{1}} \ar@{-}[r]|*\dir{ } & *{\circ}\ar@<-1ex>@{}[d]^{\alpha_{2}}\ar@{-}[r]&{\cdots}&*{\circ}\ar@{-}[l]|*\dir{ }\ar@{=}[r]|*\dir{<}\ar@<-2ex>@{}[d]^{\alpha_{n-1}}&*{\circ}\ar@<-1ex>@{}[d]^{\alpha_{n}}\\&&&& }}
\newcommand{\Dn}{\xymatrix@R=.25em{&&&&*{\circ} \\ *{\circ} \ar@{-}[r]|*\dir{ } & *{\circ}\ar@{-}[r]&{\cdots}&*{\circ}\ar@{-}[l]|*\dir{ }\ar@{-}[ur]|*\dir{ }\ar@{-}[dr]|*\dir{ }& \\ &&&&*{\circ} }}
\newcommand{\Dnlabel}{\xymatrix@R=.25em{&&&&*{\circ}\ar@<-1ex>@{}[d]^{\alpha_{n-1}} \\ *{\circ}\ar@<-1ex>@{}[d]^{\alpha_{1}} \ar@{-}[r]|*\dir{ } & *{\circ}\ar@<-1ex>@{}[d]^{\alpha_{2}}\ar@{-}[r]&{\cdots}&*{\circ}\ar@{-}[l]|*\dir{ }\ar@{-}[ur]|*\dir{ }\ar@{-}[dr]|*\dir{ }\ar@<-4ex>@{}[d]^{\alpha_{n-2}}& \\ &&&&*{\circ}\ar@<-2ex>@{}[u]^(-.5){\alpha_{n}} }}

\theoremstyle{remark}

\newcommand{\smtrx}[1]{\left (\begin{smallmatrix}#1\end{smallmatrix}\right)}
\author{Brian Collier}
\address{Brian Collier, University of Maryland, 4176 Campus Drive - William E. Kirwan Hall, College Park, MD 20742-4015}
\curraddr{}
\email{briancollier01@gmail.com}

\title{$\sSO(n,n+1)$-surface group representations and Higgs bundles}

\begin{document}

\setlength{\smallskipamount}{6pt}
\setlength{\medskipamount}{10pt}
\setlength{\bigskipamount}{16pt}
\begin{abstract}
We study the character variety of representations of the fundamental group of a closed surface of genus $g\geq 2$ into the Lie group $\sSO(n,n+1)$ using Higgs bundles. For each integer $0<d\leq n(2g-2),$ we show there is a smooth connected component of the character variety which is diffeomorphic to the product of a certain vector bundle over a symmetric product of a Riemann surface with the vector space of holomorphic differentials of degree $2,4,\cdots,2n-2.$
 In particular, when $d=n(2g-2)$, this recovers Hitchin's parameterization of the Hitchin component. 
We also exhibit $2^{2g+1}-1$ additional connected components of the $\sSO(n,n+1)$-character variety and compute their topology. 
Moreover, representations in all of these new components cannot be continuously deformed to representations with compact Zariski closure. Using recent work of Guichard-Wienhard on positivity, it is shown that each of the representations which define singularities (i.e. those which are not irreducible) in these $2^{2g+1}-1$ connected components are positive Anosov representations. 
\end{abstract}
\maketitle


\section{Introduction}

Since Higgs bundles were introduced, they have found application in parameterizing connected components of the moduli space of reductive surface group representations into a reductive Lie group $\sG$. 
In particular, for a closed surface $S$ with genus $g\geq2$, Hitchin gave an explicit parameterization of all but one of the connected components of the space of conjugacy classes of reductive representations of the fundamental group of $S$ into the Lie group $\sP\sSL(2,\R)$  \cite{selfduality}. 
Namely, he showed that each component with nonzero Euler class is diffeomorphic to the total space of a smooth vector bundle over an appropriate symmetric product of the surface. When the Euler class is maximal, this recovers a parameterization of the Teichm\"uller space of $S$ as a vector space of complex dimension $3g-3.$

Hitchin later showed that for $\sG$ a connected split real form, such as $\sP\sSL(n, \R)$ or $\sSO (n, n + 1)$, there is a connected component of this moduli space of representations which directly generalizes Teichm\"uller space \cite{liegroupsteichmuller}. 
Moreover, Hitchin parameterized this connected component, now called the Hitchin component, by a vector space of holomorphic differentials on the surface $S$ equipped with a Riemann surface structure. 
In this paper, we use Higgs bundle techniques to generalize both of these results for the group $\sSO(n,n+1)$. 

Let $\Gamma=\pi_1(S)$ be the fundamental group of a closed surface $S$ of genus $g\geq2$. For a real reductive algebraic Lie group $\sG,$ we will refer to the space of conjugacy classes of representations $\rho:\Gamma\ra\sG$ of $\Gamma$ into $\sG$ whose images have reductive Zariski closure as the {\em $\sG$-character variety}; it will be denoted by $\Xx(\sG).$ 
For connected reductive Lie groups, topological $\sG$ bundles on $S$ are classified by a characteristic class $\omega\in H^2(S,\pi_1(\sG))\cong\pi_1(\sG).$ 
Thus, the $\sG$-character variety decomposes as 
\[\Xx(\sG)=\bigsqcup\limits_{\omega\in\pi_1(\sG)}\Xx^\omega(\sG)~,\]
where the equivalence class of a reductive representation $\rho:\Gamma\ra\sG$ lies in $\Xx^\omega(\sG)$ if and only if the flat $\sG$ bundle determined by $\rho$ has topological type determined by $\omega\in\pi_1(\sG).$

The space $\Xx^\omega(\sG)$ is nonempty and connected for each $\omega\in \pi_1(\sG)$ when $\sG$ is compact and semisimple \cite{ramanathan_1975} and also when $\sG$ is complex and semisimple \cite{JunLiConnectedness}.
 Since $\sG$ is homotopic to its maximal compact subgroup, $\Xx^\omega(\sG)$ is connected if every representation in $\Xx^\omega(\sG)$ can be continuously deformed to one with compact Zariski closure. Connectedness of $\Xx^\omega(\sG)$ has been proven for many real forms using this technique, see \cite{AndrePGLnR,SO2n*connected}.  

There are exactly two known families of Lie groups for which the space $\Xx^\omega(\sG)$ is not connected. 
When $\sG$ is a split real form, the Hitchin component is not distinguished by an invariant $\omega\in\pi_1(\sG)$. 
Similarly, when $\sG$ is a group of Hermitian type, the connected components of {\em maximal representations} are usually not labeled by topological invariants $\omega\in\pi_1(\sG)$.
Both Hitchin representations and maximal representations define an important class of representations: they are the only known components of $\Xx(\sG)$ which consist entirely of Anosov representations \cite{AnosovFlowsLabourie,MaxRepsAnosov}. 

\subsection{New components for $\sG=\sSO(n,n+1)$}
The group $\sSO(n,n+1)$ has two connected components, we will denote the connected component of the identity by $\sSO_0(n,n+1).$ For $n\geq 3$, the group $\sSO(n,n+1)$ is a split group, but not of Hermitian type. Nevertheless, we show that the $\sSO_0(n,n+1)$-character variety has many non-Hitchin connected components which are not distinguished by a topological invariant $\omega\in\pi_1(\sSO_0(n,n+1))=\Z_2\oplus\Z_2.$ 
\begin{TheoremIntro}[\ref{THM1}]
	Let $\Gamma$ be the fundamental group of a closed surface $S$ of genus $g\geq2$ and let $\Xx(\sSO(n,n+1))$ be the $\sSO(n,n+1)$-character variety of $\Gamma$. 
	For each integer $d\in(0,n(2g-2)],$ there is a smooth connected component $\Xx_d(\sSO(n,n+1))$ of $\Xx(\sSO(n,n+1))$ which does not contain representations with compact Zariski closure. Furthermore, for each choice of Riemann surface structure $X$ on $S,$ the space $\Xx_d(\sSO(n,n+1))$ is diffeomorphic to the product 
	\[\Xx_d(\sSO(n,n+1))\cong\Ff_d\times \bigoplus\limits_{j=1}^{n-1} H^0(K^{2j})~,\] where $\Ff_d$ is the total space of a rank $d+(2n-1)(g-1)$ vector bundle over the symmetric product $\Sym^{n(2g-2)-d}(X)$ and $H^0(K^{2j})$ is the vector space of holomorphic differentials of degree $2j.$ 
\end{TheoremIntro}
In fact, the representations $\rho\in\Xx_d(\sSO(n,n+1))$ factor through the connected component of the identity $\sSO_0(n,n+1)\subset\sSO(n,n+1)$.
\begin{Remark}
As a direct corollary, the connected components $\Xx_d(\sSO(n,n+1))$ deformation retract onto the symmetric product $\Sym^{n(2g-2)-d}(X).$ In particular, the cohomology ring of $\Xx_d(\sSO(n,n+1))$ is the same as the cohomology ring of the symmetric product $\Sym^{n(2g-2)-d}(X)$ which was computed in \cite{SymmetricProductsofAlgebraicCurves}.

Using the isomorphism $\sP\sSL(2,\R)\cong\sSO_0(1,2),$ Theorem \ref{THM1} recovers Hitchin's parameterization of the nonzero Euler class components of $\Xx(\sP\sSL(2,\R))$ mentioned above. Also, when the label $d$ in Theorem \ref{THM1} is maximal, the vector bundle $\Ff_{n(2g-2)}$ is the  rank $(4n-1)(g-1)$ vector space of holomorphic differentials of degree $2n.$ 
Thus, we recover the parameterization of the $\sSO(n,n+1)$-Hitchin component as a vector space of holomorphic differentials. 
When $n=2$, Theorem \ref{THM1} gives a parameterization of an $\sSO_0(2,3)=\sP\sSp(4,\R)$-version of $\sSp(4,\R)$ components discovered in \cite{sp4GothenConnComp}. 
For $n>2$ and $0<d<n(2g-2)$ the components are new.
\end{Remark}
There is also a connected component associated to $d=0$ which has non-orbifold singularities. We briefly describe it here. Let $X$ be a Riemann surface structure on $S$ and let $\Pic(X)$ be the Picard group of holomorphic line bundles on $X.$ 
Consider the space $\widetilde\Ff_0$ defined by 
 	\[\xymatrix@=.5em{\widetilde\Ff_0=\{(M,\mu,\nu)\ |\ M\in\Pic^0(X),\ \mu\in H^0(M^{-1}K^n),\ \nu\in H^0(MK^n)\ \}~.}\]
Recall that the group of matrices $\smtrx{\lambda&0\\0&\lambda^{-1}}$ and $\smtrx{0&\lambda\\\lambda^{-1}&0}$ for $\lambda\in\C^*$ is isomorphic to $\sO(2,\C).$ There is a natural action of $\sO(2,\C)$ on $\widetilde\Ff_0$ given by:
\[	g\cdot(M,\mu,\nu)=\begin{dcases}(M,\lambda^{-1}\mu,\lambda\nu)&\text{if}\ g=\smtrx{\lambda&0\\0&\lambda^{-1}}\\
	(M^{-1},\lambda^{-1}\nu,\lambda\mu)&\text{if}\ g=\smtrx{0&\lambda\\\lambda^{-1}&0}
	\end{dcases}~.\]

\begin{TheoremIntro}[\ref{THM2}]
Let $\Gamma$ be the fundamental group of a closed surface $S$ of genus $g\geq2$ and let $\Xx(\sSO(n,n+1))$ be the $\sSO(n,n+1)$-character variety of $\Gamma$. For each $n\geq2,$ there is a connected component $\Xx_0(\sSO(n,n+1))$ of $\Xx(\sSO(n,n+1))$ which does not contain representations with compact Zariski closure. Furthermore, for each Riemann surface structure on $S,$ the space $\Xx_0(\sSO(n,n+1))$ is homeomorphic to
\[\Xx_0(\sSO(n,n+1))\cong\Ff_0\times\bigoplus\limits_{j=1}^{n-1}H^0(K^{2j})~,\]
where $\Ff_0$ is the GIT quotient $\widetilde\Ff_0\big\slash\big\slash\sO(2,\C)$ of the $\sO(2,\C)$-space $\widetilde\Ff_0$ described above and $H^0(K^{2j})$ is the vector space of holomorphic differentials of degree $2j$. 
\end{TheoremIntro}
In fact, the representations $\rho\in\Xx_0(\sSO(n,n+1))$ factor through the connected component of the identity $\sSO_0(n,n+1)\subset\sSO(n,n+1)$.
\begin{Remark}
	In Section \ref{singular components Section}, we provide a parameterization of the singular space $\Xx_0(\sSO(n,n+1)).$ This is a direct analogy of the associated component of the set of maximal $\sSO_0(2,3)$-representations provided in \cite{SO23LabourieConj}. Unlike the maximal $\sSO_0(2,3)$ case, the component $\Xx_0(\sSO(n,n+1)$ does not arise from a known topological invariant for $n\geq3$. Thus, to show that $\Xx_0(\sSO(n,n+1))$ does indeed define a connected component, it is necessary to analyze the local structure around the singularities (see Lemma \ref{Lemma H2=0}).
Also analogous to the computations in \cite{SO23LabourieConj}, $\Xx_0(\sSO(n,n+1)$ deformation retracts onto the quotient of $\Pic^0(X)$ by the $\Z_2$ action of inversion. In particular, the rational cohomology of $\Xx_0(\sSO(n,n+1))$ is given by 
\[H^j(\Xx_0(\sSO(n,n+1)),\Q)=\begin{dcases}
	H^{j}((S^1)^{2g},\Q)& \text{if\ $j$\ is\ even~}\\
	0&\text{otherwise}
\end{dcases}~.\]
\end{Remark}

Each nonzero cohomology class $sw_1\in H^1(S,\Z_2)$ corresponds to a connected principal $\Z_2$ bundle (i.e., orientation double cover) $X_{sw_1}\to X.$ 
If $\iota$ is the covering involution of the covering, then the Prym variety $\Prym(X_{sw_1},X)$ is defined as the kernel of $Id+\iota^*:\Pic^0(X_{sw_1})\to\Pic^0(X_{sw_1})$. That is, 
\[\Prym(X_{sw_1},X)=\{M\in \Pic^0(X_{sw_1})\ |\ \iota^*M=M^{-1}\}~.\]
 The Prym variety of an orientation double cover of a Riemann surface has two connected components determined by an invariant $sw_2\in H^2(X,\Z_2).$ Let $K_{X_{sw_1}}$ denote the canonical bundle of the double cover $X_{sw_1}.$

\begin{TheoremIntro}[\ref{THM3}]
Let $\Gamma$ be the fundamental group of a closed surface $S$ of genus $g\geq2$ and let $\Xx(\sSO(n,n+1))$ be the $\sSO(n,n+1)$-character variety of $\Gamma$. For each $n\geq2$ and each $(sw_1,sw_2)\in (H^1(X,\Z_2)\setminus\{0\})\times H^2(X,\Z_2)$, there is a connected component $\Xx^{sw_2}_{sw_1}(\sSO(n,n+1))$ of $\Xx(\sSO(n,n+1))$ which does not contain representations with compact Zariski closure. Furthermore, for each Riemann surface structure $X$ on $S,$ the space $\Xx_{sw_1}^{sw_2}(\sSO(n,n+1))$ is a smooth orbifold diffeomorphic to
		\[\Ff^{sw_2}_{sw_1}/(\Z_2\oplus\Z_2)\times \bigoplus\limits_{j=1}^{n-1}H^0(K^{2j}_X)\] where $\Ff^{sw_2}_{sw_1}\to \Prym^{sw_2}(X_{sw_1},X)$ is the rank $(4n-2)(2g-2)$ vector bundle over the connected component of the Prym variety associated to $sw_2$ with $\pi^{-1}(M)=H^0(MK_{X_{sw_1}}^n)$. Here the $\Z_2\oplus\Z_2$ action is generated by $(M,\mu)\mapsto(M,-\mu)$ and $(M,\mu)\mapsto(\iota^*M,\iota^*\mu),$ where $\iota$ is the covering involution of $X_{sw_1}$. 
\end{TheoremIntro}
When $n$ is even, the representations in the components $\Xx_{sw_1}^{sw_2}(\sSO(n,n+1))$ factor through the connected component of the identity $\sSO_0(n,n+1)$, however when $n$ is odd, they do not. 

\begin{Remark}
 Since the space $\Ff_{sw_1}^{sw_2}$ deformation retracts onto $\Prym^{sw_2}(X_{sw_1})$, the homotopy type of each component $\Xx_{sw_1}^{sw_2}(\sSO(n,n+1))$ is the same as the quotient of $(S^1)^{2g-2}$ be the $\Z_2$ action of inversion. In particular, its cohomology is given by 
 \[H^j(\Xx_0(\sSO(n,n+1)),\Q)=\begin{dcases}
	H^{j}((S^1)^{2g-2},\Q)& \text{if\ $j$\ is\ even~}\\
	0&\text{otherwise}
\end{dcases}~.\]
	Again, for $n=2,$ the connected components $\Xx_{sw_1}^{sw_2}(\sSO(2,3))$ of Theorem \ref{THM3} were described by Alessandrini and the author in \cite{SO23LabourieConj}. 
\end{Remark}

\begin{Remark}
	The components of Theorems \ref{THM1} and \ref{THM2} are labeled by an integer invariant $d\in[0,n(2g-2)]$ and the components of Theorem \ref{THM3} are labeled by $\Z_2$-invariants $(sw_1,sw_2)\in H^1(S,\Z_2)\setminus\{0\}\times H^2(S,\Z_2)$. 
	Similar types of invariants have recently been associated to the spectral data of certain $\sSO(n,n+1)$ Higgs bundles by Schaposnik and Baraglia-Schaposnik in \cite{LauraSO(pp+1)} and \cite{DavidLauraCayleyLanglands}. It would be very interesting to relate the spectral data invariants to the component description above. 
\end{Remark}
\subsection{Generalizations of low dimensional isomorphisms}
The group $\sSO_0(2,3)$ is isomorphic to $\sP\sSp(4,\R)$, and the connected components of Theorem \ref{THM1} are a $\sP\sSp(4,\R)$ version of the $\sSp(4,\R)$ components discovered by Gothen in  \cite{sp4GothenConnComp}. 
The groups $\sSp(2n,\R)$ and $\sSO_0(2,n)$ provide two families of Hermitian groups which generalize $\sSO_0(2,3)$. However,
the space of maximal $\sSp(2n,\R)$ representations behaves differently for $n=2$ and $n\geq 3,$ and the space maximal $\sSO_0(2,n)$ representations behaves differently for $n=3$ and $n\geq 4.$

For $\sSp(4,\R), $ there are $3\cdot 2^{2g}+2g-4$ connected components of the space of maximal $\sSp(4,\R)$-representations \cite{sp4GothenConnComp}, while the space of maximal $\sSp(2n,\R)$ representations has $3\cdot 2^{2g}$ connected components for $n\geq 3$ \cite{HiggsbundlesSP2nR}. 
In particular, for maximal $\sSp(4,\R)$ representations, $2g-4$ of the connected components consist entirely of Zariski dense representations \cite{MaximalSP4}. Similarly, for maximal $\sSO_0(2,3)$ representations, the $4g-5$ of the connected components from Theorem \ref{THM1} with $d\in(0,4g-4)$ consist entirely of Zariski dense representations \cite{SO23LabourieConj}. 
The remaining connected components of maximal $\sSp(4,\R)$ and $\sSO_0(2,3)$ representations contain representations which factor through a Fuchsian representation $\rho_{Fuch}:\Gamma\ra\sSL(2,\R)$ \cite{MaximalSP4,TopInvariantsAnosov}.

For $n\geq3,$ each connected component of maximal $\sSp(2n,\R)$ representations contains representations which factor through a Fuchsian representation \cite{TopInvariantsAnosov}. 
Similarly, there are $2^{2g+1}$ connected components of maximal $\sSO_0(2,n)$ representations for $n\geq4$, and each component contains representations which factor through a Fuchsian representation. 

Theorem \ref{THM1} gives an explanation of this difference as a consequence of the low dimensional isomorphism $\sSO_0(2,3)\cong \sP\sSp(4,\R)$. Namely, the extra maximal components appearing for $\sSp(4,\R)$ and $\sSO_0(2,3)$ are an $\sSO(n,n+1)$ phenomenon.  

For $n\geq 3$, the group $\sSO(n,n+1)$ is a split group but not of Hermitian type. 
As a result, Theorems \ref{THM1}, \ref{THM2} and \ref{THM3} provide the first examples of connected components of $\Xx(\pi_1,\sG)$ which are not maximal, not Hitchin, and are not distinguished by a topological invariant in $\pi_1(\sG).$ 

The component count of $\Xx(\sSO_0(n,n+1))$ was established by Goldman \cite{TopologicalComponents} for $n=1$ and by combining the work Bradlow-Garcia-Prada-Gothen \cite{HermitianTypeHiggsBGG} and Gothen-Oliveira \cite{AndreQuadraticPairs} for $n=2$. 
For $n\geq3,$ Theorems \ref{THM1}, \ref{THM2} and \ref{THM3} provide a new lower bound for the number of components of the space $\Xx(\sSO(n,n+1)).$ Namely,
\begin{equation}
	\label{lowerboundpi_0}
	\big|\pi_0\big(\Xx(\sSO(n,n+1))\big)\big|\geq 2^{2g+2}+1+ n(2g-2)+2(2^{2g}-1)~.
 \end{equation}
Here, the first $2^{2g+2}$ components contain representations with Zariski closure in $\sS(\sO(n)\times\sO(n+1))$ and the remaining components come from Theorems \ref{THM1}, \ref{THM2} and \ref{THM3}. In \cite{SOpqStabilityAndMinima}, the connected components of $\Xx(\sSO(n,m))$ are counted, and for $m=n+1$ it is shown that the lower bound in \eqref{lowerboundpi_0} is indeed an equality. 
\subsection{Positive Anosov representations}
We now turn to the geometry of the representations in the components described by Theorem \ref{THM2} and Theorem \ref{THM3}. 
Anosov representations were introduced by Labourie \cite{AnosovFlowsLabourie} and have many interesting geometric and dynamic properties which generalize convex cocompact representations into rank one Lie groups. 
Important examples of Anosov representations include quasi-Fuchsian representations, Hitchin representations into split groups and maximal representations into groups of Hermitian type.

Recently, Guichard and Wienhard \cite{PosRepsGWPROCEEDINGS} introduced the notion of a $\sP_\Theta$-positive Anosov representation which refines the notion of an Anosov representation. 
In particular, the spaces of Hitchin representations are positive with respect to the Borel subgroup \cite{fock_goncharov_2006,AnosovFlowsLabourie} and, for a Hermitian group $\sG$ of tube type, maximal representations are positive with respect to the parabolic subgroup which gives rise to the Shilov boundary of the Riemannian symmetric space of $\sG$ \cite{MaxRepsAnosov}.

Since the Lie group $\sSO(n,n+1)$ is split, it admits a notion of positivity with respect to the Borel subgroup.  Interestingly, $\sSO(n,n+1)$ also admits a notion of positivity with respect to the generalized flag variety $\sSO(n,n+1)/P_\Theta$ consisting of flags $V_1\subset\cdots\subset V_{n-1}\subset\R^{2n+1}$ where $V_j$ is an isotropic (with respect to a signature $(n,n+1)$ inner product) $j$-plane. We will call this $\sP_\Theta$-positivity.

\begin{Remark}\label{Remark positive closed in irr}
The set of positive Anosov representations is open in the character variety. In \cite{PosRepsGLW}, it is conjectured that positive Anosov representations are also closed in the character variety. 
In fact, it can be shown that the set of positive Anosov representations is closed in the set of irreducible representations \cite{AnnaPrivateCommunication}. Namely, let $\rho_j:\Gamma\to \sSO(n,n+1)$ be a sequence of $\sP_\Theta$-positive Anosov representation which converge to $\rho_\infty:\Gamma\to\sSO(n,n+1).$ If the action of each $\rho_j$ on $\R^{2n+1}$ via the standard representations of $\sSO(n,n+1)$ is irreducible and $\rho_{\infty}$ is also irreducible, then $\rho_{\infty}$ is $\sP_\Theta$ Anosov.
\end{Remark}

Here we prove that the set on {\em non-irreducible representations} in the connected components $\Xx_0(\sSO(n,n+1))$ and $\Xx_{sw_1}^{sw_2}(\sSO(n,n+1))$ are $\sP_\Theta$ positive Anosov. 
\begin{TheoremIntro}[\ref{THM4}]
	Let $\sSO(n,n+1)/P_\Theta$ be the generalized flag variety of flags 
	\[V_1\subset\cdots\subset V_{n-1}\subset V_{n-1}^\perp\subset\cdots\subset V_1^\perp\subset\R^{2n+1}~,\] where $V_j\subset\R^{2n+1}$ is an isotropic $j$-plane. If $n\geq2,$ then the set of representations $\rho$ in $\Xx_0(\sSO(n,n+1))$ or $\Xx_{sw_1}^{sw_2}(\sSO(n,n+1))$ for which the action of $\rho$ on $\R^{2n+1}$ is reducible is a nonempty set consisting entirely of $\sP_\Theta$-positive Anosov representation. 
\end{TheoremIntro}
\begin{Remark}
Assuming the results mentioned in Remark \ref{Remark positive closed in irr}, Theorem \ref{THM4} can be significantly strengthened to the statement that the components $\Xx_0(\sSO(n,n+1))$ and $\Xx_{sw_1}^{sw_2}(\sSO(n,n+1))$ consist {\em entirely of Anosov representations}. The argument is as follows: Let $\rho_0$ be a reducible representation in $\Xx_0(\sSO(n,n+1))$ or $\Xx_{sw_1}^{sw_2}(\sSO(n,n+1)).$ 
Since positive representations define an open set in the character variety, there is an open neighborhood $U_{\rho_0}$ of $\rho_0$ consisting of $\sP_\Theta$-positive representations. 
In particular, there exists $\rho\in U_{\rho_0}$ which is irreducible. Since positivity is closed in the set of irreducible representations, all irreducible representations $\rho\in\Xx_0(\sSO(n,n+1))$ are $\sP_\Theta$-positive. Thus, by Theorem \ref{THM4} all representations in $\Xx_0(\sSO(n,n+1))$ and $\Xx_{sw_1}^{sw_2}(\sSO(n,n+1))$ are $\sP_\Theta$ Anosov.
 \end{Remark}

For $n=2$, Theorem \ref{THM4} follows from maximality of the corresponding representations.
For $n\geq3,$ the proof relies heavily on the work of Guichard-Wienhard and Guichard-Labourie-Wienhard on positive representations \cite{PosRepsGWPROCEEDINGS,PosRepsGLW} and establishing that the representations which correspond to the singularities of $\Xx_{sw_1}^{sw_2}(\sSO(n,n+1))$ and $\Xx_0(\sSO(n,n+1))$ are products of Hitchin representations in $\sSO(n-1,n)$ with an $\sSO(2)$ representation or $\sSO(n,n)$-Hitchin representations. 

For $0<d<n(2g-2),$ the spaces $\Xx_d(\sSO(n,n+1))$ from Theorem \ref{THM1} are smooth; hence all the representations in these components are irreducible. Thus, if there exists a representation $\rho\in\Xx_d(\sSO(n,n+1))$ which is positive Anosov, then, by Remark \ref{Remark positive closed in irr}, $\Xx_d(\sSO(n,n+1))$ would consist entirely of positive Anosov representations. 
There are however no obvious model representations to consider in the components $\Xx_d(\sSO(n,n+1))$. In particular, for $n=2,$ all representations in these components are Zariski dense. We conjecture this holds for the components $\Xx_d(\sSO(n,n+1))$ for $0<d<n(2g-2).$
\begin{Conjecture}
	For $0<d<n(2g-2)$, all representations in the component $\Xx_d(\sSO(n,n+1))$ from Theorem \ref{THM1} are Zariski dense. 
\end{Conjecture}

\noindent\textbf{Organization of Paper:} In Sections \ref{section Background} and \ref{Section sonn+1 Higgs}, we recall the necessary features of Higgs bundles and character varieties. In Section \ref{section: smooth components}, we prove Theorem \ref{THM1} and in Section, \ref{singular components Section} we prove Theorems \ref{THM2} and \ref{THM3}. In Section \ref{Section Zariski closure}, we prove results about the Zariski closures of representations in the new components of $\Xx(\sSO(n,n+1))$. Finally, in Section \ref{section: Positive}, the notion of positive Anosov representations is recalled and we use the results on Zariski closures to prove Theorem \ref{THM4}.

\smallskip
\noindent\textbf{Acknowledgments:} This project benefited from many fruitful discussions on $\sSO(n,m)$ Higgs bundles with Daniele Alessandrini, Steve Bradlow, Oscar Garcia-Prada, Peter Gothen and Andr\'e Oliveira.  I am very grateful to Anna Wienhard for her explanations of positive structures and Jean-Philippe Burelle for many helpful conversations on positivity. I would also like to thank Jeff Adams for his help understanding representations of low dimensional groups and Olivier Guichard for his many useful comments on an early draft. 

The author is funded by a National Science Foundation Mathematical Sciences Postdoctoral Fellowship, NSF MSPRF no. 1604263 and also acknowledge the support from U.S. National Science Foundation grants DMS 1107452, 1107263, 1107367 “RNMS: GEometric structures And Representation varieties” (the GEAR Network).

\section{Higgs bundles and Surface group representations}\label{section Background}
We start by recalling the necessary facts about Higgs bundles and surface group representations which, in subsequent sections, will be used for $\sG=\sSO(n,n+1)$. 

\subsection{$\sG$ Higgs bundles}
Let $\sG$ be a {\em real} algebraic semisimple Lie group\footnote{In fact, with slight modifications, everything works for real reductive Lie groups. We will not need this more general setting.} with Lie algebra $\fg,$ and fix $\sH\subset\sG$ a maximal compact subgroup with Lie algebra $\fh$. 
Let $\fg=\fh\oplus\fm$ be the corresponding Cartan decomposition of the Lie algebra $\fg.$ 
Here $\fm$ is the orthogonal complement of $\fh$ with respect to the Killing form of $\fg,$ and 
the splitting $\fh\oplus\fm$ consists of the $\pm1$-eigenspaces of an involution $\theta:\fg\to\fg.$ 
Thus, $[\fm,\fm]\subset\fh$ and $[\fh,\fm]\subset\fm,$ and the splitting $\fg=\fh\oplus\fm$ is invariant with respect to the adjoint action of $\sH$ on $\fg.$ 
Complexifying everything, we have an $Ad_{\sH_\C}$ invariant decomposition
\[\fg_\C=\fh_\C\oplus\fm_\C~.\]

Let $X$ be a closed Riemann surface of genus $g\geq2$ and canonical bundle $K.$
For any group $\sG$, if $P$ is a principal $\sG$ bundle and $\alpha:\sG\ra\sGL(V)$ is a linear representation, denote the associated vector bundle $P\times_{\sG} V$ by $P[V].$

\begin{Definition}\label{Def GHiggsBundle}
	Fix a smooth principal $\sH_\C$ bundle $P\ra X$. A {\em $\sG$ Higgs bundle structure on $P$} is a pair $(\Pp,\varphi)$ where $\Pp$ is a holomorphic principal $\sH_\C$ bundle with underlying smooth bundle $P$ and $\varphi\in H^0(X,\Pp[\fm_\C]\otimes K)$ is a holomorphic section of the associated $\fm_\C$ bundle twisted by $K.$ The section $\varphi$ is called the {\em Higgs field}.
\end{Definition}
\begin{Example}
If $\sG$ is compact, then $\sH_\C=\sG_\C$ and $\fm_\C=\{0\}.$ Thus for compact groups, a Higgs bundle is the same as a holomorphic principal $\sG_\C$ bundle. 
When $\sG$ is a complex semisimple Lie group, we have $\sH_\C=\sG$ and $\fm_\C\cong\fg$. In this case, a $\sG$ Higgs bundle consists of a holomorphic $\sG$ bundle together with a holomorphic $K$-twisted section of the adjoint bundle.
\end{Example}

If $\alpha:\sH_\C\ra\sGL(V)$ is a linear representation of $\sH_\C$, the data of a $\sG$ Higgs bundle can be described by the vector bundle associated to $\alpha$ and a section of another associated bundle. For instance, if $\alpha:\sGL(n,\C)\to\sGL(\C^n)$ is the standard representation, then a $\sGL(n,\C)$ Higgs bundle is equivalent to a rank $n$ holomorphic vector bundle $\Ee\to X$ and a holomorphic section $\Phi$ of $\End(\Ee)\otimes K$. 
Similarly, using the standard representation, an $\sSL(n,\C)$ Higgs bundle is equivalent to a $\sGL(n,\C)$ Higgs bundle $(\Ee,\Phi)$ with $\Lambda^n\Ee=\Oo$ and $\mathrm{Tr}(\Phi)=0.$

To form the moduli space of Higgs bundles, we need the notion of stability. 


 \begin{Definition}\label{DEF stability of SLnC}
A $\sGL(n,\C)$ Higgs bundle $(\Ee,\Phi)$ is called {\em stable} if for all $\Phi$-invariant subbundles $\Ff\subset\Ee$ we have $\frac{deg(\Ff)}{rk(\Ff)}<\frac{deg(\Ee)}{rk(\Ee)}~.$
An $\sSL(n,\C)$ Higgs bundle $(\Ee,\Phi)$ is 
\begin{itemize}
    \item {\em stable} if all $\Phi$-invariant subbundles $\Ff\subset\Ee$ satisfy $deg(\Ff)<0$,
    \item {\em polystable} if $(\Ee,\Phi)=\bigoplus(\Ee_j,\Phi_j)$ where each $(\Ee_j,\Phi_j)$ is a stable $\sGL(n_j,\C)$ Higgs bundle with $deg(\Ee_j)=0$ for all $j.$
\end{itemize}
 \end{Definition} 

There are appropriate notions of stability and polystability for $\sG$ Higgs bundles. With respect to these notions, the moduli space of $\sG$ Higgs bundles is defined as a polystable quotient. 
Rather than recalling the definition of polystability for $\sG$ Higgs bundles, we will use the following result (see \cite{HiggsPairsSTABILITY}). 

\begin{Proposition}\label{Prop GHiggs polystable}
	Let $\sG$ be a real form of an irreducible subgroup of $\sSL(n,\C)$. A $\sG$ Higgs bundle $(\Pp,\varphi)$ is polystable if and only if the associated $\sSL(n,\C)$ Higgs bundle is polystable. 
\end{Proposition} 
 
The gauge group $\Gg_{\sH_\C}$ of smooth bundle automorphisms of a smooth $\sH_\C$ bundle $P_{\sH_\C}$ acts on the set of Higgs bundle structures $(\Pp,\varphi)=(\bar\p_P,\varphi)$ by the adjoint action. 

\begin{Definition} \label{Def ModuliofHiggsBundle}
Fix a smooth principal $\sH_\C$ bundle $P_{\sH_\C}$ on $X.$ The moduli space of $\sG$ Higgs bundle structures on $P_{\sH_\C}$ consists of isomorphism classes of polystable Higgs bundles with underlying smooth bundle $P_{\sH_\C},$ 
\[\Mm(P_{\sH_\C},\sG) = \{\text{polystable } \sG\text{-Higgs bundle structures on } P_{\sH_\C} \}/\Gg_{\sH_\C}~.\]
The union over the set of isomorphism classes of smooth principal $\sH_\C$ bundles on $X$ of the spaces $\Mm(P_{\sH_\C},\sG)$ will be referred to as the moduli space of $\sG$ Higgs bundles and denoted by $\Mm(\sG)$. 
\end{Definition}
The space $\Mm(\sG)$ can in fact be given the structure of an complex analytic variety of expected dimension $dim(\sG)(g-1)$ \cite{selfduality,localsystems,schmitt_2005}.  
Since $\sH_\C$ and $\sG$ are both homotopy equivalent to $\sH,$ the set of equivalence classes of topological $\sH_\C$ bundles on $X$ is the same as the set of equivalence classes of topological $\sG$ bundles on $X$. Denote this set by $\Bun_X(\sG)$. If the group $\sG$ is connected, then 
\[\Bun_X(\sG)\cong H^2(X,\pi_1(\sG))\cong \pi_1(\sG)~.\]
If $\sG$ is not connected, the description is slightly more complicated, see \cite[Section 3.1]{AndrePGLnR}.
This gives a decomposition of the Higgs bundle moduli space:
\[\Mm(\sG)=\bigsqcup\limits_{a\in\Bun_X(\sG)}\Mm^a(\sG)~,\]
where $a\in\Bun_X(\sG)$ is the topological type of the underlying $\sH_\C$ bundle of the Higgs bundles in $\Mm^a(\sG).$ 

The automorphism group $\Aut(\bar\p_\Pp,\varphi)$ of a polystable $\sG$ Higgs bundle $(\Pp,\varphi)$ is defined by 
\[\Aut(\bar\p_\Pp,\varphi)=\{g\in\Gg_{\sH_\C}|\ (Ad_g\bar\p_\Pp,Ad_g\varphi)=(\bar\p_\Pp,\varphi)\}~.\]
The center $\Zz(\sG_\C)$ of $\sG_\C$ is the intersection of the center of $\sH_\C$ and the kernel of the representation $Ad:\sH_\C\ra\sGL(\fm_\C).$ 
Thus, we always have $\Zz(\sG_\C)\subset\Aut(\bar\p_\Pp,\varphi).$ 
Using our definition of polystability from Proposition \ref{Prop GHiggs polystable}, we use the following (nonstandard) definition of stability of a $\sG$ Higgs bundle.
\begin{Definition}
	\label{DEF Stability of G Higgs bundle} Let $\sG$ be a semisimple Lie group which is a real form of an irreducible subgroup of $\sSL(n,\C)$. A polystable $\sG$ Higgs bundle $(\Pp,\varphi)$ is {\em stable} if $\Aut(\Pp,\varphi)$ is finite.
\end{Definition}

Given a polystable $\sG$ Higgs bundle $(\Pp,\varphi)$, consider the complex of sheaves \[C^\bullet(\Pp,\varphi):\xymatrix{\Pp[\fh_\C]\ar[r]^{ad_\varphi\ \ \ }&\Pp[\fm_\C]\otimes K}~.\] This gives a long exact sequence in hypercohomology:

\begin{equation}\label{Eq hypercohomology}
\xymatrix@R=1em@C=1.3em{0\ar[r]&\HH^0(C^\bullet(\Pp,\varphi))\ar[r]&H^0(\Pp[\fh_\C])\ar[r]^{ ad_\varphi\ \ \ \ }&H^0(\Pp[\fm_\C]\otimes K)\ar[r]&\HH^1(C^\bullet(\Pp,\varphi))\\\ar[r]&H^1(\Pp[\fh_\C])\ar[r]^{ad_\varphi \ \ }&H^1(\Pp[\fm_\C]\otimes K)\ar[r]&\HH^2(C^\bullet(\Pp,\varphi))\ar[r]&0~.}
\end{equation}

Note that the automorphism group $\Aut(\bar\p_\Pp,\varphi)$ acts on $\HH^1(C^\bullet(\Pp,\varphi))$.
Using standard slice methods of Kuranishi (see \cite[Chapter 7.3]{DiffGeomCompVectBun} for details for the moduli space of holomorphic bundles), a neighborhood of the isomorphism class of a polystable Higgs bundle $(\Pp,\varphi)$ in $\Mm(\sG)$ is given by 
\begin{equation}\label{EQ: local kuranishi}
	\kappa^{-1}(0)\big\slash\big\slash\Aut(\Pp,\phi)
\end{equation}
where $\kappa:\HH^1(C^\bullet(\Pp,\varphi))\to \HH^2(C^\bullet(\Pp,\varphi))$ is the so called Kuranishi map. 

When $\HH^2(C^\bullet(\Pp,\varphi))=0,$ this simplifies considerably. Namely, in this case, a neighborhood of the isomorphism class of a polystable Higgs bundle $(\Pp,\varphi)$ in $\Mm(\sG)$ is given by 
\[\HH^1(C^\bullet(\Pp,\varphi))\big\slash\big\slash \Aut(\Pp,\varphi)~.\]
\begin{Remark}
For all of the $\sSO(n,n+1)$ Higgs bundles considered in the subsequent sections we will prove that the relevant $\HH^2$ always vanishes. For this reason, we will not recall the construction of the Kuranishi map. 
\end{Remark}

When the automorphism group $\Aut(\bar\p,\varphi)$ is finite, the GIT quotient above simplifies to a regular quotient. This gives the following characterizations of smooth points and orbifold points of $\Mm(\sG).$


\begin{Proposition}\label{Prop Orbifold/Smooth GHiggs}
	Let $\sG$ be a semisimple real Lie group. If $(\Pp,\varphi)$ is a polystable $\sG$ Higgs bundle with $\HH^2(C^\bullet(\Pp,\varphi))=0$ and $\Aut(\Pp,\varphi)$ finite, then the isomorphism class of $(\Pp,\varphi)$ is an orbifold point of $\Mm(\sG)$ of type $\Aut(\Pp,\varphi)/\Zz(\sG).$ In particular, if $\Aut(\Pp,\varphi)=\Zz(\sG)$, then $(\Pp,\varphi)$ defines a smooth point of $\Mm(\sG).$ 
\end{Proposition}

Let $p_1,\cdots,p_{n-1}$ be a basis of $\sSL(n,\C)$ invariant homogeneous polynomials on $\fsl(n,\C)$ with $deg(p_j)=j+1.$ 
Given an $\sSL(n,\C)$ Higgs bundle $(E,\Phi),$ the tensor $p_j(\Phi)$ is a holomorphic differential of degree equal to the degree of $p_j.$ The map 
\[(E,\Phi)\mapsto(p_1(\Phi),\cdots,p_{n-1}(\Phi))\] from the set of Higgs bundles to the vector space $\bigoplus\limits_{j=2}^{n}H^0(K^j)$ descends to a map 
\begin{equation}\label{EQ Hitchin Fibration}
	h:\xymatrix{\Mm(\sSL(n,\C))\ar[r]&\bigoplus\limits_{j=2}^{n}H^0(K^j)}~.
\end{equation}
 The map $h$ will be referred to as the Hitchin fibration. In \cite{IntSystemFibration}, Hitchin showed that $h$ is a {\em proper} map. The properness of the Hitchin fibration will play a key role in Sections \ref{section: smooth components} and \ref{singular components Section}.


Finally, we have the notion of reducing the structure group of a $\sG$ Higgs bundle. This will be important in Sections \ref{Section Zariski closure} and \ref{section: Positive}.
\begin{Definition}\label{DEF:Higgs bundle Reduction}
	Let $\sG$ and $\sG'$ be semisimple Lie groups with maximal compact subgroups $\sH$ and $\sH'$ and Cartan decompositions $\fg=\fh\oplus\fm$ and $\fg=\fh'\oplus\fm'$. Suppose $i:\sG'\ra\sG$ is an embedding such that $i(\sH')\subset\sH$ and $di(\fm')\subset\fm.$ A $\sG$ Higgs bundle $(\Pp,\varphi)$ {\em reduces to a $\sG'$ Higgs bundle} $(\Pp',\varphi)$ if the holomorphic $\sH_\C$ bundle $\Pp$ admits a holomorphic reduction of structure group to the $\sH_\C'$ bundle $\Pp'$ and, with respect to this reduction, $\varphi\in H^0((\Pp'\times_{\sH_\C'}\fm_\C')\otimes K)\subset H^0((\Pp\times_{\sH_\C}\fm_\C)\otimes K)$. 
\end{Definition}
\begin{Remark}
	Note that a polystable $\sG$-Higgs bundle $(\Pp,\varphi)$ reduces to its maximal compact subgroup if and only if the Higgs field $\varphi$ vanishes.
\end{Remark}
\subsection{Relation to surface group representations}
Let $\Gamma$ be the fundamental group of a closed oriented surface $S$ of genus $g\geq 2$ and let $\sG$ be a real algebraic semisimple Lie group. 
\begin{Definition}
	A representation $\rho:\Gamma\ra\sG$ is {\em reductive} if the Zariski closure of the image $\rho(\Gamma)\subset\sG$ is a reductive subgroup. 
 \end{Definition}
 The conjugation action of $\sG$ on $\Hom(\Gamma,\sG)$ does not in general have a Hausdorff quotient. However, if we restrict to the set of reductive representations, the quotient will be Hausdorff.
\begin{Definition}
	The {\em $\sG$-character variety} $\Xx(\sG)$ of a surface group $\Gamma$ is the space of conjugacy classes of reductive representations of $\Gamma$ in $\sG$:
	\[\Xx(\sG)=\Hom^{red}(\Gamma,\sG)/\sG~.\]
\end{Definition}
\begin{Example}\label{EX: Fuchsian reps}
	The set of {\em Fuchsian representations} $\Fuch(\Gamma)\subset\Xx(\Gamma,\sSO(1,2))$ is defined to be the subset of conjugacy classes of {\em faithful} representations with {\em discrete image}.  
The space $\Fuch(\Gamma)$ defines one connected components of $\Xx(\Gamma,\sSO(1,2))$ \cite{TopologicalComponents} and is in one to one correspondence with the Teichm\"uller space of isotopy classes of marked Riemann surface structures on the surface $S.$ 
\end{Example}
Each representation $\rho\in\Xx(\sG)$ defines a flat $\sG$ bundle 
\[E_\rho=(\widetilde S\times\sG)/\Gamma~.\]
This gives a decomposition of the $\sG$ character variety:
\[\Xx(\sG)=\bigsqcup\limits_{a\in\Bun_S(\sG)}\Xx^a(\sG)~,\]
where $a\in\Bun_S(\sG)$ is the topological type of the flat $\sG$ bundle of the representations in $\Xx^a(\sG).$ 

We will rely heavily on the following theorem which was proven by Hitchin \cite{selfduality}, Donaldson \cite{harmoicmetric}, Corlette \cite{canonicalmetrics} and Simpson \cite{SimpsonVHS} in various generalities. For the proof of the general statement below, see \cite{HiggsPairsSTABILITY}.
\begin{Theorem}\label{Nonabelian Hodge Correspondence}
	Let $S$ be a closed oriented surface of genus $g\geq2$ and $\sG$ be a real algebraic semisimple Lie group. 
	For each Riemann surface structure $X$ on $S$ there is a homeomorphism between the moduli space $\Mm(\sG)$ of $\sG$ Higgs bundles on $X$ and the $\sG$-character variety $\Xx(\sG)$. Furthermore, this homeomorphism is a diffeomorphism when restricted to the smooth loci. 
	Moreover, for each $a\in\Bun_S(\sG)$, this homeomorphism identifies the spaces $\Mm^a(\sG)$ and $\Xx^a(\sG).$
\end{Theorem}
In Sections \ref{Section Zariski closure} and \ref{section: Positive}, it will be important to determine when a representations has smaller Zariski closure. This leads to the definition of a representation factoring through a reductive subgroup.
\begin{Definition}\label{DEF: Factoring Representations}
	Let $\sG$ and $\sG'$ be reductive Lie groups and $i:\sG'\ra\sG$ be an embedding.
	A representation $\rho:\Gamma\to\sG$ {\em factors through $\sG'$} if there exists a representations $\rho':\Gamma\ra\sG'$ such that $\rho=i\circ\rho'.$
\end{Definition}
\begin{Remark}
	The group $\sSO(1,2)$ is the set of isometries of the hyperbolic plane and $\sSO_0(1,2)$ is the set of orientation preserving isometries.  Note that since the surface $S$ is assumed to be orientable, all Fuchsian representations $\rho$ from Example \ref{EX: Fuchsian reps} factor through the connected component of the identity $\sSO_0(1,2).$ 
\end{Remark}
The following proposition is immediate from Theorem \ref{Nonabelian Hodge Correspondence}. 
\begin{Proposition}\label{Prop smaller zariski closure}
	Let $\sG'$ be a reductive Lie subgroup of a semisimple Lie group $\sG.$ A reductive representation $\rho:\Gamma\ra\sG$ factors through a representation $\rho':\Gamma\ra\sG'$ if and only if the corresponding polystable $\sG$ Higgs bundle $(\Pp,\varphi)$ reduces to a $\sG'$ Higgs bundle. In particular, $\rho$ has compact Zariski closure if and only if $\varphi=0.$
\end{Proposition}

 

\begin{Definition}
	\label{DEF irreducible REP}
	Let $\sG$ be a real form of a subgroup of $\sSL(n,\C)$. A representation $\rho:\Gamma\ra\sG$ is called {\em irreducible} if the induced representation $\Gamma\ra\sSL(n,\C)$ has no nonzero proper invariant subspaces. 
\end{Definition}

For $\sG=\sSL(n,\C)$, Theorem \ref{Nonabelian Hodge Correspondence} gives a one to one correspondence between irreducible representations and stable $\sSL(n,\C)$ Higgs bundles \cite{selfduality,localsystems}. This implies the following proposition which will play a key role in Sections \ref{Section Zariski closure} and \ref{section: Positive}.
\begin{Proposition}\label{Prop Irreducible reps and Stable SLnC Higgs bundles}
	Suppose $\sG$ is a real form of an irreducible subgroup of $\sSL(n,\C)$. Let $\rho:\Gamma\ra\sG$ be a reductive representation and let $(\Pp,\varphi)$ be the corresponding $\sG$ Higgs bundle given by Theorem \ref{Nonabelian Hodge Correspondence}. The representation $\rho$ is irreducible if and only if the $\sSL(n,\C)$ Higgs bundle associated to $(\Pp,\varphi)$ is stable.
\end{Proposition}

\section{$\sSO(n,n+1)$ Higgs Bundles}
\label{Section sonn+1 Higgs}
In this section we specialize to the group $\sSO(n,n+1)$ of orientation preserving automorphisms of $\R^{2n+1}$ which preserve a nondegenerate symmetric quadratic form of signature $(n,n+1).$ The group $\sSO(n,n+1)$ has two connected components, denote the connected component of the identity by $\sSO_0(n,n+1).$
If $Q_n$ and $Q_{n+1}$ are positive definite symmetric $n\times n$ and $(n+1)\times(n+1)$ matrices, then the Lie algebra $\fso(n,n+1)$ is defined by the matrices
\[\left\{\mtrx{A&B\\C&D}\ \left|\  \mtrx{A&B\\C&D}^T\right.\mtrx{Q_n&\\&-Q_{n+1}}+\mtrx{Q_n&\\&-Q_{n+1}}\mtrx{A_{ }&B\\C&D}=0\right\}~,\]
where $A$ is an $n\times n$ matrix, $B$ is an $n\times (n+1)$ matrix, $C$ is an $(n+1)\times n$ matrix and $D$ is an $(n+1)\times(n+1)$ matrix.
Thus, 
\begin{equation}
\label{EQ so(n,n+1) Lie algebra}	\xymatrix@C=.5em{A^TQ_n+Q_nA=0,& D^TQ_{n+1}+Q_{n+1}D=0&\text{and}&B=-Q_n^{-1}C^TQ_{n+1}~.}
\end{equation}

The maximal compact subgroup of $\sSO(n,n+1)$ is $\sS(\sO(n)\times\sO(n+1))$. Using \eqref{EQ so(n,n+1) Lie algebra}, the complexified Cartan decomposition of the Lie algebra $\fso(n,n+1)\otimes \C$ is 
\[\fso(2n+1,\C)=(\fso(n,\C)\oplus\fso(n+1,\C))\oplus \Hom(V,W)\]
 where $V$ and $W$ are the standard representations of $\sO(n,\C)$ and $\sO(n+1,\C).$ 
 Using Definition \ref{Def GHiggsBundle}, an $\sSO(n,n+1)$ Higgs bundle on $X$ is a pair $(\Pp,\varphi)$ where $\Pp\ra X$ is a holomorphic $\sS(\sO(n,\C)\times\sO(n+1,\C))$-principal bundle and $\varphi$ is a holomorphic section of $\Pp[\Hom(V,W)]\otimes K.$ 

Given a holomorphic principal $\sO(n,\C)$ bundle, the rank $n$ vector bundle $V$ associated to the standard representation satisfies $det(V)^2=(\Lambda^nV)^2\cong \Oo$. 
Furthermore, $V$ admits an orthogonal structure $Q_V\in H^0(Sym^2(V))$. An orthogonal structure $Q_V$ will be interpreted as a holomorphic symmetric isomorphism $Q_V:V\to V^*.$ 
We take the following vector bundle definition of an $\sSO(n,n+1)$ Higgs bundle. 

\begin{Definition}\label{DEF SO(n,n+1) Higgs bundles}
An {\em $\sSO(n,n+1)$ Higgs bundle} on Riemann surface $X$ is a triple $(V,W,\eta)$ where 
\begin{itemize}
	\item $V$ and $W$ are respectively rank $n$ and $(n+1)$ holomorphic vector bundles on $X$ equipped with holomorphic orthogonal structures $Q_V$ and $Q_W$ such that $det(V)=det(W).$
	\item $\eta\in H^0(\Hom(V,W)\otimes K).$
\end{itemize}
\end{Definition}
\begin{Remark}
	An $\sSO(n,n+1)$ Higgs bundle $(V,W,\eta)$ reduces to an $\sSO_0(n,n+1)$ Higgs bundle if and only if $\Lambda^nV=\Oo$. 
\end{Remark}
The $\sS(\sO(n,\C)\times\sO(n+1,\C))$-gauge group consists of pairs $(g_V,g_W)$ where $g_V$ and $g_W$ are smooth automorphisms of $V$ and $W$ such that 
\[\xymatrix{g_V^TQ_Vg_V=Q_V~,& g_W^TQ_Wg_W=Q_W& \text{and} }\]
\[Id=det(g_V)\otimes det(g_W):\Lambda^n V\otimes\Lambda^{n+1}W\longrightarrow \Lambda^n V\otimes\Lambda^{n+1}W~.\] Such a gauge transformation acts on the data $(V,W,\eta)$ by 
	\[(g_V,g_W)\cdot (\bar\p_V,\bar\p_W,\eta)=(g_V\bar\p_Vg_V^{-1},g_{W}\bar\p_Wg_{W}^{-1},g_W\eta g_{V}^{-1})~.\]

The $\sSL(2n+1,\C)$ Higgs bundle $(E,\Phi)$ associated to an $\sSO(n,n+1)$ Higgs bundle $(V,W,\eta)$ is given by 
\begin{equation}\label{Eq SL(2n+1,C) Higgs bundle}
	(E,\Phi)=\left(V\oplus W, \mtrx{0&\eta^*\\\eta&0}\right)~,
\end{equation}
where $\eta^*$ is defined by $\eta^*=-Q_V^{-1}\circ \eta^T\circ Q_W:W\ra V\otimes K$.
 Such a Higgs bundle will be represented schematically as 
\[
	\xymatrix{V\ar@/_.5pc/[r]_\eta& W\ar@/_.5pc/[l]_{\eta^*}}
\]
where we have suppressed the twisting by $K$ from the notation. 

Note that, when restricted to $\sSO(n,n+1)$-Higgs bundles, the Hitchin fibration \eqref{EQ Hitchin Fibration} maps to the space of even holomorphic differentials. Indeed, $Tr(\Phi^j)$ form a basis of invariant polynomials and for Higgs fields $\Phi$ of the form \eqref{Eq SL(2n+1,C) Higgs bundle}, $Tr(\Phi^j)=0$ for $j$ odd and $Tr(\Phi^{2j})=2Tr((\eta^*\otimes\eta)^j).$ The expected dimension of the moduli space $\Mm(\sSO(n,n+1))$ is
	\begin{equation}
		\label{EQ Expected Dimension for SO(n,n+1)}
		dim(\sSO(n,n+1))\cdot(2g-2)= n(2n+1)(2g-2)=dim\left(\bigoplus\limits_{j=1}^n H^0(K^{2j})\right)~.
	\end{equation}

 We will use the following proposition to conclude the hypercohomology group $\HH^2(C^\bullet(V,W,\eta))$ vanishes in some nice cases (see \cite[Proposition 3.17]{HiggsPairsSTABILITY}).
\begin{Proposition}\label{Prop: H2 vanish of stable}
	If $(V,W,\eta)$ is a polystable $\sSO(n,n+1)$ Higgs bundle such that the associated $\sSL(2n+1,\C)$ Higgs bundle given by \eqref{Eq SL(2n+1,C) Higgs bundle} is {\em stable}, then 
	\[\HH^2(C^\bullet(V,W,\eta))\equiv0~.\]
\end{Proposition}
By Proposition \ref{Prop Orbifold/Smooth GHiggs}, we have the following corollary.
\begin{Corollary}\label{Corollary smooth/orbifold points}
If $(V,W,\eta)$ is a polystable $\sSO(n,n+1)$ Higgs bundle such that the associated $\sSL(2n+1,\C)$ Higgs bundle given by \eqref{Eq SL(2n+1,C) Higgs bundle} is {\em stable}, then the isomorphism class of $(V,W,\eta)$ defines an $\Aut(V,W,\eta)$-orbifold point of $\Mm(\sSO(n,n+1))$. 
In particular, since the center of $\sS(\sO(n,\C)\times \sO(n+1,\C))$ is trivial, $(V,W,\eta)$ defines a smooth point of $\Mm(\sSO(n,n+1))$ if and only if $\Aut(V,W,\eta)$ is trivial.
 \end{Corollary}

\subsection{Topological classes of $\sSO(n,n+1)$ bundles on $X$} 
Recall that the set of equivalences classes of $\sSO(n,n+1)$ bundles $\Bun_X(\sSO(n,n+1))$ on $X$ gives a decomposition of the moduli space of $\sSO(n,n+1)$ Higgs bundles. Recall also that $\Bun_X(\sSO(n,n+1))=\Bun_X(\sS(\sO(n)\times\sO(n+1)))$ since $\sSO(n,n+1)$ is homotopy equivalent to its maximal compact subgroup.

An $\sO(n)$ bundle $V\to X$ has a first and second Stiefel-Whitney class  
\[\xymatrix{sw_1(V)=sw_1(\Lambda^nV)\in H^1(X,\pi_0(\sO(n)))&\text{and}& sw_2(V)\in H^2(X,\Z_2)~.}\] 
When $n\geq3,$ $\pi_1(\sO(n))=\Z_2$ and these characteristic classes are in bijective correspondence with $\Bun_X(\sO(n)).$

\begin{Proposition}
 	For $n\geq3$, we have 
 	\[\Bun_X(\sSO(n,n+1))\cong H^1(X,\Z_2)\times H^2(X,\Z_2)\times H^2(X,\Z_2)~.\]
 \end{Proposition} 

\begin{proof}
	An $\sS(\sO(n)\times\sO(n+1))$ bundle is equivalent to a pair $(V,W)$ where $V$ is an $\sO(n)$ bundle and $W$ is an $\sO(n+1)$ bundle with $det(V)=det(W).$ 
The Stiefel-Whitney classes of $V$ and $W$ determine the topological class of an $\sSO(n,n+1)$ bundle, but, since $det(V)=det(W)$ we have $sw_1(V)=sw_1(W).$ 
\end{proof}
For $n\geq 3$, the moduli space of $\sSO(n,n+1)$ Higgs bundles thus decomposes as 
\begin{equation}\label{EQ Decomp SOnn+1 top inv}
	\Mm(\sSO(n,n+1))=\bigsqcup\limits_{\substack{(a,b,c)\in\\H^1(X,\Z_2)\times H^2(X,\Z_2)\times H^2(X,\Z_2)}}\Mm^{a,b,c}(\sSO(n,n+1))~.
\end{equation}
Moreover, an $\sSO(n,n+1)$ Higgs bundle in $\Mm^{a,b,c}(\sSO(n,n+1))$ reduces to an $\sSO_0(n,n+1)$ Higgs bundle if and only if $a=0.$


\subsection{$\sSO(1,2)=\sP\sGL(2,\R)$ Higgs bundles} For $\sSO(1,2),$ we can explicitly describe the Higgs moduli space. Moreover, in this case, the connected component description is deduced from topological invariants of orthogonal bundles. Although these results are not new, we include the arguments here since the methods will be generalized in subsequent sections. One important difference of the $\sSO(n,n+1)$ generalizations is that they are not distinguished by a known topological invariant for $n\geq 3$. 

Using Definition \ref{DEF SO(n,n+1) Higgs bundles}, an $\sSO(1,2)$ Higgs bundle $(V,W,\eta)$ is given by $(\Lambda^2 W,W,\eta)$ where $W$ is a rank two holomorphic vector bundle with an orthogonal structure $Q_W$. The $\sSL(3,\C)$ Higgs bundle associated to $(\Lambda^2 W,W,\eta)$ is represented by 
\begin{equation}
	\label{EQ SL3 Higgsbundle}	\xymatrix{\Lambda^2W\ar@/_1pc/[r]_{\eta}&W\ar@/_1pc/[l]_{\eta^*}}~.
\end{equation}
As above, rank $2$ orthogonal bundles on $X$ have first and second Stiefel-Whitney classes $(sw_1,sw_2)\in H^1(X,\Z_2)\oplus  H^2(X,\Z_2).$ 
If $\Mm^{sw_2}_{sw_1}(\sSO(1,2))$ is the moduli space of $\sSO(1,2)$ Higgs bundles consisting of triple $(\Lambda^2 W, W,\eta)$ where the first and second Stiefel-Whitney classes $W$ are $(sw_1,sw_2)$, then
\begin{equation}
	\label{EQ sw1 sw2 decomposition} \Mm(\sSO(1,2))=\bigsqcup\limits_{\substack{(sw_1,sw_2)\in\\ H^1(X,\Z_2)\oplus H^2(X,\Z_2)}}\Mm_{sw_1}^{sw_2}(\sSO(1,2))~.
\end{equation}

If the first Stiefel-Whitney class of $W$ vanishes, then the structure group of $W$ reduces to $\sSO(2,\C)$. 
Since $\sSO(2,\C)\cong\C^*,$ a holomorphic orthogonal bundle $(W,Q_W)$ is isomorphic to 
\[(W,Q_W)=\left(M\oplus M^{-1},\mtrx{0&1\\1&0}\right)~,\] where $M\in\Pic^d(X)$ is a degree $d$ holomorphic line bundle. 
In this case, $\Lambda^2W\cong\Oo$, the second Stiefel-Whitney class is given by the degree of $M$ mod $2,$ and the Higgs field $\eta$ decomposes as $\eta=(\mu,\nu)\in H^0(M^{-1}K)\oplus H^0(MK)$. 
The associated $\sSL(3,\C)$ Higgs bundle given by
\begin{equation}
	\label{EQ SL3 Higgsbundle sw_1=0}
	\xymatrix{M\ar@/_1pc/[r]_{\mu}&\Oo\ar@/_1pc/[r]_{\mu}\ar@/_1pc/[l]_{\nu}&M^{-1}\ar@/_1pc/[l]_{\nu}}~.
\end{equation} 

If $deg(M)>0$, then the $\sSO(1,2)$ Higgs bundle \eqref{EQ SL3 Higgsbundle sw_1=0} is polystable if and only if $\mu\neq0\in H^0(M^{-1}K).$ Thus $deg(M)\leq 2g-2.$ 
Note that the $\sS(\sO(1,\C)\times\sO(2,\C))$ gauge transformation\footnote{Note that this switching isomorphism is in the $\sS(\sO(1,\C)\times\sO(2,\C))$-gauge group but not the $\sSO(1,\C)\times\sSO(2,\C)$-gauge group. In fact, the moduli space $\Mm(\sSO_0(1,2))$ is a double cover of $\Mm_{sw_1=0}(\sSO(1,2))$. The fiber of the map $\Mm(\sSO_0(1,2))\to\Mm_d(\sSO(1,2))$ is connected when $d=0$ and consists of two isomorphic components if $d\neq0$.}
\begin{equation}\label{EQ switching iso 1,2}
	\mtrx{&&-1\\&-1&\\-1&&}:M\oplus\Oo\oplus M^{-1}\longrightarrow M^{-1}\oplus \Oo\oplus M
\end{equation}
gives an isomorphism between the data $(M,\mu,\nu)$ and $(M^{-1},\nu,\mu).$ Thus we may assume $deg(M)\geq0.$ 

Let $\Mm_d(\sSO(1,2))$ denote the moduli space of polystable $\sSO(1,2)$ Higgs bundles of the form \eqref{EQ SL3 Higgsbundle} with vanishing first Stiefel-Whitney class and $deg(M)=d.$ The moduli space $\Mm_{sw_1=0}(\sSO(1,2))$ decomposes as 
	\[\Mm_{sw_1=0}(\sSO(1,2))=\bigsqcup\limits_{0\leq d\leq 2g-2}\Mm_d(\sSO(1,2))~.\]
Hitchin proved the following theorem for $\sP\sSL(2,\R)=\sSO_0(1,2)$.
\begin{Theorem}( \cite[Theorem 10.8]{selfduality})\label{THM Md SO(1,2)} For each integer $d\in (0,2g-2]$, the moduli space $\Mm_d(\sSO(1,2))$ is smooth and diffeomorphic to a rank $(d+g-1)$-vector bundle $\Ff_d$ over the $(2g-2-d)$-symmetric product $\Sym^{2g-2-d}(X).$
\end{Theorem}
\begin{proof}
	Let $\widetilde\Ff_d=\{(M,\mu,\nu)\ |\ M\in\Pic^d(X),\ \mu\in H^0(M^{-1}K),\ \nu\in H^0(MK)\}$. By the above discussion, there is a surjective map $\widetilde \Ff_d\to\Mm_d(\sSO(1,2))$ defined by sending $(M,\mu,\nu)$ to the isomorphism class of the Higgs bundle \eqref{EQ SL3 Higgsbundle}. 
	It is straight forward to check that the $\sSO(1,2)$ Higgs bundles associated two points $(M,\mu,\nu)$ and $(M',\mu',\nu')$ lie in the same gauge orbit if and only if $M'=M,$ $\mu'=\lambda\mu$ and $\nu'=\lambda^{-1}\nu$ for $\lambda\in\C^*$.

	This gives a diffeomorphism between the quotient space $\Ff_d=\widetilde\Ff_d/\C^*$ and the moduli space $\Mm_d(\sSO(1,2))$. 
	The map $\pi_d:\Ff_d\to\Sym^{2g-2-d}(X)$ defined by taking the projective class of $\mu$ is surjective. 
	For a divisor $D\in\Sym^{2g-2-d}(X),$ the fiber $\pi^{-1}(D)$ is (non-canonically) identified with $H^0(\Oo(-D)K^2)\cong\C^{d+g-1}$ where $\Oo(-D)$ is the inverse of the line bundle associated to $D.$
\end{proof}
\begin{Remark}\label{Rem Higgs bundles for Fuch}
	Note that when the integer invariant $d=2g-2,$ the connected component $\Mm_d(\sSO(1,2))$ is diffeomorphic to the vector space $H^0(K^2)$ of holomorphic differentials on $X.$ These are the Higgs which correspond to the Fuchsian representations from Example \ref{EX: Fuchsian reps}. In particular, we recover the classical result that the Teichm\"uller space of $S$ is diffeomorphic to a vector space of complex dimension $3g-3$. Moreover, the Fuchsian representation which corresponds to zero in $H^0(K^2)$ uniformizes the Riemann surface $X.$
\end{Remark}
\begin{Theorem}\label{THM M0 homotopy type}
	The space $\Mm_0(\sSO(1,2))$ deformation retracts onto $\Pic^0(X)/\Z_2$ where $\Z_2$ acts by inversion. In particular, $\Mm_0(\sSO(1,2))$ is homotopy equivalent to the quotient of a $2g$-dimensional torus by inversion.
\end{Theorem}
\begin{proof}
	Let $(M,\mu,\nu)$ be an $\sSO_0(1,2)$ Higgs bundle with $deg(M)=0.$ The associated $\sSL(3,\C)$ Higgs bundle is given by \eqref{EQ SL3 Higgsbundle}. Since $deg(M)=0,$ the bundle $M\oplus M^{-1}\oplus \Oo$ is polystable as a holomorphic vector bundle. Thus, the family $(M,t\mu,t\nu)$ is a family of polystable $\sSO(1,2)$ Higgs bundles which converge to $(M,0,0).$ Finally, the $\sS(\sO(2)\times\sO(1))$-gauge transformation \eqref{EQ switching iso 1,2} defines an isomorphism between the $\sSO(1,2)$ Higgs bundles associated to $(M,0,0)$ and $(M^{-1},0,0).$ 
\end{proof}

So far we have assumed that the first Stiefel-Whitney class of the $\sO(2,\C)$ bundle $W$ is zero. 
Equivalently, we have only considered $\sSO(1,2)$ Higgs bundles which reduce to $\sSO_0(1,2)$ Higgs bundles. 
We now recall Mumford's description of holomorphic $\sO(2,\C)$ bundles \cite{MumO2Bun}. 

\begin{Proposition}\label{Prop Mumford O2} Let $sw_1\in H^1( X,\Z_2)\setminus\{0\}$ with corresponding unramified double cover $\pi:X_{sw_1}\to X$, and denote the covering involution by $\iota:X_{sw_1}\to X_{sw_1}$. 
Consider the following space:
\begin{equation}
	\label{Eq Prym def}
	\Prym(X_{sw_1},X)=\{M\in\Pic^0(X_{sw_1})\ |\ \iota^*M=M^{-1}\}~.
\end{equation}
There is a bijection between $\Prym(X_{sw_1},X)$ and holomorphic $\sO(2,\C)$ bundles on $ X$ with first Stiefel-Whitney class $sw_1$ given by 
\[\xymatrix{M\ar@{|->}[r]&(W,Q_W)=(\pi_*M,\pi_*\iota^*)}~.\]
\end{Proposition}
\begin{proof}
Let $(W,Q_W)$ be a holomorphic $\sO(2,\C)$ bundle on $X$ with first Stiefel-Whitney class $sw_1\neq0.$ Since $X_{sw_1}$ is the orientation double cover, we have $sw_1(\pi^*W,\pi^*Q_W)=0$. 
Thus,  
\[(\pi^*W,\pi^*Q_W)\cong\left(M\oplus M^{-1},\smtrx{0&1\\1&0}\right)\] and $(\pi^*W,\pi^*Q_W)$ is invariant under the covering involution 
	\[\iota^*(M\oplus M^{-1})=\iota^*M\oplus\iota^* M^{-1}\cong M\oplus M^{-1}~.\]
Given $M\in Pic^0(X_{sw_1})$ with $\iota^*M=M^{-1}$ we get an orthogonal bundle $(W,Q_W)=(\pi_*M,\pi_*\iota^*)$. 
Since $X_{sw_1}\to X$ is unramified, $\pi^*\pi_*(M)=M\oplus \iota^*M$, and the above construction gives a bijection. 
\end{proof}
\begin{Remark}\label{Remark prym variety torsor}
	The space $\Prym(X_{sw_1},X)$ has two connected components. For $M\in\Prym(X_{sw_1},X)$, the second Stiefel-Whitney class of the orthogonal bundle $\pi_*M$ distinguishes the connected component which contains $M$ \cite{MumO2Bun}. We will write 
	\begin{equation}\label{EQ prym decomp}
		\Prym(X_{sw_1},X)=\bigsqcup\limits_{sw_2\in H^2(X,\Z_2)}\Prym^{sw_2}(X_{sw_1},X)~.
	\end{equation}
	The connected component of the identity, $\Prym^{0}(X_{sw_1},X)$, is an $g-1$ dimensional abelian variety called the {\em Prym variety} of the covering $X_{sw_1}\to X.$ Moreover, $\Prym^1(X_{sw_1},X)$ is a $\Prym^0(X_{sw_1},X)$ torsor. 
\end{Remark}
As in the proof of Theorem \ref{THM M0 homotopy type}, the bundle $\Lambda^2 W\oplus W$ is a polystable vector bundle for an $\sSO(1,2)$ Higgs bundle $(W,\eta)$ which defines a point of $\Mm_{sw_1}^{sw_2}(\sSO(1,2))$. Thus, the family of polystable Higgs bundles $(W,t\eta)$ converges to $(W,0).$ 
Furthermore, the $\sS(\sO(1,\C)\times\sO(2,\C))$-gauge transformation $(g_{\Lambda^2W},g_W)=(det(Q_W),Q_W)$ defines an isomorphism between $(W,\eta)$ and $(W^*,\eta^*).$ 
Thus, we have proven the following:
\begin{Theorem}\label{THM Msw1 SO(1,2)}
	 For $(sw_1,sw_2)\in (H^1(X,\Z_2)\setminus\{0\})\times H^2(X,\Z_2)$, the connected component $\Mm_{sw_1}^{sw_2}(\sSO(1,2))$ from \eqref{EQ sw1 sw2 decomposition} deformation retracts onto the moduli space $\Mm_{sw_1}^{sw_2}(\sS(\sO(1)\times\sO(2))).$ Since $\Mm_{sw_1}^{sw_2}(\sS(\sO(1)\times\sO(2)))$ is given by the quotient of the torus $\Prym^{sw_2}(X_{sw_1},X)$ by inversion, the space $\Mm_{sw_1}^{sw_2}(\sSO(1,2))$ is homotopy equivalent to the quotient of a $(2g-2)$-dimensional torus by inversion.
\end{Theorem}

\section{Parameterizing the smooth components $\Mm_d(\sSO(n,n+1))$}\label{section: smooth components}
In this section we will prove Theorems \ref{THM1}. We start by recalling Hitchin's parameterization of the $\sSO_0(n,n+1)$-Hitchin component. 

Recall from Example \ref{EX: Fuchsian reps} that the set of Fuchsian representations $\Fuch(\Gamma)\subset\Xx(\sSO(1,2))$ defines a particularly interesting class of representations. Recall that the second symmetric product of the standard representation of $\sGL(2,\R)$ on $\R^2$ is the standard representation of $\sSO(1,2)$ on $\R^3$.
 The $2n^{th}$-symmetric product of the standard representation of $\sGL(2,\R)$ defines an irreducible representation $\sSO(1,2)\to\sSL(2n+1,\R)$ which preserves a signature $(n,n+1)$ quadratic form on $\R^{2n+1}.$ Thus we have an irreducible representation
\[i:\sSO(1,2)\to\sSO(n,n+1).\]
This defines a map $\iota:\Xx(\sSO(1,2))\to\Xx(\sSO(n,n+1))$, where $\iota(\rho)=i\circ\rho.$

\begin{Definition}\label{Def: Hitchin comp}
The $\sSO(n,n+1)$-{\em Hitchin component} $\Hit(\sSO(n,n+1))$ is the connected component of $\Xx(\sSO(n,n+1))$ that contains $\iota(\Fuch(\Gamma)).$
\end{Definition}
\begin{Remark}
The map $i:\sSO(1,2)\to\sSO(n,n+1)$ is an example of a principal embedding of $\sP\sSL(2,\R)\cong\sSO_0(1,2)$ into a split real Lie group $\sG$ of adjoint type. The Hitchin component for a split group $\sG$ is defined as the deformation space of the image of $\iota(\Fuch(\Gamma))$ in $\Xx(\Gamma,\sG)$. See \cite{ptds} and \cite{liegroupsteichmuller} for more details.
\end{Remark}

\begin{Theorem}\label{THM Hitchin component}
	\cite[Theorem 7.5]{liegroupsteichmuller} The Hitchin component $\Hit(\sSO(n,n+1))$ is diffeomorphic to the vector spaces of holomorphic differentials $\bigoplus\limits_{j=1}^{n} H^0(K^{2j}).$ 
\end{Theorem}
For $\Hit(\sSO(n,n+1))$, the map $\bigoplus\limits_{j=1}^{n} H^0(K^{2j})\longrightarrow \Mm(\sSO(n,n+1))$ can be defined by sending a tuple of differentials $(q_2,q_4,\cdots,q_{2n})$ to the Higgs bundle $(V,W,\eta)$ where
\[\xymatrix@R=.2em{V=K^{n-1}\oplus K^{n-3}\oplus\cdots\oplus K^{3-n}\oplus K^{1-n}~,\\ W=K^{n}\oplus K^{n-2}\oplus\cdots\oplus K^{2-n}\oplus K^{-n}~,}  \]
\vspace{-.5em}
\begin{equation}
	\label{EQ Hitchin component Higgs field SO(n,n+1)}
	\ \ \ \ \ \ \eta=\mtrx{q_2&q_4&q_6&\cdots&q_{2n-2}&q_{2n}\\
			   1&q_2&q_4&\cdots&q_{2n-4}&q_{2n-2}\\
			   0&1&q_2&q_4&\cdots&q_{2n-4}\\
			   0&0&1&q_2&\cdots&q_{2n-6}\\
			   &\ddots&\ddots&\ddots&&\\
			   &&&0&1&q_2\\
			   &&&&0&1}:V\longrightarrow W\otimes K
\end{equation}
The orthogonal structures on $V$ and $W$ are the standard ones:
\[Q_V=\smtrx{&&1\\&\iddots&\\1&&}:V\to V^*\ \ \ \ \ \ \ \ \text{and}\ \ \ \ \ \ \ \ \ Q_W=\smtrx{&&1\\&\iddots&\\1&&}:W\ra W^*~.\]

\begin{Remark}\label{REMARK Aut of Hit comp is trivial}
	Since the Hitchin component is smooth, the automorphism group of $(V,W,\eta)$ of the form \eqref{EQ Hitchin component Higgs field SO(n,n+1)} is trivial. Also, since $\Lambda^n(V)=\Oo$ and $\Lambda^{n+1}W=\Oo,$ all Higgs bundles in $\Hit(\sSO(n,n+1))$ reduce to $\sSO_0(n,n+1)$ Higgs bundles. 
\end{Remark}
Recall from Remark \ref{Rem Higgs bundles for Fuch} that the $\sSO(2,1)$ Higgs bundles which give rise to the Fuchsian representations which uniformizes the Riemann surface $X$ has $\sSL(3,\C)$ Higgs bundle given by 
\begin{equation}
	\label{EQ Fuch Higgs bundle}(E,\Phi)=\left(K\oplus\Oo\oplus K^{-1}\ ,\ \mtrx{0&0&0\\1&0&0\\0&1&0}\right)~.
\end{equation} 
Moreover, the $\sSL(2n+1,\C)$ Higgs bundle $\left(V\oplus W,\mtrx{0&\eta^*\\\eta&0}\right)$ associated to the locus where the differential $q_2,\cdots,q_{2n}$ are all zero is the $n^{th}$ symmetric product of the Higgs bundle $(E,\Phi)$ from \eqref{EQ Fuch Higgs bundle}. Thus, Theorem \ref{THM Hitchin component} really does parameterizes the Hitchin component from Definition \ref{Def: Hitchin comp}.
\subsection{The components $\Mm_d(\sSO(n,n+1))$}
We will now show that the connected components $\Mm_d(\sSO(1,2))$ from Theorem \ref{THM Md SO(1,2)} generalize to $\Mm(\sSO(n,n+1)).$ We start with some preliminary lemmas.
 \begin{Lemma}\label{Lemma defining Higgs bundles in Md}
 	For each integer $d\in(0,n(2g-2)]$, define the space $\widetilde\Ff_d$ by 
 	\begin{equation}\label{EQ tildeFd DEF}
 		\widetilde\Ff_d=\{(M,\mu,\nu)\ |\ M\in\Pic^d(X),\ \mu\in H^0(M^{-1}K^n)\setminus\{0\},\ \text{and}\ \nu\in H^0(MK^n)\}~.
 	\end{equation} 
 	There is a well defined smooth map 
 	\begin{equation}
 		\label{EQ tildePsid Def}
 		\widetilde\Psi_d:\xymatrix{\widetilde\Ff_d\times \bigoplus\limits_{j=1}^{n-1} H^0(K^{2j})\ar[r]&\{\text{stable}\ \sSO(n,n+1)\text{-Higgs\ bundles}}\}
 	\end{equation}
 	defined by $\widetilde\Psi_d(M,\mu,\nu,q_2,\cdots q_{2p-2})=(V,W,\eta)$ where 
 	\[(V,Q_V)=\big(K^{n-1}\oplus K^{n-3}\oplus\cdots\oplus K^{3-n}\oplus K^{1-n},\smtrx{&&1\\&\iddots&\\1&&}\big)~,\]
 	\[(W,Q_W)=\big(M\oplus K^{n-2}\oplus K^{n-4}\oplus\cdots\oplus K^{4-n}\oplus K^{2-n}\oplus M^{-1},\smtrx{&&1\\&\iddots&\\1&&}\big)~,\]
 	\begin{equation}
 		\label{EQ Md Higgs field}
 		\eta=\mtrx{0&0&0&\cdots&0&\nu\\
			   1&q_2&q_4&\cdots&q_{2n-4}&q_{2n-2}\\
			   0&1&q_2&q_4&\cdots&q_{2n-4}\\
			   0&0&1&q_2&\cdots&q_{2n-6}\\
			   &\ddots&\ddots&\ddots&&\\
			   &&&0&1&q_2\\
			   &&&&0&\mu}:V\longrightarrow W\otimes K~.
 	\end{equation} 
 \end{Lemma}
 \begin{proof}
 	We will show the $\sSL(2n+1,\C)$ Higgs bundle $(E,\Phi)=\left(V\oplus W,\mtrx{0&\eta^T\\\eta&0}\right)$ corresponding to a Higgs bundle $(V,W,\eta)$ in \eqref{EQ Md Higgs field} is stable. 
 	For 
 	\[(q_2,\cdots,q_{2n-2},\nu)=(0,\cdots,0)~,\] the corresponding $\sSL(2n+1,\C)$ Higgs bundle can be written schematically as
 	\begin{equation}
 		\label{EQ Fixedpoint in Md}
 		\xymatrix@C=2em{M\ar[r]_{\mu\ \ }&K^{n-1}\ar[r]_1&K^{n-2}\ar[r]_1&\cdots\ar[r]_1&K^{2-n}\ar[r]_1&K^{1-n}\ar[r]_{\mu}&M^{-1}}~.
 	\end{equation}
For such Higgs bundles the above summands are the eigen-bundles of a holomorphic gauge transformation of $V\oplus W$. In particular, polystable Higgs bundles of this form define fixed points of the $\C^*$ action on the moduli space. 
 To check stability for such Higgs bundles, it suffices to consider invariant subbundles of each summand (see Proposition 6.3 of \cite{KatzMiddleInvCyclicHiggs}\footnote{ In \cite{KatzMiddleInvCyclicHiggs}, Simpson works with parabolic bundles, however the proof for the non parabolic case is identical.}). The Higgs bundle \eqref{EQ Fixedpoint in Md} is stable since each summand is a line bundle and the only invariant bundle is the negative degree line bundle $M^{-1}$. 

Stability is an open condition. Hence, there is an open neighborhood $U$ of $(0,\cdots,0)$ such that the Higgs bundles \eqref{EQ Md Higgs field} are stable for $(q_2,\cdots,q_{2n-2},\nu)\in U$. 
Let $(V,W,\eta)=\widetilde\Psi_d(M,\mu,\nu,q_2,\cdots,q_{2n-2})$. For $\lambda\in\C^*,$ consider the following holomorphic orthogonal gauge transformations of $V$ and $W$
\begin{equation}
	\label{EQ scaling gauges}g_V=\smtrx{\lambda^{n-1}&&&\\&\lambda^{n-3}&&\\&&\ddots&\\&&&\lambda^{1-n}}\ \ \ \text{and}\ \ \ g_W=\smtrx{\lambda^{n}&&&&\\&\lambda^{n-2}&&&\\&&\ddots&&\\&&&\lambda^{2-n}&\\&&&&\lambda^{-n}}~.
\end{equation}
A straight forward computation shows that 
\begin{equation}\label{EQ scaling gauge action}
	(g_V,g_W)\cdot(V,W,\lambda\eta)=\widetilde\Psi_d(M,\mu,\lambda^{2n}\nu,\lambda^2 q_2,\lambda^4q_4,\cdots,\lambda^{2n-2} q_{2n-2})~.
\end{equation} 
Since stability is preserved by scaling the Higgs field by $\C^*,$ the associated Higgs bundles \eqref{EQ Md Higgs field} is stable for all values of $(q_2,\cdots,q_{2n-2},\nu).$
 \end{proof}
 The next lemma is the key technical step in proving Theorem \ref{THM1}.  
 \begin{Lemma}\label{Lemma Image in Md}
 	Let $\widetilde \Psi_d:\widetilde\Ff_d\times \bigoplus\limits_{j=1}^{n-1}H^0(K^{2j})\to \{\text{polystable }\sSO(n,n+1)\ \text{Higgs bundles}\}$ be given by \eqref{EQ tildePsid Def}. If $(M,\mu,\nu,q_2,\cdots,q_{2n-2})$ and $(M',\mu',\nu',q_2',\cdots,q_{2n-2}')$ are two points in $\widetilde\Ff_d\times\bigoplus\limits_{j=1}^{n-1}H^0(K^{2j})$ , then the Higgs bundles $\widetilde\Psi_d(M,\mu,\nu,q_2,\cdots,q_{2n-2})$ and $\widetilde\Psi_d(M',\mu',\nu',q_2',\cdots,q_{2n-2}')$ lie in the same $\sS(\sO(n,\C)\times\sO(n+1,\C))$ gauge orbit if and only if for $\lambda\in\C^*$ and all $j,$
 	\[\xymatrix{M=M', &\mu=\lambda\mu',& \nu=\lambda^{-1}\nu'&\text{and}&q_{2j}'=q_{2j}}~.\]
 	 
 \end{Lemma}
 
 \begin{proof}
 	For $(M,\mu,\nu,q_2,\cdots,q_{2n-2})\in\widetilde\Ff_d\times\bigoplus\limits_{j=1}^{n-1}H^0(K^{2j})$, let $(V,W,\eta)$ be the $\sSO(n,n+1)$ Higgs bundle defined by $\widetilde\Psi_d(M,\mu,\nu,q_2,\cdots,q_{2n-2})$. It is given by \eqref{EQ Md Higgs field}. 
 	Write 
 	\begin{equation}\label{EQ W decomp}
 		\xymatrix{W=M\oplus W_0\oplus M^{-1}&\text{where}&W_0=K^{n-2}\oplus\cdots\oplus K^{2-n}}.
 	\end{equation}
 	Recall that the action of an element $(g_V,g_W)\in\Gg_{\sH_\C}(V,W,\eta)$ in the gauge group is given by 
 	\[(g_V,g_W)\cdot(\bar\p_V,\bar\p_W,\eta)=(g_V\bar\p_V g_V^{-1}\ ,\  g_W\bar\p_W g_W^{-1}\ ,\ g_W\eta g_V^{-1})\]
 	where $g_V$ and $g_W$ are smooth orthogonal gauge transformation with $det(g_V)\cdot det(g_W)=1.$
 	With respect to the decomposition \eqref{EQ W decomp}, a gauge transformation $g_W$ decomposes as
 	\begin{equation}\label{EQ decomposition gWd}
 		g_W=\mtrx{b_0&A&c_0\\
 				B&g_{W_0}&C\\
 				b_n&D&c_n}\ \ \ \ \ \ \ \ \ \ \text{and}\ \ \ \ \ \ \ \ \ \eta=\mtrx{\alpha\\\eta_0\\\beta}~,
 	\end{equation}
 where $g_{W_0}$ is an orthogonal gauge transformation of $W_0$ and
 \[ A=\mtrx{a_1&\cdots&a_{n-1}},\ \ D=\mtrx{d_1&\cdots&d_{n-1}},\ \ B=\mtrx{b_1\\\vdots\\b_{n-1}},\ \  C=\mtrx{c_1\\\vdots\\c_{n-1}}~,\]
\[\alpha=\mtrx{0&\cdots&0&\nu},\ \ \  \ \ \ \beta=\mtrx{0&\cdots&0&\mu}\ \  \ \ \text{and}\ \  \ \ \eta_0=\mtrx{1&q_2&\cdots& q_{2n-2}\\&\ddots&\ddots&\\0&\cdots&1&q_2}~.\]
For the Higgs field, $g_W\eta g_V^{-1}$ is given by 
\begin{equation}
	\label{EQ gauged transformed eta}
	g_W\eta g_V^{-1}=\mtrx{b_0\alpha+A\eta_0+c_0\beta\\B\alpha+g_{W_0}\eta_0+D\beta\\ c_0\alpha+C\eta_0+c_n\beta}g_V^{-1}=\mtrx{\alpha'\\\eta_0'\\\beta'}~.
\end{equation}
For $\alpha'=\mtrx{0&\cdots&0&\nu'}$, $\beta'=\mtrx{0&\cdots&0&\mu'}$, $\eta_0'=\mtrx{1& q_2'&\cdots&  q_{2n-2}'\\&\ddots&\ddots&\\0&\cdots&1& q_2'},$ the goal is to show that, if
\[(g_V,g_W)\cdot\left(\bar\p_V,\bar\p_W,\mtrx{\alpha\\\eta_0\\\beta}\right)= \left(\bar\p_V,\bar\p_W,\mtrx{\alpha'\\\eta_0'\\\beta'}\right), \]
then $g_V=Id_V$ and $g_W=\mtrx{\lambda&&\\&Id_{W_0}&\\&&\lambda^{-1}}.$

We will do this by first showing $A,B,C,D,c_0,b_n$ all vanish. Since $g_V$ and $g_W$ are holomorphic and negative degree line bundles do not have nonzero holomorphic sections, both $g_V$ and $g_{W_0}$ are upper triangular and $b_n=0$.
The term $\alpha'=\mtrx{0&\cdots&0&\nu'}$ is given by  $(b_0\alpha+A\eta_0+c_0\beta)\cdot g_V^{-1}$ from \eqref{EQ gauged transformed eta}. A computation shows $(b_0\alpha+A\eta_0+c_0\beta)$ is given by   
\begin{equation}
	\mtrx{a_1& a_1q_2+a_2&a_1q_4+a_2q_2+a_3&\cdots&b_0\nu+ \sum\limits_{j=1}^{n-1}a_{j}q_{2n-2j}+c_0\mu }~.
\end{equation}
Since $g_V^{-1}$ is invertible and upper triangular and $\alpha'=\mtrx{0&\cdots&0&\nu'}$, we conclude that $a_j=0$ for $1\leq j\leq n-1$. Hence the matrix $A$ in \eqref{EQ decomposition gWd} vanishes. By a similar computation, the matrix $D$ from \eqref{EQ decomposition gWd} also vanishes.

Recall that the gauge transformation $g_W$ is orthogonal with respect to the orthogonal structure $Q_W$, i.e. $g_W^TQ_Wg_W=Q_W$. If $Q_{W_0}$ is the restriction of the orthogonal structure $Q_W$ to the subbundle $W_0$, then, using the decomposition \eqref{EQ decomposition gWd}, 
\[g_W^TQ_Wg_W=\mtrx{b_0&B^T&0\\0&g_{W_0}^T&0\\c_0&C^T&c_n}\mtrx{&&1\\&Q_{W_0}\\1&&}\mtrx{b_0&0&c_0\\B&g_{W_0}&C\\0&0&c_n}=\mtrx{&&1\\&Q_{W_0}&\\1&&}.\]
Thus,
\begin{equation}
	\label{EQ b0bn=0}\mtrx{B^TQ_{W_0}b_0&B^TQ_{W_0}g_{W_0}&B^TQ_{W_0}C+b_0c_n\\ g_{W_0}^TQ_{W_0}B&g_{W_0}^TQ_{W_0}g_{W_0}&g_{W_0}^TQ_{W_0}C\\c_nb_0+C^TQ_{W_0}B&C^TQ_{W_0}g_{W_0}&2c_nc_0+C^TQ_{W_0}C}=\mtrx{0&0&1\\0&Q_{W_0}&0\\1&0&0}
\end{equation}
The term $C^TQ_{W_0}$ is given by $\mtrx{c_{n-1}&c_{n-2}&\cdots&c_1}.$ Since $C^TQ_{W_0}g_{W_0}=0$ and $g_{W_0}$ is invertible and upper triangular, we conclude that $C=0$. Similarly, the term $B$ also vanishes. This forces $0\neq b_0=c_n^{-1}$ and thus $c_0=0.$

Finally, by Hitchin's parameterization of $\Hit(\sSO(n-1,n))$, we conclude $g_V$ and $g_{W_0}$ are either both the identity or minus the identity. However, since $det(g_V)=det(g_W)=det(g_{W_0})$ we conclude, both $g_V$ and $g_{W_0}$ are the identity.
 \end{proof}
 Using Corollary \ref{Corollary smooth/orbifold points}, we have:
 \begin{Corollary}\label{COR Md Smooth}
If $(V,W,\eta)$ is an $\sSO(n,n+1)$ Higgs bundle of the form \eqref{EQ Md Higgs field}, then the automorphism group $\Aut(V,W,\eta)$ is trivial. In particular, the isomorphism class of such a $(V,W,\eta)$ defines a smooth point of $\Mm(\sSO(n,n+1)).$
 \end{Corollary}
\begin{Remark}\label{Remark Open and degM 0}
Note that the only time we used the fact that the degree of $M$ is nonzero was to conclude that $b_n$ and hence $c_0$ both vanish.
\end{Remark}
For $d\in(0,n(2g-2)]$, let
\[\widetilde\Ff_{-d}=\{(M,\nu,\mu)\ |\ M\in\Pic^{-d}(X),\ \nu\in H^0(MK^{n})\setminus \{0\},\ \text{and}\ \mu\in H^0(M^{-1}K^n)\}~,\]
and define a map 
\[\widetilde\Psi_{-d}:\xymatrix{\widetilde\Ff_{-d}\times \bigoplus\limits_{j=1}^{n-1} H^0(K^{2j})\ar[r]&\{\text{polystable}\ \sSO(n,n+1)\text{-Higgs\ bundles}}\}\]
by \eqref{EQ Md Higgs field}. 
We have the following proposition relating the images of $\widetilde\Psi_d$ and $\widetilde\Psi_{-d}.$
\begin{Proposition}
For $d\in(0,n(2g-2)],$ and $(M,\mu,\nu)\in\widetilde\Ff_{d},$ the stable $\sSO(n,n+1)$ Higgs bundles $\Psi_d(M,\mu,\nu,q_2,\cdots,q_{2n-2})$ and $\widetilde\Psi_{-d}(M^{-1},\nu,\mu,q_2,\cdots,q_{2n-2})$ are in the same $\sS(\sO(n,\C)\times\sO(n+1,\C))$-gauge orbit.
\end{Proposition}
\begin{proof}
	By Lemma \ref{Lemma Image in Md}, if $d\in(0,n(2g-2)]$ and $M\in\Pic^d(X),$ then the $\sSO(n,n+1)$ Higgs bundles $(V_{\pm d},W_{\pm d},\eta_{\pm d})$ which are given by $\widetilde\Psi_{d}(M,\mu,\nu,q_2,\cdots,q_{2n-2})$ and $\widetilde\Psi_{-d}(M^{-1},\nu,\mu,q_2,\cdots,q_{2n-2})$ have
	\[\xymatrix@R=0em{V_d=K^{n-1}\oplus K^{n-3}\oplus\cdots\oplus K^{3-n}\oplus K^{1-n}& W_d= M\oplus W_0\oplus M^{-1}\\ V_{-d}=K^{n-1}\oplus K^{n-3}\oplus\cdots\oplus K^{3-n}\oplus K^{1-n}& W_{-d}= M^{-1}\oplus W_0\oplus M} \]
	where $W_0=K^{n-2}\oplus K^{n-4}\oplus\cdots\oplus K^{2-n}$ and Higgs field $\eta_{\pm d}$ given by \eqref{EQ Md Higgs field}. 

Consider the following orthogonal gauge transformation 
	\begin{equation}
		\label{EQ switching M and M-1}
		g_W=\mtrx{&&-1\\&-Id_{W_0}&\\-1&&}:\xymatrix{ M\oplus W_0\oplus M^{-1}\ar[r]& M^{-1}\oplus W_0\oplus M}~.
	\end{equation} 
	A simple calculation shows that $(-Id_V,g_W)$ defines an $\sS(\sO(n,\C)\times\sO(n+1,\C))$-gauge transformation which provides the desired isomorphism.
\end{proof}
\begin{Remark}
	Note that if the gauge transformation \eqref{EQ switching M and M-1} defines an $\sSO(n,\C)\times\sSO(n+1,\C)$ gauge transformation if and only if $n$ is even. 
\end{Remark}
We are now set up to prove the $\widetilde\Psi_d$ maps onto a connected component of $\Mm(\sSO(n,n+1)),$ and hence onto a connected component of the $\sSO(n,n+1)$ character variety $\Xx(\sSO(n,n+1))$.

 \begin{Theorem}\label{THM1}
	Let $\Gamma$ be the fundamental group of a closed surface $S$ of genus $g\geq2$ and let $\Xx(\sSO(n,n+1))$ be the $\sSO(n,n+1)$-character variety of $\Gamma$. 
	For each integer $d\in(0,n(2g-2)],$ there is a smooth connected component $\Xx_d(\sSO(n,n+1))$ of $\Xx(\sSO(n,n+1))$ which does not contain representations with compact Zariski closure. Furthermore, for each choice of Riemann surface structure $X$ on $S,$ the space $\Xx_d(\sSO(n,n+1))$ is diffeomorphic to the product 
	\[\Xx_d(\sSO(n,n+1))\cong\Ff_d\times \bigoplus\limits_{j=1}^{n-1} H^0(K^{2j})~,\] where $\Ff_d$ is the total space of a rank $d+(2n-1)(g-1)$ vector bundle over the symmetric product $\Sym^{n(2g-2)-d}(X)$ and $H^0(K^{2j})$ is the vector space of holomorphic differentials of degree $2j.$ 
\end{Theorem}
\begin{proof}
Let $\widetilde\Ff_d$ be as in \eqref{EQ tildeFd DEF}. There is a free $\C^*$-action on $\widetilde\Ff_d$ given by 
\[\lambda\cdot(M,\mu,\nu)=(M,\lambda\mu,\lambda^{-1}\nu)~.\] 
Let $\Ff_d$ be the quotient, $\Ff_d=\widetilde\Ff_d/\C^*$.
By Lemma \ref{Lemma defining Higgs bundles in Md}, Lemma \ref{Lemma Image in Md} and Corollary \ref{COR Md Smooth}, there is a smooth map 
\[\Psi_d:\xymatrix{\Ff_d\times \bigoplus\limits_{j=1}^{n-1}H^0(K^{2j})\ar[r]&\Mm(\sSO(n,n+1))}\]
defined by 
\[\Psi_d([M,\mu,\nu], q_2,\cdots, q_{2n-2})= [\widetilde\Psi_d(M,\mu,\nu,q_2,\cdots,q_{2n-2})]\in\Mm(\sSO(n,n+1)).\]
which is a diffeomorphism onto its image.

Just as in Hitchin's proof of Theorem \ref{THM Md SO(1,2)}, there is a map from $\Ff_d$ to the $(n(2g-2)-d)^{th}$ symmetric product of $X$
defined by taking the projective class of the nonzero section $\mu\in H^0(M^{-1}K^n)\setminus\{0\}$ 
\[\pi_d:\xymatrix@R=0em{\Ff_d\ar[r]&\Sym^{n(2g-2)-d}(X)\\([M,\mu,\nu])\ar@{|->}[r]&[\mu]}~.\]

Given a divisor $D\in\Sym^{n(2g-2)-d}(X),$ the fiber $\pi_d^{-1}(D)$ is non-canonically identified with $H^0(\Oo(-D)\otimes K^n)\cong\C^{d+(2n-1)(g-1)}$, where $\Oo(-D)$ is the inverse of the line bundle associated to $D.$ 
Thus, the space $\Ff_d$ is a rank $d+(2n-1)(g-1)$-vector bundle over the compact space $\Sym^{n(2g-2)-d}(X).$

Let $(E,\Phi)$ be the $\sSL(2n+1,\C)$ Higgs bundle associated to the $\sSO(n,n+1)$ Higgs bundle $\widetilde\Psi_d(M,\mu,\nu,q_2,\cdots,q_{2n-2})$.  Let $h:\Mm(\sSL(2n+1,\C))\ra\bigoplus\limits_{j=2}^{2n+1}H^0(K^j)$ denote the Hitchin fibration defined by the basis of invariant polynomials $(p_1,\cdots,p_n)$ so that
\begin{equation}\label{EQ Hitchin fibration on Md}
p_j(\Phi)=\begin{dcases}
	q_{2j}&  1\leq j\leq n-1\\
	\mu\otimes\nu& j=n
\end{dcases}~.
  \end{equation}  

To show the image of $\Psi_d$ is closed, consider a divergent sequence $\{x_i\}$ in the image of $\Psi_d.$ Denote, the inverse image of $\{x_i\}$ by $\Psi_d^{-1}(x_i)=(y_i,q_2^i,\cdots,q_{2n-2}^i)$ for $y_i\in\Ff_d;$ this sequence diverges in $\Ff_d\times\bigoplus\limits_{j=1}^{n-1} H^0(K^{2j})$. 
Thus, either there exists $j$ so that $q_{2j}^i$ goes to infinity in $H^0(K^{2j})$ or, since $\Sym^{n(2g-2)-d}(X)$ is compact, there is a subsequence, $x_{i_k}$ so that $\pi_d(y_{i_k})$ converges to $y_\infty\in\Sym^{n(2g-2)-d}(X)$ and $y_{i_k}$ goes to infinity in the fiber direction. In either case, \eqref{EQ Hitchin fibration on Md} and the properness of the Hitchin fibration imply the sequence $\{x_i\}=\{\Psi_d(y_i,q_{2}^i,\cdots,q_{2n-2}^i)\}$ diverges in $\Mm(\sSO(n,n+1)).$ Thus, the image of $\Psi_d$ is a closed subset. 

By a simple calculation, the dimension of $\Ff_d\times\bigoplus\limits_{j=1}^{n-1}H^0(K^{2j})$ is the expected dimension of the moduli space $\Mm(\sSO(n,n+1)).$ 
Hence, since the $\Ff_d\times\bigoplus\limits_{j=1}^{n-1}H^0(K^{2j})$ is a manifold without boundary and the image of $\Psi_d$ is closed, we conclude that the image of $\Psi_d$ is also open.

We have established that for each integer $d\in(0,n(2g-2)]$ the image of $\Psi_d$ defines a smooth connected component $\Mm_d(\sSO(n,n+1))$ of $\Mm(\sSO(n,n+1))$. 
Recall that the correspondence between the $\sG$ Higgs bundle moduli space and the $\sG$ character variety is a diffeomorphism on the smooth locus.
Since, the connected components $\Mm_d(\sSO(n,n+1))$ are smooth, we conclude that for each integer $d\in(0,n(2g-2)]$ there is a smooth connected component $\Xx_d(\sSO(n,n+1))$ of the $\sSO(n,n+1)$-character variety which is diffeomorphic to $\Ff_d\times\bigoplus\limits_{j=1}^{n-1}H^0(K^{2j}).$ Finally, since the Higgs field is non-vanishing for Higgs bundles in $\Mm_d(\sSO(n,n+1)),$ no representation in $\Xx_d(\sSO(n,n+1))$ has compact Zariski closure by Proposition \ref{Prop smaller zariski closure}.
\end{proof}
\begin{Remark}
	For the maximal value $d=n(2g-2)$, the space $\Xx_{n(2g-2)}(\sSO(n,n+1))$ is the $\sSO(n,n+1)$-Hitchin component $\Hit(\sSO(n,n+1)).$
\end{Remark}

\begin{Corollary}
	The component $\Xx_d(\sSO(n,n+1))$ deformation retracts onto the symmetric product $\Sym^{n(2g-2)-d}(X).$ In particular, 
	\[H^*(\Xx_d(\sSO(n,n+1),\Z)\cong H^*(\Sym^{n(2g-2)-d}(X),\Z)~.\]
\end{Corollary}

Recall from \eqref{EQ Decomp SOnn+1 top inv} that the moduli space $\Mm(\sSO(n,n+1)$ decomposes into a disjoint union of spaces $\Mm^{a,b,c}(\sSO(n,n+1))$ where the isomorphism class of a Higgs bundle $(V,W,\eta)$ lies in $\Mm^{a,b,c}(\sSO(n,n+1))$ if and only if the first Stiefel-Whitney classes of $V$ and $W$ are given by 
\[\xymatrix{a=sw_1(V)=sw_1(W)~,&b=sw_2(V)&\text{and}&c=sw_2(W)}~.\]
Thus, we have  $\Mm_d(\sSO(n,n+1))\subset\Mm^{0,0,d\ \text{mod}\ 2}(\sSO(n,n+1)).$ 
Moreover, if $(V,W,\eta)$ is a polystable $\sSO(n,n+1)$ Higgs bundle, then the corresponding representation $\rho\in\Xx(\sSO(n,n+1))$ lifts to the split real form $\sSpin(n,n+1)\subset\sSpin(2n+1,\C)$ if and only the second Stiefel-Whitney classes of $V$ and $W$ are the same. 
\begin{Corollary}
A representation in the component $\Xx_d(\sSO(n,n+1))$ lifts to $\sSpin(n,n+1)$ if and only if $d=0\ mod\ 2.$
\end{Corollary}

\section{The singular components $\Mm_0(\sSO(n,n+1))$ and $\Mm_{sw_1}^{sw_2}(\sSO(n,n+1))$}\label{singular components Section}
We now show the components of $\Mm(\sSO(1,2)) $ from Theorems \ref{THM M0 homotopy type} and \ref{THM Msw1 SO(1,2)} also generalize to $\Mm(\sSO(n,n+1)$. These components are more difficult to describe because they are singular. 

Consider the space $\widetilde\Ff_0$ defined by 
 	\begin{equation}\label{EQ tildeF0 DEF}
 	\xymatrix@=.5em{\widetilde\Ff_0=\{(M,\mu,\nu)\ |\ M\in\Pic^0(X),\ \mu\in H^0(M^{-1}K^n),\ \nu\in H^0(MK^n)\ \}~.}
 	\end{equation}
The group $\sO(2,\C)$ is isomorphic to the group of $2\times2$ matrices generated by $\smtrx{\lambda&0\\0&\lambda^{-1}}$ and $\smtrx{0&1\\1&0}$ for $\lambda\in\C^*.$ There is a natural action of $\sO(2,\C)$ on $\widetilde\Ff_0$ given by:
\begin{equation}\label{EQ O2 action on F0}
	\smtrx{\lambda&\\&\lambda^{-1}}\cdot(M,\mu,\nu)=(M,\lambda^{-1}\mu,\lambda\nu)\ \ \ \ \ \ \text{and}\ \ \ \ \ \smtrx{0&1\\1&0}\cdot(M,\mu,\nu)=
	(M^{-1},\nu,\mu)~.
\end{equation}


\begin{Theorem}\label{THM2}
Let $\Gamma$ be the fundamental group of a closed surface $S$ of genus $g\geq2$ and let $\Xx(\sSO(n,n+1))$ be the $\sSO(n,n+1)$-character variety of $\Gamma$. For each $n\geq2,$ there is a connected component $\Xx_0(\sSO(n,n+1))$ of $\Xx(\sSO(n,n+1))$ which does not contain representations with compact Zariski closure. Furthermore, for each Riemann surface structure on $S,$ the space $\Xx_0(\sSO(n,n+1))$ is homeomorphic to
\[\Xx_0(\sSO(n,n+1))\cong\Ff_0\times\bigoplus\limits_{j=1}^{n-1}H^0(K^{2j})~,\]
where $\Ff_0$ is the GIT quotient $\widetilde\Ff_0\big\slash\big\slash\sO(2,\C)$ of the $\sO(2,\C)$-space $\widetilde\Ff_0$ from \eqref{EQ tildeF0 DEF} and $H^0(K^{2j})$ is the vector space of holomorphic differentials of degree $2j$. 
\end{Theorem}
\begin{Corollary}
	Since $\Ff_0$ deformation retracts onto $\Pic^0(X)/\Z_2,$ the connected component $\Xx_0(\sSO(n,n+1))$ is homotopic to the quotient of $(S^1)^{2g}$ by the $\Z_2$ action given by inversion. In particular, its rational cohomology is given by 
	\[H^j(\Xx_0(\sSO(n,n+1)),\Q)=\begin{dcases}
	H^{j}((S^1)^{2g},\Q)& \text{if\ $j$\ is\ even~}\\
	0&\text{otherwise}~.
\end{dcases}\]
\end{Corollary}


For $sw_1\in H^1(S,\Z_2)\setminus\{0\}$, let $X_{sw_1}\to X$ be the associated orientation double cover. Denote the covering involution by $\iota$ and set 
\[\Prym(X_{sw_1},X)=\{M\in\Pic^0(X_{sw_1})\ |\ \iota^* M=M^{-1}\}~.\]
As in \eqref{EQ prym decomp}, $\Prym(X_{sw_1},X)$ has two connected components $\Prym^{sw_2}(X_{sw_1},X)$ which are labeled by an invariant $sw_2\in H^2(X,\Z_2)$. 
 
Consider the following space 
\begin{equation}
	\label{eq: Fsw1sw2}\Ff_{sw_1}^{sw_2}=\{(M,\mu): M\in\Prym^{sw_2}(X_{sw_1},X)\ \text{and}\ \mu\in H^0(M^{-1}K_{X_{sw_1}}^n)\}~.
\end{equation}
For $n>1,$ $H^0(M^{-1}K^n_{X_sw_1})= \C^{2(2n-1)(g-1)};$ thus, $\Ff_{sw_1}^{sw_2}$ is a vector bundle. 
The group $\Z_2\oplus\Z_2$ acts on $\widetilde\Ff_{sw_1}^{sw_2}$ by
\begin{equation}
	\label{eq Z2 action on Fsw1}
	\xymatrix{(M,\mu)\to(M,-\mu)&\text{and}&(M,\mu)\to(\iota^*M,\iota^*\mu)}~.
\end{equation}
The quotient space $\Ff_{sw_1}^{sw_2}/\Z_2\oplus\Z_2$ is an orbifold. Here the orbifold points correspond to pairs $(M,\mu)$ with $M= M^{-1}$ and $\iota^*\mu=\pm\mu.$
\begin{Theorem}\label{THM3}
Let $\Gamma$ be the fundamental group of a closed surface $S$ of genus $g\geq2$ and let $\Xx(\sSO(n,n+1))$ be the $\sSO(n,n+1)$-character variety of $\Gamma$. For each $n\geq2$ and each $(sw_1,sw_2)\in (H^1(X,\Z_2)\setminus\{0\})\times H^2(X,\Z_2)$, there is a connected component $\Xx^{sw_2}_{sw_1}(\sSO(n,n+1))$ of $\Xx(\sSO(n,n+1))$ which does not contain representations with compact Zariski closure. Furthermore, for each Riemann surface structure $X$ on $S,$ the space $\Xx_{sw_1}^{sw_2}(\sSO(n,n+1))$ is a smooth orbifold diffeomorphic to
		\[\Ff^{sw_2}_{sw_1}/(\Z_2\oplus\Z_2)\times \bigoplus\limits_{j=1}^{n-1}H^0(K^{2j}_X)\] where $\Ff^{sw_2}_{sw_1}\ra \Prym^{sw_2}(X_{sw_1},X)$ is the rank $(4n-2)(2g-2)$ vector bundle from \eqref{eq: Fsw1sw2}, and the $\Z_2\oplus\Z_2$ action is given by \eqref{eq Z2 action on Fsw1}
\end{Theorem}
 Recall from Remark \ref{Remark prym variety torsor} that $\Prym^{sw_2}(X_{sw_1},X)$ is topologically a $(2g-2)$-dimensional torus. 
\begin{Corollary}
	Since $\Ff_{sw_1}^{sw_2}/\Z_2\oplus\Z_2$ deformation retracts onto $\Prym^{sw_2}(X)/\Z_2,$ the connected component $\Xx_{sw_1}^{sw_2}(\sSO(n,n+1))$ is homotopic to the quotient of $(S^1)^{2g-2}$ by the $\Z_2$ action given by inversion. In particular, its rational cohomology is 
	\[H^j(\Xx_{sw_1}^{sw_2}(\sSO_0(n,n+1)),\Q)=\begin{dcases}
	H^{j}((S^1)^{2g-2},\Q)& \text{if\ $j$\ is\ even~}\\
	0&\text{otherwise}
\end{dcases}~.\]
\end{Corollary}
\begin{Remark}
	For $n=2,$ these results were proven in \cite{SO23LabourieConj}. In the $n=2$ case, the extra invariants arise from topological invariants of the Cayley partner of a maximal $\sSO_0(2,3)$ Higgs bundle. Since $\sSO(n,n+1)$ is not a group of Hermitian type for $n\geq3,$ the proofs of the above theorems require a more technical analysis.
\end{Remark}
\begin{Corollary}\label{COR: connected comp}
	For $n\geq 3,$ the character variety $\Xx(\sSO(n,n+1))$ of a genus $g\geq 2$ closed surface has at least $2^{2g+2}+2^{2g+1}-1+n(2g-2)$ connected components.
\end{Corollary}
\begin{proof}
	The topological invariants of a flat $\sSO(n,n+1)$ bundle are a first Stiefel-Whitney class and two second Stiefel-Whitney classes. This gives $2^{2g+2}$ topological invariants. 
	For each value of these invariants, there is a connected component of the character variety $\Xx(\sSO(n,n+1))$ which contains representations $\rho:\Gamma\ra\sSO(n,n+1)$ with compact Zariski closure \cite{Oliveira_GarciaPrada_2016}. 
	The $n(2g-2)$ components from Theorem \ref{THM1}, the connected component from Theorem \ref{THM2} and the $2(2^{2g}-1)$ components from Theorem \ref{THM3} do not contain any representations with compact Zariski closures.
	This gives $2^{2g+2}+2^{2g+1}-1+n(2g-2)$ connected components.
\end{proof}
\begin{Remark}
	In \cite{SOpqStabilityAndMinima}, the connected components of $\Xx(\sSO(n,m))$ are computed. In particular, we show there are exactly $2^{2g+2}+2^{2g+1}-1+n(2g-2)$ connected components of the character variety $\Xx(\sSO(n,n+1)).$ 
\end{Remark}
\subsection{Proof of Theorem \ref{THM2}}
To prove Theorem \ref{THM2} we start with a sequence of lemmas.

\begin{Lemma}
	\label{Lemma Closed O2 orbits} Consider the space $\widetilde\Ff_0$ from \eqref{EQ tildeF0 DEF} with the $\sO(2,\C)$ action given by \eqref{EQ O2 action on F0}. Define the subspace $\widetilde\Ff_0^{ps}\subset \widetilde\Ff_0$ by
	\begin{equation}\label{EQ F0ps}
		\xymatrix{\widetilde\Ff^{ps}_0=\{(M,\mu,\nu)\in\widetilde\Ff_0\ |\ \mu=0\ \text{if and only if\ }\nu=0 \}}~.
	\end{equation}
	For each $x\in\widetilde\Ff_0$, the orbit $\sO(2,\C)\cdot x$ is closed if and only if $x\in\widetilde\Ff_0^{ps}.$ In particular,
	\[\Ff_0= \widetilde\Ff_0//\sO(2,\C) = \widetilde\Ff_0^{ps}/\sO(2,\C)~.\]
\end{Lemma}
\begin{proof}
	If $x\in \widetilde\Ff_0\setminus\widetilde\Ff_0^{ps}$, then either $\mu=0$ or $\nu=0$ but not both. Suppose $\mu\neq0$ and $\nu=0.$ The orbit through $x$ is not closed since, for $\lambda\in \R^>0$, we have 
	\[\lim\limits_{\lambda\to0}\smtrx{\lambda&0\\0&\lambda^{-1}}\cdot(M,\mu,0)=(M,0,0)~.\]
	It is clear that the orbit through a point $(M,0,0)\in\widetilde\Ff_0^{ps}$ is closed. Similarly, if $(M,\mu,\nu)\in\widetilde\Ff_0^{ps}$ with $\mu\neq0$ and $\nu\neq0$, then the $\sO(2,\C)$-orbit through $(M,\mu,\nu)$ is closed. 
\end{proof}
It is straight forward to compute the $\sO(2,\C)$-stabilizers of points of $\widetilde\Ff_0^{ps}.$
\begin{Lemma}\label{Lemma O2 stabilizer}
	Let $\widetilde \Ff_0^{ps}$ be the space from \eqref{EQ F0ps}, the $\sO(2,\C)$-stabilizer of a point $(M,\mu,\nu)\in\widetilde\Ff_0^{ps}$ is
	\[\Stab_{\sO(2,\C)}(M,\mu,\nu)=\begin{dcases}
\sO(2,\C)&\text{if}\ M=M^{-1}\ \text{and\ } \mu=\nu=0\\
\sSO(2,\C)& \text{if}\ M\neq M^{-1}\ \text{and\ } \mu=\nu=0\\
\Z_2 & \text{if}\ M=M^{-1}\,\ \mu\neq0\ \text{and\ } \mu=\lambda\nu\ \text{for\ } \lambda\in\C^*\\
\{1\} & \text{otherwise}
	\end{dcases}~.\]
\end{Lemma}

For each point of $(M,\mu,\nu,q_2,\cdots q_{2p-2})\in \widetilde\Ff_0\times\bigoplus\limits_{j=1}^{n-1}H^0(K^{2j})$ define the $\sSO(n,n+1)$-Higgs bundle 
 $\widetilde\Psi_0(M,\mu,\nu,q_2,\cdots q_{2p-2})=(V,W,\eta)$  by 
 	\[\xymatrix@=.5em{V=K^{n-1}\oplus K^{n-3}\oplus\cdots K^{3-n}\oplus K^{1-n}~,\\W=M\oplus K^{n-2}\oplus K^{n-4}\oplus\cdots\oplus K^{4-n}\oplus K^{2-n}\oplus M^{-1}~,}\]
 	\begin{equation}
 		\label{EQ M0 Higgs field}
 		\eta=\mtrx{0&0&0&\cdots&0&\nu\\
			   1&q_2&q_4&\cdots&q_{2n-4}&q_{2n-2}\\
			   0&1&q_2&q_4&\cdots&q_{2n-4}\\
			   0&0&1&q_2&\cdots&q_{2n-6}\\
			   &&\ddots&\ddots&\ddots&\\
			   &&&0&1&q_2\\
			   &&&&0&\mu}:V\longrightarrow W\otimes K~.
 	\end{equation} 

\begin{Lemma}\label{Lemma Gauge group fixing im Psi0}
For two points $(M,\mu,\nu,q_2,\cdots,q_{2n-2})$ and $(M',\mu',\nu',q_2',\cdots,q_{2n-2}')$ in $\widetilde\Ff_0\times\bigoplus\limits_{j=1}^{n-1}H^0(K^{2j})$ the associated $\sSO(n,n+1)$ Higgs bundles from \eqref{EQ M0 Higgs field} are gauge equivalent if and only if $q_{2j}=q_{2j}'$ for all $j$ and $(M,\mu,\nu)$ and $(M',\mu',\nu')$ are in the same $\sO(2,\C)$-orbit.
 \end{Lemma}
\begin{proof}
	As in the proof of Lemma \ref{Lemma Image in Md}, for $(M,\mu,\nu,q_2,\cdots,q_{2n-2})\in\widetilde\Ff_0\times\bigoplus\limits_{j=1}^{n-1}H^0(K^{2j})$, let $(V,W,\eta)$ be the $\sSO(n,n+1)$ Higgs bundle defined by $\widetilde\Psi_0(M,\mu,\nu,q_2,\cdots,q_{2n-2})$. It is given by \eqref{EQ M0 Higgs field}. 
 	Write 
 	\begin{equation}\label{EQ W decomp 0}
 		\xymatrix{W=M\oplus W_0\oplus M^{-1}&\text{for}&W_0=K^{n-2}\oplus\cdots\oplus K^{2-n}}.
 	\end{equation}
 With respect to the decomposition \eqref{EQ W decomp 0}, a gauge transformation $g_W$ of $W$ and the Higgs field decomposes as
 	\begin{equation}\label{EQ decomposition gW}
 		g_W=\mtrx{b_0&A&c_0\\
 				B&g_{W_0}&C\\
 				b_n&D&c_n}\ \ \ \ \ \ \ \ \ \ \text{and}\ \ \ \ \ \ \ \ \ \eta=\mtrx{\alpha\\\eta_0\\\beta}~.
 	\end{equation}

 Recall that in the proof of Lemma \ref{Lemma Image in Md}, we only used the positivity assumption on the degree of the line bundle $M$ to show that $c_0$ and $b_n$ where zero. Thus, using the same arguments as the proof of Lemma \ref{Lemma Image in Md}, we have 
\begin{equation}
 	\label{EQ d=0 gauge Trans}
 	g_W=\mtrx{b_0&0&c_0\\
 				0& g_{W_0}&0\\
 				b_n&0&c_n} ~.
 \end{equation} 
Using \eqref{EQ b0bn=0}, we have $b_0b_n=0$, $c_0cn=0$, $c_0b_n+b_0c_n=1$.  Thus, either $b_n=0,$ $c_0=0$ and $b_0=c_n^{-1}$ or $b_0=0,$ $c_n=0$ and $b_n=c_0^{-1}$.
Furthermore, by Hitchin's parameterization of $\Hit(\sSO(n,n-1))$, we have 
\[\xymatrix{(g_{W_0},g_V)=(Id_{W_0},Id_V)&\text{or}&(g_{W_0},g_V)=(-Id_{W_0},-Id_V)}~.\] 
However, since $det(g_W)\cdot det(g_V)=1,$ we must have 
\begin{itemize}
	\item $g_V=Id_V$ and $g_{W_0}=Id_{W_0}$ if $b_n=0,c_0=0$ and $b_0=c_n^{-1}$,
	\item $g_V=Id_V$ and $g_{W_0}=Id_{W_0}$ if n is odd and $b_0=0,c_n=0$ and $b_n=c_0^{-1}$,
	\item $g_V=-Id_V$ and $g_{W_0}=-Id_{W_0}$ if n is even and $b_0=0,c_n=0$ and $b_n=c_0^{-1}$.
\end{itemize}

To conclude, the subgroup $\Gg_{\Psi_0}$ of the $\sS(\sO(n,\C)\times\sO(n+1,\C))$ gauge group which preserves the image of $\widetilde\Psi_0$ is given by gauge transformations $(g_W,g_V)$ of the form
\[
	\left(\mtrx{\lambda&0&0\\0&Id_{W_0}&0\\0&0&\lambda^{-1}} ,Id_V\right) \ \ \text{and}\ \ \left(\mtrx{0&0&\lambda\\0&(-1)^{n+1}Id_{W_0}&0\\\lambda^{-1}&0&0},(-1)^{n+1}Id_V\right)
~,\]
where $\lambda\in\C^*.$ In both cases, the group $\Gg_{\Psi_0}$ is isomorphic to $\sO(2,\C).$
\end{proof}
Define an $\sO(2,\C)$ action on $\widetilde\Ff_0\times \bigoplus\limits_{j=1}^{n-1}H^0(K^{2j})$ by trivially extending the action on $\widetilde\Ff_0$ from \eqref{EQ O2 action on F0}. 
We will now show that the Higgs bundle $\widetilde\Psi_0(x)$ is polystable if and only if the $x\in\widetilde\Ff_0^{ps}.$

\begin{Lemma}\label{Lemma Polystable image of Psi0}
For $x\in\widetilde\Ff_0\times\bigoplus\limits_{j=1}^{n-1} H^0(K^{2j}),$ the Higgs bundle $\widetilde\Psi_0(x)$ from \eqref{EQ M0 Higgs field} is polystable if and only if $x\in \widetilde\Ff_0^{ps}\times\bigoplus\limits_{j=1}^{n-1} H^0(K^{2j})$. The Higgs bundle $\widetilde\Psi_0(x)$ is stable if and only if $x\in\widetilde\Ff_0^s\times\bigoplus\limits_{j=1}^{n-1} H^0(K^{2j})$, where $\widetilde\Ff^{s}_0=\{(M,\mu,\nu)\in\widetilde\Ff_0^{ps}\ |\ \mu\neq0~\text{and}~\nu\neq0 \}.$
\end{Lemma}
\begin{proof}
Let $(V,W,\eta)$ be the Higgs bundle $\widetilde\Psi_0(M,\mu,\nu,q_2,\cdots,q_{2n-2})$, it is given by \eqref{EQ M0 Higgs field}.
First note that if $\mu=0$, then $M$ is an $\Phi$-invariant degree zero subbundle of the associated $\sSL(2n+1,\C)$ Higgs bundle $(E,\Phi)=\left(V\oplus W,\smtrx{0&\eta^*\\\eta&0}\right)$. 
Similarly, if $\nu=0,$ then $M^{-1}$ is an degree zero $\Phi$-invariant subbundle. If $x\in(\widetilde \Ff_0\setminus\widetilde \Ff_0^{ps})\times\bigoplus\limits_{j=1}^{n-1} H^0(K^{2j})$, then the Higgs bundle $(E,\Phi)$ has a degree zero invariant subbundle but is not polystable. Thus, the associated Higgs bundle $\widetilde\Psi_0(x)$ is not polystable. 

For $x\in\widetilde\Ff_0^{ps}\setminus\widetilde\Ff^{s}\times\bigoplus\limits_{j=1}^{n-1} H^0(K^{2j})$, the $\sSL(2n+1,\C)$ Higgs bundle $(E,\Phi)$ associated to $\widetilde\Psi_0(x)$ is the direct sum of the polystable $\sSL(2,\C)$ Higgs bundle $(M\oplus M^{-1},\Phi=0)$ and a Higgs bundle in $\Hit(\sSO(n-1,n))$. Thus, $\widetilde\Psi_0(x)$ is polystable. By Lemma \ref{Lemma O2 stabilizer}, the automorphism group of such a Higgs bundle is not finite, hence, $\widetilde\Psi_0(x)$ is polystable but not stable. 

The rest of the proof is similar to Lemmas \ref{Lemma defining Higgs bundles in Md} and \ref{Lemma Image in Md}.  The $\sSL(2n+1,\C)$ Higgs bundle associated to $x=(M,\mu,\nu,0,\cdots,0)\in\widetilde\Ff_0^{s}\times\bigoplus\limits_{j=1}^{n-1}H^0(K^{2j})$ can be written schematically as 
	\begin{equation}
		\label{EQ cyclic d=0 higgs bundle}
		\xymatrix@R=1em{K^{n-1}\ar[r]^1&K^{n-2}\ar[r]^1&\cdots\ar[r]^1&K^{2-n}\ar[r]&K^{1-n}\ar[dll]^{(\nu\ \ \mu)^T}\\&&M\oplus M^{-1}\ar[ull]^{(\mu\ \ \nu)}&&}
	\end{equation}
Such a Higgs bundle is not fixed by the $\C^*$-action, but is fixed by the subgroup of $2n^{th}$-roots of unity and the above summands are each eigen-bundles of a holomorphic gauge transformation. For such {\em cyclic} Higgs bundles, checking polystability reduces to checking for destabilizing subbundles in each bundle in the chain (see Proposition 6.3 \cite{KatzMiddleInvCyclicHiggs}). 
	Since none of the line bundles in the chain are invariant, it suffices to check $(M\oplus M^{-1}).$
	As $M$ and $M^{-1}$ both have degree zero, $M\oplus M^{-1}$ has no positive degree subbundles. 
	
	Note that the only isotropic subbundles of $M\oplus M^{-1}$ are the summands $M$ and $M^{-1}$. Since neither of these summands are invariant, if $N\subset M\oplus M^{-1}$ is a degree zero invariant line subbundle, then $N$ is an orthogonal subbundle. Thus, we can take its orthogonal complement and split the Higgs bundle as a stable $\sSL(2n,\C)$ Higgs bundle plus an invariant degree zero line bundle. 
	This implies the Higgs bundle \eqref{EQ cyclic d=0 higgs bundle} above is polystable. 
	Moreover, by Lemmas \ref{Lemma O2 stabilizer} and \ref{Lemma: s is open closure is ps} the the stabilizer of $\widetilde\Psi_0(x)$ is finite. Hence, by Definition \ref{DEF Stability of G Higgs bundle}, the $\sSO(n,n+1)$ Higgs bundle $\Psi_0(x)$ is stable.

As in the proof of Lemma \ref{Lemma defining Higgs bundles in Md}, stability is an open condition. Thus, there is an open neighborhood $U$ of $(\mu,\nu,0,\cdots,0)$ such that the Higgs bundles \eqref{EQ M0 Higgs field} are stable for $(\mu,\nu,q_2,\cdots,q_{2n-2})\in U$. 
Using the gauge transformations \eqref{EQ scaling gauges}, the Higgs bundle $(V,W,\lambda\eta)$ is gauge equivalent to $\widetilde\Psi_0(M,\mu,\lambda^{2n}\nu,\lambda^2 q_2,\lambda^4q_4,\cdots,\lambda^{2n-2} q_{2n-2})~.$
Since stability is preserved by scaling the Higgs field, $\widetilde\Psi_0(\widetilde\Ff_0^s\times\bigoplus\limits_{j=1}^{n-1}H^0(K^{2j}))$ consists of stable $\sSO(n,n+1)$ Higgs bundles. 
\end{proof}
	
	
Putting together the above lemmas, we have the following proposition.
\begin{Proposition}
	The following spaces are homeomorphic
	\[\left(\widetilde\Ff_0\big\slash\big\slash \sO(2,\C)\right)\times\bigoplus\limits_{j=1}^{n-1}H^0(K^{2j})\ \cong\ \left(\widetilde\Ff^{ps}_0\big\slash \sO(2,\C)\right)\times\bigoplus\limits_{j=1}^{n-1}H^0(K^{2j})\ \cong\  \]
	\[\widetilde\Psi_0\left(\widetilde\Ff^{ps}_0\times\bigoplus\limits_{j=1}^{n-1}H^0(K^{2j})\right)\big\slash\sO(2,\C)\ \cong\ \widetilde\Psi_0\left(\widetilde\Ff^{ps}_0\times\bigoplus\limits_{j=1}^{n-1}H^0(K^{2j})\right)\big\slash\big\slash\sO(2,\C)\ ~. \]
In particular, if $\widetilde\Ff^{ps}_0\big\slash \sO(2,\C)=\Ff^{ps}_0,$ then we have a continuous map 
\begin{equation}
	\label{EQ Psi0 map}\Psi_0:\xymatrix{\Ff^{ps}_0\times\bigoplus\limits_{j=1}^{n-1}H^0(K^{2j})\ar[r]&\Mm(\sSO(n,n+1))}~
	\end{equation}
which is a homeomorphism onto its image. 
\end{Proposition}

We will show that the image of $\Psi_0$ is open and closed. However, since the image of $\Psi_0$ is singular, this it is more complicated than the proof of Theorem \ref{THM1}. We start by analyzing the stable locus.
\begin{Lemma}\label{Lemma: s is open closure is ps}
	Consider the map $\Psi_0$ from \eqref{EQ Psi0 map} and the spaces $\widetilde\Ff_0^s\subset\widetilde\Ff_0^{ps}$ from Lemma \ref{Lemma Polystable image of Psi0}. Denote the quotients of these spaces by $\sO(2,\C)$ by $\Ff_0^s$ and $\Ff_0^{ps}$ respectively. The image $\Psi_0\big(\Ff^{s}_0\times\bigoplus\limits_{j=1}^{n-1}H^0(K^{2j})\big)$ is open in $\Mm(\sSO(n,n+1)),$ and the closure of the image is given by
	\[\overline{\Psi_0\big(\Ff^{s}_0\times\bigoplus\limits_{j=1}^{n-1}H^0(K^{2j})\big)}=\Psi_0\big(\Ff^{ps}_0\times\bigoplus\limits_{j=1}^{n-1}H^0(K^{2j})\big)~.\]
	\end{Lemma}
\begin{proof}
	 As in \eqref{EQ Hitchin fibration on Md}, we can choose a basis of invariant polynomials $(p_1,\cdots,p_n)$ so that 
	\[p_j(\Psi_0(M,\mu,\nu,q_2,\cdots,q_{2n-2}))=\begin{dcases}
		q_{2j}& 1\leq j\leq n-1\\
		\mu\otimes\nu&j=n
	\end{dcases}~.
	%
	\]
	Let $h:\Mm(\sSO(n,n+1))\to\bigoplus\limits_{j=1}^nH^0(K^{2j})$ denote the Hitchin fibration.  
	Note that image $h\big(\Psi_0(\Ff^s_0\times\bigoplus\limits_{j=1}^{n-1}H^0(K^{2j}))\big)$ is $H^0(K^{2n})\setminus\{0\}\times \bigoplus\limits_{j=1}^{n-1}H^0(K^{2j}).$

    Since all of the Higgs bundles in the image are stable, the image is a smooth orbifold. Moreover, the dimension of $\Psi_0\big(\Ff_0^s\times\bigoplus\limits_{j=1}^{n-1}H^0(K^{2j})\big)$ is the expected dimension of the moduli space $\Mm(\sSO(n,n+1))$.
    If $U$ is a sufficiently small neighborhood of $x\in\Ff_0^s\times\bigoplus\limits_{j=1}^{n-1}H^0(K^{2j})$, then $h(\Psi_0(U))$ is clearly open in the Hitchin base. Since $\Psi_0(x)$ is either a smooth point or an orbifold point, $h^{-1}(h(\Psi_0(U)))=\Psi_0(U)$. Thus, the image $\Psi_0\big(\Ff^{s}_0\times\bigoplus\limits_{j=1}^{n-1}H^0(K^{2j})\big)$ is open.

	As in the proof of closedness for Theorem \ref{THM1}, to compute the closure of the image we use the properness of the Hitchin fibration.  
	Let $x_i=(y_i,q_2^i,\cdots,q_{2n-2}^i)$ be a sequence in $\Ff_0^s\times\bigoplus\limits_{j=1}^{n-1}H^0(K^{2j})$ which diverges and such that the image $\Psi_0(x_j)$ converges to $\Mm(\sSO(n,n+1)).$ 
	By the properness of the Hitchin fibration, the sequence of differentials $q_{2j}^i$ must converge for all $j$ and $\lim\limits_{i\to\infty}y_i\in\Ff_0^{ps}\setminus\Ff_0^s.$ Thus, the closure of $\Psi_0\big(\Ff_0^s\times \bigoplus\limits_{j=1}^{n-1}H^0(K^{2j})\big)$ in $\Mm(\sSO(n,n+1))$ is $\Psi_0\big(\Ff_0^{ps}\times \bigoplus\limits_{j=1}^{n-1}H^0(K^{2j})\big)$.
\end{proof}

To show that the image of $\Psi_0$ defines a connected component of the moduli space $\Mm(\sSO(n,n+1)),$ it remains to show that $\Psi_0\big(\Ff^{ps}_0\times\bigoplus\limits_{j=1}^{n-1}H^0(K^{2j})\big)$ is open (see Lemma \ref{Lemma image psi0 closed}). To do this, the local structure of points in the boundary of the closure from Lemma \ref{Lemma: s is open closure is ps} must be examined. We will show that a local neighborhood of such a point in $\Mm(\sSO(n,n+1))$ is homeomorphic the corresponding open neighborhood in $\Ff_0^{ps}\times\bigoplus\limits_{j=1}^{n-1}H^0(K^{2j})$. This amounts to studying the complex \eqref{Eq hypercohomology}.
\begin{Remark}\label{Rem C* open}
	Let $x=([M,0,0],q_2,\cdots,q_{2n-2})$ be a point in $\Ff_0^{ps}\setminus\Ff^s_0\times\bigoplus\limits_{j=1}^{n-1}H^0(K^{2j})$. Using the $\C^*$ action on the Higgs bundle moduli space, a sufficiently small open neighborhood of $\Psi_0(x)$ can be brought into an open of $\Psi_0([M,0,0],0,\cdots,0).$ Thus, it suffices to prove that an open neighborhood of $\Psi_0([M,0,0],0,\cdots,0)$ is homeomorphic to an open neighborhood of $([M,0,0],0,\cdots,0)$ in $\Ff_0^{ps}\times\bigoplus\limits_{j=1}^{n-1}H^0(K^{2j}).$ 
\end{Remark}
 For $M\in\Pic^0(X),$ consider the $\sSO(n,n+1)$ Higgs bundle $(V,W,\eta)$ given by 
\[(V,Q_V)=\left( K^{n-1}\oplus K^{n-3}\oplus\cdots\oplus K^{3-n}\oplus K^{1-n},\ \smtrx{&&1\\&\iddots&\\1&&}\right)~,\]
\[ (W, Q_W)=\left(M\oplus K^{n-2}\oplus K^{n-4}\oplus\cdots\oplus K^{2-n}\oplus M^{-1},\ \smtrx{&&1\\&\iddots&\\1&&}\right)\]
and $\eta$ given by \eqref{EQ M0 Higgs field} with $q_2=\cdots=q_{2n-2}=\mu=\nu=0.$ Similar to \eqref{EQ Fixedpoint in Md}, the $\sSL(2n+1,\C)$ can be represented schematically by
\begin{equation}
	\label{EQ: Higgs field at singularity}
	\xymatrix@C=2em@R=-.3em{K^{n-1}\ar[r]_1&K^{n-2}\ar[r]_1&\cdots\ar[r]_1&\Oo\ar[r]_1&\cdots\ar[r]_1&K^{2-n}\ar[r]_1&K^{1-n}\\&&&\oplus&&&\\&&&M&&&\\&&&\oplus&&&\\&&&M^{-1}&&&}~.
\end{equation}

The Lie algebra bundle with fiber $\fso(n,\C)\oplus\fso(n+1,\C)$ consists of $Q_V$ and $Q_W$ skew symmetric endomorphisms of $V$ and $W$ respectively, we will use the notation  
\[\Lambda_Q^2V\oplus\Lambda_Q^2W\subset\End(V)\oplus\End(W)~.\]

Write the line bundle decompositions of $V$ and $W$ from \eqref{EQ M0 Higgs field} as follows
\[V= V_{1-n}\oplus V_{3-n}\oplus\cdots\oplus V_{n-3}\oplus V_{n-1}~,\]
\[W=W_{2-n}\oplus W_{4-n}\oplus\cdots\oplus W_0\oplus\cdots W_{n-4}\oplus W_{n-2}~,\]
where $V_j=K^{-j},$ $W_j=K^{-j}$ if $j\neq 0$ and $W_0=M\oplus M^{-1}$ if $n$ is odd and $W_0=M\oplus\Oo\oplus M^{-1}$ if $n$ is even. In terms of the above splittings, sections of $\Lambda^2_QV$ consist of $n\times n$ matrices which are antisymmetric with respect to reflecting about the anti-diagonal.  This gives a grading 
\[\Lambda^2_QV\oplus\Lambda^2_QW\cong\bigoplus\limits_{k=4-2n}^{2n-4}(\Lambda^2_QV\oplus\Lambda^2_QW)_k~.\]  
One computes that $(\Lambda^2_QV)_{k}=0$ for $k$ odd and 
\begin{equation}\label{EQ Lambda2V2k}
	\xymatrix@=.4em{(\Lambda^2_QV)_{2k}\cong\bigoplus\limits_{j=0}^{\floor{\frac{n-k}{2}}-1} \Hom(V_{1-n+2j},V_{1-n+2j+2k})&\text{for}\ 0\leq 2k\leq 2n-4\\(\Lambda^2_QV)_{-2k}\cong(\Lambda^2_QV)^*_{2k}}~.
\end{equation}
 Similarly (but changing the indexing scheme), 
\begin{equation}
	\label{EQ Lambda2W2k}\xymatrix@=.4em{(\Lambda^2_QW)_{2k}\cong\bigoplus\limits_{j=\floor{\frac{n-k}{2}}}^{n-k-2} \Hom(W_{2-n+2j},W_{2-n+2j+2k})& \text{for}\ 0<2k\leq 2n-6\\(\Lambda^2_QW)_{-2k}\cong(\Lambda^2_QW)^*_{2k}}~.
\end{equation}
 For $k=0$, we have 
\begin{equation}
	\label{EQ Lambda2W0}
	(\Lambda^2_QW)_0\cong\Lambda^2W_0\oplus\bigoplus\limits_{j=\floor{\frac{n}{2}}}^{n-2}\Hom(W_{2-n+2j},W_{2-n+2j})~.
\end{equation}
For $n$ even, $(\Lambda^2_QW)_{2k-1}=0$ for all $k$. But when $n$ is odd, we have 
\begin{equation}
	\label{EQ Lambda2W2k+1}(\Lambda_Q^2W)^*_{1-2k}\cong(\Lambda_Q^2W)_{2k-1}\cong\begin{dcases}
		\Hom(W_{1-2k},W_0)&\text{for}\ 0< 2k-1\leq 2-n\\
		0&\text{otherwise}
	\end{dcases}
\end{equation}  

Similarly, the bundle $\Hom(V,W)\otimes K$ acquires a grading $\bigoplus\limits_{k=3-2n}^{2n-3}\Hom(V,W)_k\otimes K$, where 
\begin{equation}
	\label{EQ HomVW2k+1}
	\Hom(V,W)_{2k+1}\otimes K=\begin{dcases}
		\bigoplus\limits_{j=0}^{n-k-2}\Hom(V_{1-n+2j},W_{2-n+2k+2j})\otimes K&\text{for\ } k\geq0\\
		 \bigoplus\limits_{j=0}^{n+k-1}\Hom(V_{n-1-2j},W_{n-2j+2k})\otimes K&\text{for\ } k<0
	\end{dcases}~,
\end{equation}
and 
\begin{equation}
	\label{EQ HomVW2k}\Hom(V,W)_{2k}\otimes K=\begin{dcases}
	0&\text{if\ } n\ \text{is\ even}\\
	\Hom(V_{-2k},W_0)\otimes K&\text{if\ } n\ \text{is\ odd}
\end{dcases}~.
\end{equation}

Moreover, the Higgs field $\eta$ is a holomorphic section of $\Hom(V,W)_{+1}\otimes K.$ 
Thus, $ad_{\eta}$ maps $(\Lambda_Q^2V\oplus\Lambda_Q^2W)_k$ to $\Hom(V,W)_{k+1}\otimes K$, and we have a graded complex
\[C^\bullet_k=C^\bullet(V,W,\eta)_k:\xymatrix@R=.3em{(\Lambda_Q^2V\oplus\Lambda_Q^2W)_k\ar[r]^{ad_\eta}&\Hom(V,W)_{k+1}\otimes K\\(g_V,g_W)\ar@{|->}[r] & \eta \circ g_V-g_W\circ\eta}~.\] 
In the hypercohomology sequence from \eqref{Eq hypercohomology} we have 
\[\xymatrix@R=1em@C=1.8em{0\ar[r]&\HH^0(C^\bullet_k)\ar[r]&H^0((\Lambda_Q^2V\oplus\Lambda_Q^2W)_k)\ar[r]^{ ad_\eta\ \ \ \ }&H^0(\Hom(V,W)_{k+1}\otimes K)\\\ar[r]&\HH^1(C^\bullet_k)\ar[r]&H^1((\Lambda_Q^2V\oplus\Lambda_Q^2W)_k)\ar[r]^{ad_\eta \ \ }&H^1(\Hom(V,W)_{k+1}\otimes K)\\\ar[r]&\HH^2(C^\bullet_k)\ar[r]&0~~.}\]

\begin{Lemma}\label{Lemma H2=0}
	If $(V,W,\eta)$ is an $\sSO(n,n+1)$ Higgs bundle given by \eqref{EQ: Higgs field at singularity}, then the second hypercohomology groups $\HH^2(C^\bullet_k)$ in the above sequences vanish for all $k.$
\end{Lemma}
\begin{proof}
Recall that $V_j=K^{-j}$ for all $j$, $W_j=K^{-j}$ for $j\neq0$ and $W_0=M\oplus M^{-1}$ if $n$ is odd and $W_0=M\oplus M^{-1}\oplus\Oo$ if $n$ is even. 

If $k\leq -2$, then, using the decompositions \eqref{EQ HomVW2k+1} and \eqref{EQ HomVW2k}, the holomorphic bundle $\Hom(V,W)_{k+1}\otimes K$ is a direct sum of line bundles with degree at least $4g-4.$ Thus, $H^1(\Hom(V,W)_{k+1}\otimes K)=0,$ and so $\HH^2(C^\bullet_{k}(V,W,\eta))=0.$

For $k\geq -1,$ we will show that the map \[ad_\eta:H^1((\Lambda_Q^2V\oplus\Lambda_Q^2W)_k)\to H^1(\Hom(V,W)_{k+1}\otimes K)\] is surjective.
First assume $2k\geq 2.$ In this case, $(\Lambda_Q^2V\oplus\Lambda_Q^2W)_{2k}$ is given by \eqref{EQ Lambda2V2k} and \eqref{EQ Lambda2W2k} and $\Hom(V,W)_{2k+1}\otimes K$ is given by \eqref{EQ HomVW2k+1}. 
We claim that $ad_\eta$ defines an isomorphism between these two sheaves. Indeed, for $0\leq j\leq \floor{\frac{n-k}{2}}-1$ and each $\alpha\in\Hom(V_{1-n+2j},V_{1-n+2j+2k})$ we have 
\[\xymatrix{V_{1-n+2j}\ar[r]_{\alpha \ \ \ }&V_{1-n+2j+2k}\ar[r]_{1\ \ \ }&W_{2-n+2j+2k}\otimes K}~.\]
Similarly, for $\floor{\frac{n-k}{2}}\leq j\leq n-k-2$ and each $\beta\in \Hom(W_{2-n+2j},W_{2-n+2j+2k})$ we have 
\[\xymatrix{V_{1-n+2j}\ar[r]_{1\ \ \ }&W_{2-n+2j}\otimes K\ar[r]_{\beta\ \ }&W_{2-n+2j+2k}\otimes K}~.\]
Thus we have an isomorphism
\[\xymatrix{(\Lambda_Q^2V)_{2k}\oplus(\Lambda_Q^2W)_{2k}\ar[r]^{ad_\eta}&\Hom(V,W)_{2k+1}\otimes K}~.\]

If $n$ is even, we are done. If $n$ is odd, then we have $(\Lambda_Q^2V\oplus\Lambda_Q^2W)_{2k-1}\cong\Hom(W_{-2k+1},W_0)$ and $\Hom(V,W)_{2k}\otimes K\cong\Hom(V_{-2k},W_0)\otimes K$. The Higgs field again defines an isomorphism since, for any $\gamma\in\Hom(W_{-2k+1},W_0)$ we have
\[\xymatrix{V_{-2k}\ar[r]_{1\ \ \ \ \ }& W_{-2k+1}\otimes K\ar[r]_{\ \  \gamma}&W_0\otimes K}~.\]

 For $k=-1,$ first note that if $n$ is even, then $\Hom(V,W)_0\otimes K=0$. If $n$ is odd, then
 $(\Lambda_Q^2V\oplus\Lambda_Q^2W)_{-1}\cong\Hom(K^{-1},W_0)$ and $\Hom(V,W)_0\otimes K=\Hom(\Oo,W_0)\otimes K$. 
 Again, the Higgs field gives an isomorphism since, for any $\delta\in\Hom(K^{-1},W_0)$
 \[\xymatrix{\Oo\ar[r]_{1\ \ }& K^{-1}\ar[r]_{ \delta}&W_0}~.\]

 Finally, for $k=0$, $(\Lambda_Q^2V\oplus\Lambda_Q^2W)_{0}$ is given by \eqref{EQ Lambda2V2k} and \eqref{EQ Lambda2W0} and $\Hom(V,W)_1\otimes K$ is given by \eqref{EQ HomVW2k+1}. Recall that
	\[\Lambda^2W_0\cong \begin{dcases}
		\Hom(M,M)& \text{if}\  n\ \text{is odd}\\
        \Hom(\Oo,M)\oplus \Hom(\Oo,M^{-1})\oplus \Hom(M,M)&\text{if}\ n\ \text{is even}
	\end{dcases}~,\]
First note that $\Hom(M,M)$ is in the kernel  of $(\Lambda_Q^2V\oplus\Lambda_Q^2W)_{0}\to\Hom(V,W)_1\otimes K$ since $M\oplus M^{-1}$ is invariant by the Higgs field. We claim that the map is surjective and the kernel is exactly the summand $\Hom(M,M).$ In particular, the induced map on $H^1$ is surjective. 
The Higgs field defines isomorphisms
\[\begin{dcases}
	\Hom(V_{1-n+2j},V_{1-n+2j})\cong\Hom(V_{1-n+2j},W_{2-n+2j})\otimes K&2-n+2j<0\\
	\Hom(W_{2-n+2j},W_{2-n+2j})\cong \Hom(V_{1-n+2j},W_{2-n+2j})\otimes K&\floor{\frac{n}{2}}\leq j\leq n-2
\end{dcases}\] 
since for each $\epsilon\in \Hom(V_{1-n+2j},V_{1-n+2j})$ and $\epsilon'\in\Hom(W_{n-2-2j},W_{n-2-2j})$
\[\xymatrix@R=1em{V_{1-n+2j}\ar[r]_\epsilon&V_{1-n+2j}\ar[r]_{1\ \ \ }&W_{2-n+2j}\otimes K&\text{and}\\V_{1-n+2j}\ar[r]_{1\ \ \ }&
W_{2-n+2j}\otimes K\ar[r]_{\epsilon'\ \ \ }&W_{2-n+2j}\otimes K~. }\]
If $n$ is odd, we are done, but if $n$ is even we need to consider the map 
\[\Hom(V_{-1},V_{-1})\oplus\Lambda^2W_0\to\Hom(V_{-1},W_0)\otimes K~.\]
In this case, 
$\Lambda^2W_0\cong\Hom(\Oo,M)\oplus\Hom(\Oo,M^{-1})\oplus\Hom(M,M)$
and
\[\Hom(V_{-1},W_0)\otimes K=(\Hom(V_{-1},\Oo)\oplus\Hom(V_{-1},M)\oplus\Hom(V_{-1},M^{-1}))\otimes K~.\]
As above, the Higgs field defines isomorphisms 
\[\xymatrix@R=.3em@C=.5em{\Hom(V_{-1},V_{-1})\cong\Hom(V_{-1},\Oo)\otimes K~,&\Hom(\Oo,M)\cong \Hom(V_{-1},M)\otimes K\\\text{and}&\Hom(\Oo,M^{-1})\cong\Hom(V_{-1},M^{-1})\otimes K~.}\]
In particular, $ad_\eta:H^1((\Lambda_Q^2V\oplus\Lambda_Q^2W)_0)\to H^1(\Hom(V,W)_1\otimes K)$ is surjective.
\end{proof} 


\begin{Lemma}
	\label{Lemma HH1 decomp} The hypercohomology group $\HH^1(C_k^\bullet(V,W,\eta))$ for a Higgs bundle of type \eqref{EQ: Higgs field at singularity} satisfy the following properties:
	\begin{itemize}
		\item $\HH^1(C_k^\bullet(V,W,\eta))\equiv0$ for $k>0,$
		\item $\HH^1(C_0^\bullet(V,W,\eta))	\cong H^1(\Hom(M,M))$,
		\item for $k<0$ and $k\neq -n$, $\HH^1(C^\bullet_{k})\cong \begin{dcases}
		H^0(K^{k})&\text{if}\ k\ {is\ even}\\
		0&\text{if}\ k\ {is\ odd}
		\end{dcases}~,$
		\item $\HH^1(C_{-n}^\bullet(V,W,\eta))\cong\begin{dcases}
		H^0(MK^n)\oplus H^0(M^{-1}K^n)\oplus H^0(K^n)& \text{if}\ n\ {is\ even}\\
			H^0(MK^n)\oplus H^0(M^{-1}K^n)& \text{if}\ n\ {is\ odd}
		\end{dcases}~.$
	\end{itemize}
\end{Lemma}
\begin{proof}
	In the proof of Lemma \ref{Lemma H2=0} it was shown that the map 
	\[ad_\eta:(\Lambda_Q^2V\oplus\Lambda_Q^2W)_k\to\Hom(V,W)_{k+1}\otimes K\]
	is an isomorphism for $k>0.$ Thus, $\HH^1(C_k^\bullet(V,W,\eta))\equiv0$ for $k>0$.

	Also in the proof of Lemma \ref{Lemma H2=0}, it was shown that $(\Lambda_Q^2V\oplus\Lambda_Q^2W)_0\cong\Hom(M,M)\oplus \mathbb \Hom(V,W)_1\otimes K$ and that the map $ad_\eta$ is given by 
	\[ad_\eta:\xymatrix@C=3em{\Hom(M,M)\oplus \mathbb \Hom(V,W)_1\otimes K\ar[r]^{\ \ \ \ \ \ \ \ \ \ \ (0\ \ Id)}&\Hom(V,W)_1\otimes K}.\]
	Thus, $\HH^1(C^\bullet_0(V,W,\eta))\cong H^1(\Hom(M,M)).$
 
	Recall from \eqref{EQ Lambda2V2k}, \eqref{EQ Lambda2W2k}, \eqref{EQ HomVW2k+1} that for $k<0,$ $(\Lambda_Q^2V\oplus\Lambda_Q^2W)_{2k}$ is given by 
	\[\bigoplus\limits_{j=0}^{\floor{\frac{n+k}{2}}-1}\Hom(V_{n-1-2j},V_{n-1-2j+2k})\oplus\bigoplus\limits_{j=\floor{\frac{n+k}{2}}}^{n+k-2}\Hom(W_{n-2-2j},W_{n-2-2j+2k})\]
	and 
	\[\Hom(V,W)_{2k+1}\otimes K=\bigoplus\limits_{j=0}^{n+k-1}\Hom(V_{n-1-2j},W_{n-2j+2k})\otimes K~.\]
	First note that $H^1(\Lambda_Q^2(V\oplus W)_{-2k})=0$, since it is direct sum of line bundles with degree at least $4g-4.$ A simple computation similar to those in the proof of Lemma \ref{Lemma H2=0} shows that if $2k\neq-n,$ then the Higgs field gives isomorphisms  
	\begin{equation}\label{EQ iso 0}
		\Hom(V_{n-1-2j},V_{n-1-2j+2k})\cong\Hom(V_{n-1-2j},W_{n-2j+2k})\otimes K
	\end{equation}
	for $0\leq j\leq \floor{\frac{n+k}{2}}-1$ and 
		\[\Hom(W_{n-2-2j},W_{n-2-2j+2k})\cong\Hom(V_{n-1-2j},W_{n-2-2j+2k})\otimes K\] 
		for $\floor{\frac{n+k}{2}}\leq j\leq n+k-2$. In particular, this proves that, for $k<0$ and $k\neq n,$
		\[\HH^1(C_{2k}^\bullet)\cong H^0(\Hom(V_{-n-2k+1},W_{-n+2})\otimes K)\cong H^0(K^{2k})~.\]
	When $-n=2k,$ the isomorphism \eqref{EQ iso 0} holds for $0<j\leq\floor{\frac{n+k}{2}}-1.$ For $j=0,$ we have 
	\[\Hom(V_{n-1},W_0)\otimes K\cong \Hom(V_{n-1},K)\oplus\Hom(V_{n-1},MK)\oplus\Hom(V_{n-1},M^{-1}K)\]
	and, with respect to this splitting, the map induced by the Higgs field is given by
	  \[\smtrx{1\\0\\0}: \Hom(V_{n-1},V_{-1})\to\Hom(V_{n-1},W_0)\otimes K~.\]
Since $\Hom(V_{1},W_{-n+2})\cong K^{n}$, we conclude
\[\HH^1(C^\bullet_{-n})\cong H^0(MK^n)\oplus H^0(M^{-1}K^n)\oplus H^0(K^n)~.\]

If $n$ is even, we are done. When $n$ is odd, then from \eqref{EQ Lambda2W2k+1} and \eqref{EQ HomVW2k} we have 
\[(\Lambda_Q^2V\oplus\Lambda_Q^2W)_{2k+1}=\begin{dcases}
	\Hom(W_{-2k-1},M\oplus M^{-1})& n<2k+1<0\\
	0&\text{otherwise}
\end{dcases}~,\] and 
\[\Hom(V,W)_{2k+2}\otimes K=\begin{dcases}
	\Hom(V_{-2k-2},M\oplus M^{-1})\otimes K&  0\leq 2k+2\leq n \\
	0&\text{otherwise}
\end{dcases}~.\]
For $n<2k+1<0,$ the Higgs field defines an isomorphism 
\[\Hom(W_{-2k-1},M\oplus M^{-1})\cong \Hom(V_{-2k-2},M\oplus M^{-1})\otimes K~,\] and so 
$\HH^1(C^\bullet_{2k+1})=\begin{dcases}
	H^0(MK^n)\oplus H^0(M^{-1}K^n)& -n=2k+1\\
	0&\text{otherwise}
\end{dcases}~.$
\end{proof}

The following lemma completes the proof of Theorem \ref{THM2}.
\begin{Lemma}\label{Lemma image psi0 closed}
	The image of the map $\Psi_0:\Ff_0^{ps}\times\bigoplus\limits_{j=1}^{n-1}H^0(K^{2j})\to\Mm(\sSO(n,n+1))$ from \eqref{EQ Psi0 map} is open and closed. 
\end{Lemma}
\begin{proof}
	By Lemma \ref{Lemma: s is open closure is ps}, the image of $\Psi_0$ is closed in $\Mm(\sSO(n,n+1)).$ Also, by Lemma \ref{Lemma: s is open closure is ps}, for any $x\in\Ff^s_0\times\bigoplus\limits_{j=1}^{n-1}H^0(K^{2j})$ there is an open neighborhood of $x$ which is contained in the image of $\Psi_0.$ 

	Now suppose $x\in \Ff_0^{ps}\setminus\Ff^s_0\times\bigoplus\limits_{j=1}^{n-1}H^0(K^{2j})$, and recall that $x$ can be written as
	\[x=([M,0,0],q_2,\cdots,q_{2n-2})\]
	for $M\in\Pic^0(X)$ and $q_{2j}\in H^0(K^{2j}).$ By Remark \ref{Rem C* open}, it suffices to consider points of the form $x=([M,0,0],0,\cdots,0)$ in $\Ff_0^{ps}\times\bigoplus\limits_{j=1}^{n-1}H^0(K^{2j})$. By Lemma \ref{Lemma H2=0}, we have $\HH^2(C^\bullet(\Psi_0(x))=0$, thus, by \eqref{EQ: local kuranishi}, an open neighborhood of the Higgs bundle $\Psi_0(x)$ is given by
\[\HH^1(C^\bullet(\Psi_0(x))\big\slash\big\slash\Aut(\Psi_0(x))=\begin{dcases}
	\HH^1(C^\bullet(\Psi_0(x))\big\slash\big\slash\sO(2,\C)&\text{if\ } M^2=\Oo\\
	\HH^1(C^\bullet(\Psi_0(x))\big\slash\big\slash\sSO(2,\C)& \text{if\ } M^2\neq \Oo
\end{dcases}~.\]
By Lemma \ref{Lemma HH1 decomp}, we have 
	\[\HH^1(C^\bullet(\Psi_0(x)))\cong H^0(MK^n)\times H^0(M^{-1}K^n)\times H^1(\Oo)\times \bigoplus\limits_{j=1}^{n-1}H^0(K^{2j})~.\]
Here, $H^1(\Oo)=H^1(\Lambda^2(M\oplus M^{-1}))$ and $\delta\in H^1(\Oo)$ is given by 
\[\smtrx{\delta&\\&-\delta}\in H^1(\End(M\oplus M^{-1}))~.\]
If $M^2\neq\Oo,$ then by Lemmas \ref{Lemma O2 stabilizer} and \ref{Lemma Gauge group fixing im Psi0}, $\Aut(\Psi_0(x))$ is generated by the orthogonal gauge transformations 
	\[(g_V,g_W)=\left(Id_V,\smtrx{\lambda&&\\&Id_{W_0}&\\&&\lambda^{-1}}\right)~.\]
Such a gauge transformation acts on $(\mu,\nu,\delta,q_2,\cdots, q_{2n-2})\in\HH^1(C^\bullet(V,W,\eta))$ by 
\[(\delta,\mu,\nu,q_2,\cdots, q_{2n-2})\longmapsto(\delta,\lambda^{-1}\mu,\lambda\nu,q_2,\cdots, q_{2n-2})~.\]
 If $M^2=\Oo,$ then by Lemmas \ref{Lemma O2 stabilizer} and \ref{Lemma Gauge group fixing im Psi0}, $\Aut(\Psi_0(x))$ is generated by the orthogonal gauge transformations
\[\left(Id_V,\smtrx{\lambda&&\\&Id_{W_0}&\\&&\lambda^{-1}}\right) \ \ \ \text{and} \ \ \ \left((-1)^{n+1}Id_V,(-1)^{n+1}\smtrx{&&\lambda\\&Id_{W_0}&\\\lambda^{-1}&&}\right).\]
The second gauge transformation acts on $(\delta,\mu,\nu,q_2,\cdots, q_{2n-2})\in\HH^1(C^\bullet(V,W,\eta))$ by 
$(\delta,\mu,\nu,q_2,\cdots, q_{2n-2})\longmapsto(-\delta,\lambda^{-1}\nu,\lambda\mu,q_2,\cdots, q_{2n-2})~.$

Since an open neighborhood of $M\in\Pic^0(X)$ is given by an open neighborhood of zero in $H^1(\Oo),$ an open neighborhood of a lift $\tilde x\in\widetilde \Ff_0\times \bigoplus\limits_{j=1}^{n-1}H^0(K^{2j})$ of $x$ is also given by a neighborhood of zero in
\[H^1(\Oo)\times H^0(MK^n)\times H^0(M^{-1}K^n)\times  \bigoplus\limits_{j=1}^{n-1}H^0(K^{2j})~.\]
Since the map $\Psi_0:\Ff^{ps}_0\times\bigoplus\limits_{j=1}^{n-1}H^0(K^{2j})\longrightarrow \Mm(\sSO(n,n+1))$ is a homeomorphism onto its image, the map $\Psi_0$ in open at $x$. 
\end{proof}

\subsection{Proof of Theorem \ref{THM3}}

As in previous sections, for each nonzero $sw_1\in H^1(X,\Z_2)$, let $\pi:X_{sw_1}\to X$ be the corresponding connected orientation double cover.
If $\iota$ denotes the covering involution, then consider the space
\[\Prym(X_{sw_1},X)=\{M\in\Pic^0(X_{sw_1})\ |\ \iota^* M=M^{-1}\}.\] 
Recall that Proposition \ref{Prop Mumford O2} defines a one to one correspondence between holomorphic rank two orthogonal bundles $(W,Q_W)$ with first Stiefel-Whitney class $sw_1$ and $\Prym(X_{sw_1},X)$ given by $M\to (\pi_*M,\pi_*\iota^*).$
 Recall also that $\Prym(X_{sw_1},X)$ has two connected components $\Prym^{sw_2}(X_{sw_1},X)$ labeled by the second Stiefel-Whitney class $sw_2$ of the orthogonal bundle $(\pi_*M,\pi_*\iota^*).$ 
\begin{Lemma}
	\label{Lemma Defining Higgs field sw1not0}
Define the space $\Ff^{sw_2}_{sw_1}$ by 
 	\begin{equation}\label{EQ tildeFsw1 DEF}
 		\Ff^{sw_2}_{sw_1}=\{(M,\mu)\ |\ M\in\Prym^{sw_2}(X_{sw_1},X),\ \mu\in H^0(M^{-1}K_{X_{sw_1}}^n)\}
 	\end{equation}
 	There is a well defined smooth map 
 	\begin{equation}
 		\label{EQ tildePsi sw1 Def}
 		\widetilde\Psi^{sw_2}_{sw_1}:\xymatrix{\Ff^{sw_2}_{sw_1}\times \bigoplus\limits_{j=1}^{n-1} H^0(K^{2j})\ar[r]&\{\text{stable}\ \sSO(n,n+1)\text{-Higgs\ bundles}}\}
 	\end{equation}
 	defined by $\widetilde\Psi^{sw_2}_{sw_1}(M,\mu,q_2,\cdots q_{2p-2})=(V,W,\eta)$,  where 
 	\[\xymatrix@=.5em{V=K^{n-1}\oplus K^{n-3}\oplus\cdots \oplus K^{3-n}\oplus K^{1-n}~,\\W=\pi_*M\oplus K^{n-2}\oplus K^{n-4}\oplus\cdots\oplus K^{4-n}\oplus K^{2-n}~,}\]
 	\begin{equation}
 		\label{EQ M sw1 Higgs field}
 		\eta=\mtrx{0&0&0&\cdots&0&\pi_*\mu\\
			   1&q_2&q_4&\cdots&q_{2n-4}&q_{2n-2}\\
			   0&1&q_2&q_4&\cdots&q_{2n-4}\\
			   0&0&1&q_2&\cdots&q_{2n-6}\\
			   &\ddots&\ddots&\ddots&&\\
			   &&&0&1&q_2}:V\longrightarrow W\otimes K~.
 	\end{equation} 
 	Moreover, $\widetilde\Psi^{sw_2}_{sw_1}(M,\mu,q_2,\cdots,q_{2n-2})$ and $\widetilde\Psi^{sw_2}_{sw_1}(M',\mu',q_2',\cdots,q_{2n-2}')$ lie in the same $\sS(\sO(n,\C)\times\sO(n+1,\C))$ gauge orbit if and only if, for all $j$
 	\[\xymatrix{M'=M &\mu'=\pm\mu& q_{2j}'=q_{2j}}\] 
 	 
 \[\text{or}\ \ \ \ \ \xymatrix{M'=\iota^*M &\mu'=\pm\iota^*\mu& q_{2j}'=q_{2j}}~.\]
\end{Lemma}
\begin{proof}
Before checking the image of $\widetilde\Psi_{sw_1}^{sw_2}$ consists of stable Higgs bundles, we first prove the statement about gauge orbits. Two Higgs bundles $(V,W,\eta)$ and $(V',W',\eta')$ in the image of $\widetilde\Psi_{sw_1}^{sw_2}$ lie in the same $\sS(\sO(n)\times\sO(n+1))$ gauge orbit if and only if $\pi^*(V,W,\eta)$ to $\pi^*(V',W',\eta')$ are gauge equivalent on $X_{sw_1}$ via an $\iota^*$-invariant gauge transformation.
Since $\pi^*K_X=K_{X_{sw_1}},$ the $\sSO(n,n+1)$ Higgs bundle $\pi^*\widetilde\Psi_{sw_1}^{sw_2}(M,\mu,q_2,\cdots,q_{2n-2})$ on $X_{sw_1}$ are in the image of $\widetilde\Psi_0$ from \eqref{EQ M0 Higgs field}, with $M\in\Prym(X_{sw_1})$ and $\nu=\iota^*\mu.$ 
By Lemma \ref{EQ M sw1 Higgs field}, the $\sSO(n,n+1)$ Higgs bundles $\widetilde\Psi_0(M,\mu,\nu,q_2,\cdots,q_{2n-2})$ and $\widetilde\Psi_0(M',\mu',\nu',q_2,\cdots,q_{2n-2})$ on $X_{sw_1}$ are in the same gauge orbit if and only 
\[\xymatrix@R=0em{(M',\mu',\nu',q_2',\cdots,q_{2n-2}')= (M,\lambda\mu,\lambda^{-1}\nu,q_2,\cdots,q_{2n-2})&\text{or}\\(M',\mu',\nu',q_2',\cdots,q_{2n-2}')= (M^{-1},\lambda^{-1}\nu,\lambda\mu,q_2,\cdots,q_{2n-2})&}\]
for $\lambda\in\C^*.$ The corresponding gauge transformations are given by \eqref{EQ d=0 gauge Trans}, and are $\iota^*$-invariant if and only if $\lambda=\lambda^{-1},$ i.e. $\lambda=\pm1.$

	Polystability of the Higgs bundle \eqref{EQ M sw1 Higgs field} follows almost immediately from the proof of the Lemma \ref{Lemma Polystable image of Psi0}. 
	Namely, for the zero locus of the holomorphic differentials $(q_2,\cdots,q_{2n-2})$ the corresponding $\sSL(2n+1,\C)$ Higgs bundles are cyclic and can be represented schematically as:
	\[\xymatrix@R=1em{K^{p-1}\ar[r]^1&\cdots\ar[r]^1&K^{1-p}\ar[dl]^{\pi_*\mu^T}\\ &\pi_*M\ar[ul]^{\pi_*\mu}&}.\]

	To check that this Higgs bundle is polystable it suffices to show $\pi_*M$ has no positive invariant subbundles. In fact, $\pi_*M$ does not have any positive degree subbundles. 
	Indeed, if $0\to L\to \pi_*M$ is a holomorphic subbundle if and only if there is a positive degree $\iota^*$-invariant subbundle $0\to\tilde L\to\pi^*\pi_*M=M\oplus M^{-1}.$ But $M\oplus M^{-1}$ has no positive subbundles. 

	If $\mu=0$, then the Higgs bundle is a direct sum of a stable $\sSL(2n-1,\C)$ Higgs bundle with a degree zero stable rank 2 bundle. Since $sw_1\neq0,$ the orthogonal bundle $\pi_*M$ does not have any isotropic line subbundles. Thus, if $\mu\neq0$ and $N\subset \pi_*M$ is a degree zero invariant line subbundle, then $N$ is an orthogonal subbundle. Thus, we can take its orthogonal complement and split the Higgs bundle as a stable $\sSL(2n,\C)$ Higgs bundle plus an invariant degree zero line bundle. 
	This implies the Higgs bundles in \ref{EQ M sw1 Higgs field} are polystable for $q_{2j}=0$. Since the automorphism group of such a Higgs bundles is finite, the Higgs bundle is stable.
	Using the openness of stability, as in Lemma \ref{Lemma Polystable image of Psi0} we conclude that $\widetilde\Psi^{sw_2}_{sw_1}$ is well defined.
\end{proof}
There is a $\Z_2\oplus\Z_2$ action on $\Ff_{sw_1}^{sw_2}$ generated by
\[\xymatrix{(M,\mu)\to (M,-\mu)&\text{and}&(M,\mu)\to (\iota^*M,\iota^*\mu)}~.\]
Moreover, if we extend this action trivially to $\Ff_{sw_2}^{sw_1}\times\bigoplus\limits_{j=1}^{n-1}H^0(K^{2j})$ then the map $\widetilde\Psi_{sw_1}^{sw_2}$ from \eqref{EQ tildePsi sw1 Def} is $\Z_2\oplus\Z_2$-equivariant. 
This gives a well defined continuous map 
\begin{equation}\label{eq psisw1 map}
	\Psi_{sw_1}^{sw_2}:\xymatrix{\Ff_{sw_1}^{sw_2}/(\Z_2\oplus\Z_2)\times\bigoplus\limits_{j=1}^{n-1}H^0(K^{2j})\ar[r]&\Mm(\sSO(n,n+1))}~
\end{equation}
which is a homeomorphism onto its image. 
The following lemma completes the proof of Theorem \ref{THM3}.
\begin{Lemma}
	For each $(sw_1,sw_2)\in H^1(X,\Z_2)\setminus\{0\}\times H^2(X,\Z_2)$, the image of the map $\Psi_{sw_1}^{sw_2}$ from \eqref{eq psisw1 map} is open and closed in $\Mm(\sSO(n,n+1)).$ 
\end{Lemma}
\begin{proof}
	The proof is almost equivalent to the proof of Theorem \ref{THM1}. 
	Let $(E,\Phi)$ be the $\sSL(2n+1,\C)$ Higgs bundle associated to $\Psi_{sw_1}^{sw_2}(M,\mu,q_2,\cdots,q_{2n-2})$. As in \eqref{EQ Hitchin fibration on Md}, we can choose a basis of invariant polynomials $(p_1,\cdots,p_n)$ so that 
	\[
	p_j(\Psi_{sw_1}^{sw_2}(M,\mu,\nu,q_2,\cdots,q_{2n-2}))=\begin{dcases}
		q_{2j}& 1\leq j\leq n-1\\
		\pi_*\mu^T\otimes\pi_*\mu&j=n
	\end{dcases}~.
				\] 
By properness of the Hitchin fibration, the image of any divergent sequence in $\Ff_{sw_1}^{sw_2}\times\bigoplus\bigoplus\limits_{j=1}^{n-1}H^0(K^{2j})$ also diverges in $\Mm(\sSO(n,n+1)).$ 
Thus, the image of $\Psi_{sw_1}^{sw_2}$ is closed. 

A simple calculation shows that the dimension of the image of $\Psi_{sw_1}^{sw_2}$ is the expected dimension of the moduli space $\Mm(\sSO(n,n+1))$. Since every point in the image is a smooth point or an orbifold point, the image of $\Psi_{sw_1}^{sw_2}$ is open by the same argument for openness in Lemma \ref{Lemma: s is open closure is ps}.
\end{proof}




\section{Zariski closures of reducible representations}\label{Section Zariski closure}
Recall from Proposition \ref{Prop Irreducible reps and Stable SLnC Higgs bundles} that a representation $\rho:\Gamma\to\sSO(n,n+1)$ is reducible if and only if the corresponding $\sSL(2n+1,\C)$ Higgs bundle is strictly polystable. Moreover, a representation $\rho$ has Zariski closure $\sG'\subset\sSO(n,n+1)$ if and only if the structure group of the corresponding $\sSO(n,n+1)$ Higgs bundle reduces to $\sG'$ (see Proposition \ref{Prop smaller zariski closure}).
 \subsection{A few important subgroups of $\sSO(n,n+1)$}
 Recall that $\sSO(n,n+1)$ is the group of orientation preserving linear automorphisms of $\R^{2n+1}$ which preserve a signature $(n,n+1)$-inner product. More generally, the group $\sO(n,m)$ is the group of linear automorphism of $\R^{n+m}$ which preserve a signature $(n,m)$-inner product. 

 If $Q_n$ and $Q_m$ are positive definite symmetric $n\times n$ and $m\times m$ matrices, then $\sO(n,m)$ consists of elements of $g\in\sGL(\R^{n+m})$ so that 
 \[g^T\mtrx{Q_n&\\&-Q_m}g=\mtrx{Q_n&\\&-Q_m}~.\] 
 The group has $\sO(n,m)$ has four connected components which we will denote by $\sO^{\pm,\pm}(n,m)$. If $\R^{n,0}\subset\R^{n+m}$ is a positive definite subspace of maximal dimension and $\R^{0,m}\subset\R^{n,m}$ is a negative definite subspace with maximal dimension, then an element $g\in\sO(n,m)$ is in $\sO^{+,-}(n,m)$ if it preserves an orientation of $\R^{n,0}$ and reverses an orientation of $\R^{0,m}.$ The components $\sO^{+,+}(n,m)$, $\sO^{-,+}(n,m)$ and $\sO^{-,-}(n,m)$ are defined similarly. The group $\sO^{+,\pm}(n,m)$ consists of elements which preserve the orientation an $\R^{n,0}$. 

\begin{Proposition}\label{Prop: Important subgroups}
If the quadratic form $Q_{n+1}$ is given by $Q_{n+1}=\smtrx{Q_a&\\&Q_b}$, then matrices of the form $\smtrx{A&0\\0&B}$ define subgroups of $\sSO(n,n+1)$ isomorphic to
	\begin{itemize}
		\item $\sSO(n,n-1)\times \sSO(2)$ if $a=n-1,$ $A\in\sSO(n,n-1)$ and $B\in\sSO(2)$,
		\item $\sSO(n,n)$ if $a=n$, $A\in\sSO(n,n)$ and $B=1$,
		\item $\sO^{+,\pm}(n,n)$ if $a=n,$ $A\in\sO^{+,\pm}(n,n)$ and $B=det(A)$,
		\item $\sS(\sO^{+,\pm}(n,n-1)\times\sO(2))$ if $a=n-1,$ $A\in\sO^{+,\pm}(n,n-1)$, $B\in\sO(2)$ and $det(A)=det(B).$
	\end{itemize}
\end{Proposition}
The definition of an $\sO(n,m)$ Higgs bundle is similar to that of an $\sSO(n,n+1)$ Higgs bundle.
\begin{Definition}
	An $\sO(n,m)$ Higgs bundle over a Riemann surface $X$ is a triple $(V,W,\eta)$ where 
	\begin{itemize}
	\item $V$ and $W$ are respectively rank $n$ and $m$ holomorphic vector bundles on $X$ equipped with holomorphic orthogonal structures $Q_V$ and $Q_W$.
	\item $\eta\in H^0(\Hom(V,W)\otimes K).$
\end{itemize}
An $\sO(n,m)$ Higgs bundle $(V,W,\eta)$ is an $\sO^{+,\pm}(n,m)$ Higgs bundle if $det(V)=\Oo.$ 
\end{Definition}

For $\sG'\subset\sSO(n,n+1)$ one of the subgroups from Proposition \ref{Prop: Important subgroups}, the following characterizes when an $\sSO(n,n+1)$ Higgs bundle reduces to a $G'$ Higgs bundle.

\begin{Proposition}\label{Prop: Higgs bundle reductions}
  	Let $\sG'$ be one of the subgroups from Proposition \ref{Prop: Important subgroups}. An $\sSO(n,n+1)$ Higgs bundle $(V,W,\eta)$ on $X$ reduces to a $\sG'$ Higgs bundle if  
  	\begin{itemize}
  		\item $G'=\sSO(n,n-1)\times\sSO(2)$, $(W,Q_W)=\left(W_0\oplus M\oplus M^{-1},\smtrx{Q_{W_0}&&\\&0&1\\&1&0}\right)~,$
  		where $(W_0,Q_{W_0})$ is a rank $(n-1)$ holomorphic orthogonal bundle with trivial determinant, $M\in\Pic^0(X)$ and 
  		\[ \eta=\smtrx{\eta_0\\0\\0}:V\to (W_0\oplus M\oplus M^{-1})\otimes K~.\]
  		\item $G'=\sSO(n,n),$ $(W,Q_W)=\left(W_0\oplus \Oo,\smtrx{Q_{W_0}&\\&1}\right),$
  		where $(W_0,Q_{W_0})$ is a rank $n$ holomorphic orthogonal bundle with trivial determinant and 
  		\[\eta=\smtrx{\eta_0\\0}:V\to (W_0\oplus \Oo)\otimes K~.\]
  		\item $G'=\sO^{+,\pm}(n,n)$, $(W,Q_W)=\left(W_0\oplus M,\smtrx{Q_{W_0}&\\&1}\right),$
  		where $(W_0,Q_{W_0})$ is a rank $n$ holomorphic orthogonal bundle, $M\in\Pic^0(X)$ such that $det(W_0)=M$ and 
  		\[\eta=\smtrx{\eta_0\\0}:V\to (W_0\oplus M)\otimes K~.\]
  		\item $\sG'=\sS(\sO^{+,\pm}(n,n-1)\times\sO(2)),$ $(W,Q_W)=\left(W_0\oplus W',\smtrx{Q_{W_0}&\\&Q_{W'}}\right),$ 
  		where $(W_0,Q_{W_0})$ is a rank $(n-1)$ holomorphic orthogonal bundle, $(W',Q_{W'})$ is a rank $2$ holomorphic orthogonal bundle with $det(W_0)=det(W')$ and 
  		\[\eta=\smtrx{\eta_0\\0}:V\to (W_0\oplus W')\otimes K~.\]
  	\end{itemize}
 \end{Proposition}  

 The $\sSO(n,n)$-Hitchin component is diffeomorphic to \[\Hit(\sSO(n,n))\cong \bigoplus\limits_{j=1}^{n-1}H^0(K^{2j})\oplus H^0(K^n)~.\] 
 The map $\bigoplus\limits_{j=1}^{n-1} H^0(K^{2j})\oplus H^0(K^n)\longrightarrow \Mm(\sSO(n,n))$ is defined by sending a tuple of holomorphic differentials $(q_2,q_4,\cdots,q_{2n-2},q_n)$ to the Higgs bundle $(V,W,\eta)$ where
\[\xymatrix@=0em{V=K^{n-2}\oplus K^{n-4}\oplus\cdots\oplus K^{4-n}\oplus K^{2-n}\oplus\Oo~,\\ W=K^{n-1}\oplus K^{n-3}\oplus\cdots\oplus K^{3-n}\oplus K^{1-n}~,}  \]
\begin{equation}
	\label{EQ Hitchin component Higgs field SO(n,n)}
	\eta=\mtrx{q_2&q_4&q_6&\cdots&q_{2n-2}&0\\
			   1&q_2&q_4&\cdots&q_{2n-4}&0\\
			   0&1&q_2&\cdots&q_{2n-6}&0\\
			   &\ddots&\ddots&\ddots&&\\
			   &&0&1&q_2&0\\
			   &&&0&1&0\\
			   &&&&0&q_n}:V\longrightarrow W\otimes K~.
\end{equation}
The orthogonal structures on $V$ and $W$ are given by 
\[Q_{V}=\smtrx{&&1&\\&\iddots&&\\1&&&\\&&&1}\ \ \ \ \ \ \ \ \text{and}\ \ \ \ \ \ \ \ Q_W=\smtrx{&&1\\&\iddots&\\1} .\]
\begin{Remark}
	Note that in \eqref{EQ Hitchin component Higgs field SO(n,n)}, if $q_n=0$, then the Higgs bundle reduces to $\sSO(n,n-1)\subset\sSO(n,n).$
\end{Remark}
 \subsection{Zariski closures of reducible representations}
 Recall from \eqref{EQ M0 Higgs field} that a Higgs bundle in $\Mm_0(\sSO(n,n+1))$ is determined by a tuple $(M,\mu,\nu,q_2,\cdots,q_{2n-2})$ where $M\in\Pic^0(X),$ $\mu\in H^0(M^{-1}K^n)$, $\nu\in H^0(MK^n)$ and $q_{2j}\in H^0(K^{2j})$ such that $\mu=0$ if and only if $\nu=0.$ Moreover, by Lemma \ref{Lemma O2 stabilizer}, the isomorphism class associated to a tuple $(M,\mu,\nu,q_2,\cdots,q_{2n-2})$ is a singular point of $\Mm_0(\sSO_0(n,n+1))$ if and only if $\mu=\nu=0$ or $M=M^{-1}$ and $\mu=\lambda\nu$ for $\lambda\in\C^*.$
\begin{Proposition}\label{PROP: Zariski closures}
	With the above notation, a Higgs bundle in $\Mm_0(\sSO_0(n,n+1))$ given by a tuple $(M,\mu,\nu,q_2,\cdots,q_{2n-2})$ 
	\begin{itemize}
		\item reduces to an $\sSO(n,n-1)\times\sSO(2)$ Higgs bundle whose $\sSO(n,n-1)$-factor is in the Hitchin component if $\mu=\nu=0$,
		\item reduces to an $\sSO(n,n)$ Higgs bundle in the Hitchin component if $M=\Oo$ and $\mu=\lambda\nu$ for $\lambda\in\C^*$,
		\item reduces to an $\sO^{+,\pm}(n,n)$ Higgs bundle if $M^2=\Oo$ and $\mu=\lambda\nu$ for $\lambda\in\C^*,$
	\end{itemize} 
\end{Proposition}
\begin{proof}
	Let $(V,W,\eta)$ denote the $\sSO(n,n+1)$ Higgs bundle corresponding to a tuple $(M,\mu,\nu,q_2,\cdots,q_{2n-2})$. It is given by \eqref{EQ M0 Higgs field}. The bundle $W$ can be written as $W_0\oplus M\oplus M^{-1}$ where $W_0=K^{n-2}\oplus K^{n-4}\oplus\cdots\oplus K^{2-n}.$ If $\mu=\nu=0$, then with respect to this splitting the Higgs field decomposes as 
	\[\eta=\mtrx{\eta_0\\0\\0}:V\longrightarrow (W_0\oplus M\oplus M^{-1})\otimes K~,\]
	where $\eta_0:V\to W_0\otimes K$ is the Higgs field in the $\sSO(n,n-1)$-Hitchin component associated to the holomorphic differentials $(q_2,\cdots,q_{2n-2}).$ 
	Thus, by Proposition \ref{Prop: Higgs bundle reductions}, the structure group reduces to $\sSO(n,n-1)\times \sSO(2)$. 

	If $M=M^{-1}$ and $\mu=\lambda\nu$, consider the portions of $\eta$ and $\eta^*$ given by 
	\[\eta_\mu=\mtrx{\lambda\mu\\\mu}:K^{-n}\ra M\oplus M^{-1}\ \ \ \ \ \ \text{and} \ \ \ \ \ \ \eta^*_\mu=\mtrx{\mu&\lambda\mu}:M\oplus M^{-1}\ra K^n.\] 
	The kernel of $\eta_\mu^*$ is an orthogonal subbundle of $M\oplus M^{-1}$ which is isomorphic to $M.$ Moreover, the image of $\eta_\mu$ is exactly the orthogonal complement of $ker(\eta_\mu^*).$ Thus the orthogonal bundle $(W,Q_W)$ can be written as 
	\[(W,Q_W)=\left(M\oplus K^{n-2}\oplus\cdots\oplus K^{2-n}\oplus M\smtrx{1&&&&\\&&&1&\\&&\iddots&&\\&1&&\\&&&&1}\right)\]
	with Higgs field given by 
	\[\eta=\mtrx{0&\cdots&0&\mu\\
			   1&q_2&\cdots&q_{2n-2}\\
			   &\ddots&\ddots&&\\
			   &0&1&q_2\\
			   &&0&0}:V\longrightarrow W\otimes K~.\]
Thus, the Higgs bundle $(V,Q_V,W,Q_W,\eta)$ decomposes a direct sum of $M$ (with zero Higgs field) and $(V,Q_V,W_0,Q_{W_0},\eta_0)$  where $W_0= K^{n-2}\oplus \cdots\oplus K^{2-n}\oplus M.$ The determinant of $W_0$ is $M$, thus the structure group of the Higgs bundle reduces to $\sO^{+,\pm}(n,n)\subset\sSO(n,n+1)$ by Proposition \ref{Prop: Higgs bundle reductions}. When $M=\Oo,$ the structure group reduces to $\sSO(n,n)$ and $\mu\in H^0(K^n)$. Thus, the Higgs field $\eta':V\ra W_0$ is in the $\sSO(n,n)$-Hitchin component \eqref{EQ Hitchin component Higgs field SO(n,n)}. 
	\end{proof} 

\begin{Theorem}\label{THM: Zariski closures}
	If $\rho:\Gamma\to\sSO(n,n+1)$ is a reducible representation which defines a point in $\Xx_0(\sSO(n,n+1)$ or $\Xx_{sw_1}^{sw_2}(\sSO(n,n+1))$, then there is a finite index subgroup $\widehat\Gamma\subset\Gamma$ such that the restriction of $\rho$ to $\widehat\Gamma$ either factors through $\sSO(n,n-1)\times \sSO(2)$ with $\sSO(n,n-1)$ factor in the Hitchin component, or factors through an $\sSO(n,n)$-Hitchin representation.
\end{Theorem}
\begin{proof}
If $\rho\in\Xx_0(\sSO(n,n+1))$ is reducible, then the associated $\sSO(n,n+1)$ Higgs bundle is determined by a tuple $(M,\mu,\nu,q_2,\cdots,q_{2n-2})$ with $\mu=\nu=0$ or $M=M^{-1}$ and $\mu=\lambda\nu$ for $\lambda\in\C^*.$ 
Indeed, if $0\neq\mu\neq\lambda\nu$, then the corresponding $\sSL(n,\C)$ Higgs bundle is stable, and hence the representation $\rho$ is irreducible. 

By Proposition \ref{Prop: Higgs bundle reductions}, if $\mu=\nu=0$ the Higgs bundle reduces to an $\sSO(n,n-1)\times\sSO(2)$ Higgs bundle whose $\sSO(n,n-1)$-factor is in the Hitchin component. Similarly, if $M=\Oo$ and $\mu=\lambda\mu$, the Higgs bundle reduces to an $\sSO(n,n)$ Higgs bundle in the Hitchin component. 
In both of these cases, the corresponding representation either factors through $\sSO(n,n-1)\times \sSO(2)$ with $\sSO(n,n-1)$ factor in the Hitchin component or an $\sSO(n,n)$-Hitchin representation.
If $M^2\cong\Oo$, $M\neq \Oo$ and $\mu=\lambda\nu$, the first Stiefel-Whitney class $sw_1$ of the orthogonal bundle $M$ is nonzero. Let $\pi: X_{sw_1}\to X$ be the associated connected orientation double cover. Since $\pi^*M=\Oo$ and $\pi^* K=K_{X_{sw_1}}$, the pull back of the Higgs bundle to $X_{sw_1}$ reduces to an $\sSO(n,n)$ Higgs bundle in the Hitchin component. Thus, the restriction of the representation $\rho$ to the index two subgroup $\pi_1(X_{sw_1})=\widehat{\Gamma}$ factors through an $\sSO(n,n)$-Hitchin representation.

For $\rho\in\Xx_{sw_1}^{sw_2}(\sSO(n,n+1))$, let $X_{sw_1}\to X$ be the connected orientation double cover associated to $sw_1\in H^1(X,\Z_2)\setminus\{0\}$ and let $\pi_1(X_{sw_1})=\widehat\Gamma\subset \Gamma$ be the associated index two subgroup. 
By construction, the restriction of $\rho$ to $\widehat\Gamma$ defines a representation in the connected component $\Xx_0(\sSO(n,n+1))$ of the character $\sSO(n,n+1)$-character variety of $\widehat\Gamma.$ 
Thus, there is a finite index subgroup of $\widehat{\widehat\Gamma}\subset\widehat\Gamma$ such that the restriction of $\rho$ factors through $\sSO(n,n-1)\times \sSO(2)$ with $\sSO(n,n-1)$ factor in the Hitchin component or an $\sSO(n,n)$-Hitchin representation.
\end{proof}

\subsection{Zariski closure of representations in $\Xx_d(\sSO(n,n+1))$ for $d>0$}
Recall that the connected components $\Xx_d(\sSO(n,n+1))$ from Theorem \ref{THM1} are smooth for $d\in(0,n(2g-2)]$. In particular, every representation in such a component is irreducible. Recall also that every representations in the components $\Xx_d(\sSO(n,n+1))$ factors through the connected component of the identity $\sSO_0(n,n+1)$. Note that the Zariski closure of $\sSO_0(n,n+1)\subset\sSO(n,n+1)$ is the full group $\sSO(n,n+1).$
\begin{Conjecture}
	For $0<d<n(2g-2)$, all representations in $\Xx_d(\sSO(n,n+1))$ are Zariski dense. 
\end{Conjecture}
For $d=n(2g-2)$ the component $\Xx_{n(2g-2)}(\sSO(n,n+1))$ is the $\sSO(n,n+1)$-Hitchin component. Thus, by the definition of the Hitchin component (Definition \ref{Def: Hitchin comp}), $\Xx_{n(2g-2)}(\sSO(n,n+1))$ contains representations which are not Zariski dense. 
In \cite{SO23LabourieConj}, it is shown that, for $n=2$ and $d\in(0,4g-4)$, every representation in the components $\Xx_d(\sSO(2,3))$ is Zariski dense. The proof relies on the fact that $\sSO_0(2,3))$ is a group of Hermitian type and that the representations in $\Xx_d(\sSO(2,3))$ are maximal representations. Thus, by results of \cite{BIWmaximalToledoAnnals}, the Zariski closure of such a representation is a tightly embedded subgroup of Hermitian type. 
Using the Higgs bundles, one can rule out the handful of proper subgroups of $\sSO_0(2,3)$ which are tightly embedded. 

For $n>2,$ the group $\sSO_0(n,n+1)$ is not of Hermitian type, so the above methods do not apply.
However, since the components $\Xx_d(\sSO(n,n+1))$ are smooth, the only way a representation $\rho\in\Xx_d(\sSO(n,n+1))$ can have a Zariski closure $\sG'$ smaller than $\sSO_0(n,n+1)$ is if $\sG'$ is a {\em simple} Lie group and there is a {\em faithful irreducible representation} representation $\psi:\sG'\to \sGL(\R^{2n+1})$ which preserves a signature $(n,n+1)$ inner product and $\rho$ factors through $\sG':$
\[\xymatrix{\Gamma\ar[r]^{\rho\ \ \ \ }\ar@{-->}[dr]&\sSO(n,n+1)\\&\sG'\ar@{^{(}->}[u]_\psi}~.\]

As an example, the signature of the Killing form for $\sSU(p,p)$ is $(2p^2,2p^2-1)$, thus the adjoint representation of $\sSU(p,p)$ provides such an irreducible representation. This doesn't occur until $\sSO(7,8).$ 
Also, there is an irreducible seven dimensional representation of $\sG_2$ which preserves a signature $(3,4)$ inner product. However, one can show directly that the Higgs bundles in the components $\Mm_d(\sSO(3,4))$ do not reduce to $\sG_2$. 

Using the software {\em Atlas}, one can list the irreducible representations of a fixed Lie group $\sG'$ which admit an irreducible representations which preserves a signature $(n,n+1)$ inner product. In particular, for $3<n<7,$ there are no simple Lie groups $\sG'$ which admits a faithful irreducible representation $\psi:\sG'\to\sGL(\R^{2n+1})$ which preserves an signature $(n,n+1)$ inner product.

\section{Positive Anosov representations}
\label{section: Positive}
In this section we show that all reducible representations in the connected components of $\Xx(\sSO(n,n+1))$ described in Theorems \ref{THM2} and \ref{THM3} are positive Anosov representations. We first recall the notion of an Anosov representation, then review the work of Guichard-Wienhard \cite{PosRepsGWPROCEEDINGS} and Guichard-Labourie-Wienhard \cite{PosRepsGLW} on positive representations. After describing the positive structures for the groups $\sSO(n,n)$ and $\sSO(n,n+1),$ Theorem \ref{THM4} is proven.

Anosov representations were introduced by Labourie \cite{AnosovFlowsLabourie} and have many interesting geometric and dynamic properties which generalize convex cocompact representations into rank one Lie groups. 
Important examples of Anosov representations include Hitchin representations into split real groups and maximal representations into Lie groups of Hermitian type. We will describe the main properties of Anosov representations which will be useful for our setting, and refer the reader to \cite{AnosovFlowsLabourie,guichard_wienhard_2012,AnosovAndProperGGKW,KLPDynamicsProperCocompact} for more details. 

Let $\sG$ be a semisimple Lie group and $\sP\subset\sG$ be a parabolic subgroup. Let $\sL\subset\sP$ be the Levi factor (the maximal reductive subgroup) of $\sP$. If $\sP^{opp}$ denotes the opposite parabolic of $\sG$, then $\sL=\sP\cap\sP^{opp}.$ We will mostly be interested in $\sG=\sSO(n,n+1)$, in this case all parabolic subgroups are conjugate to there opposites. We will assume all parabolic subgroups are conjugate to their opposite from now on.   

The homogeneous space $\sG/\sL$ is the unique open $\sG$ orbit in $\sG/\sP\times\sG/\sP$. A pair of distinct generalized flags $(x,y)\in\sG/\sP\times\sG/\sP$ are called {\em transverse} if they are in the unique open $\sG$-orbit $\sG/\sL.$

\begin{Definition}\label{DEF: Anosov rep}
	Let $\Gamma$ be the fundamental group of a closed surface of genus $g\geq 2$. Let $\p_\infty\Gamma$ be the Gromov boundary of the group $\Gamma,$ topologically $\p_\infty\Gamma\cong\R\P^1$. A representation $\rho:\Gamma\ra\sG$ is $\sP$ Anosov if and only if there exists a unique continuous boundary map 
\[\xi_\rho:\xymatrix{\p_\infty\Gamma\ar[r]&\sG/\sP}\]
which satisfies 
\begin{itemize}
	\item Equivariance: $\xi(\gamma\cdot x)=\rho(\gamma)\cdot\xi(x)$ for all $\gamma\in\Gamma$ and all $x\in\p_\infty\Gamma$.
	\item Transversality: for all distinct $x,y\in\p_\infty\Gamma$ the generalized flags $\xi(x)$ and $\xi(y)$ are transverse.
	\item Dynamics preserving: see \cite{AnosovFlowsLabourie,guichard_wienhard_2012,AnosovAndProperGGKW,KLPDynamicsProperCocompact} for the precise notion. 
\end{itemize}
The map $\xi_\rho$ will be called the {\em $\sP$ Anosov boundary curve}.
\end{Definition}
\begin{Remark}
	\label{REM: FACTS about ANOSOV}
	The following facts about Anosov representations will be important:
\begin{itemize}
	\item Openness: Let $\rho:\Gamma\to\sG$ be a $\sP$ Anosov representation, there is an open neighborhood of $\rho$ in $\Xx(\sG)$ consisting of $\sP$ Anosov representations.
	\item Action of centralizer: The centralizer of $\rho$ acts trivially on $\xi(\p_\infty\Gamma)$. 
	\item Finite index subgroups: A representation $\rho$ is a $\sP$ Anosov representation if and only if the restriction of $\rho$ to any finite index subgroup $\widehat\Gamma\subset\Gamma$ is $\sP$ Anosov. 
\end{itemize}
\end{Remark}

\subsection{Positive Anosov representations}
The important cases of Hitchin representations and maximal representations define connected components of Anosov representations. Both Hitchin representations and maximal representations satisfy an additional ``positivity'' property which is a closed condition. 
For Hitchin representations this was proven by Labourie \cite{AnosovFlowsLabourie} and Fock-Goncharov \cite{fock_goncharov_2006}, and for maximal representations by Burger-Iozzi-Wienhard \cite{MaxRepsAnosov}. 
These notions of positivity have recently been unified by Guichard-Wienhard \cite{PosRepsGWPROCEEDINGS}. The generalized notion of positivity defined below is conjectured to be a closed condition. 

For a parabolic subgroup $\sP\subset\sG$, denote the Levi factor of $\sP$ by $\sL$ and the unipotent subgroup by $\sU\subset\sP$. 
The Lie algebra $\fp$ of $\sP$ admits an $Ad_{\sL}$-invariant decomposition $\fp=\fl\oplus\fu$ where $\fl$ and $\fu$ are the Lie algebras of $\sL$ and $\sU$ respectively. 
Moreover, the unipotent Lie algebra $\fu$ decomposes as 
\[\fu=\bigoplus\fu_\beta~,\]
 where $\fu_\beta$ is an irreducible $\sL$-representation. 
Recall that a parabolic subgroup $\sP$ is determined by fixing a simple restricted root system $\Delta$ of a maximal $\R$-split torus of $\sG$, and choosing a subset $\Theta\subset\Delta$ of simple roots. 
To each simple root $\beta_j\in\Theta$ there is a corresponding irreducible $\sL$-representation space $\fu_{\beta_j}.$

\begin{Definition}(\cite[Definition 4.2]{PosRepsGWPROCEEDINGS})\label{DEF: Positive Structure}
	A pair $(\sG,\sP_\Theta)$ admits a positive structure if for all $\beta_j\in\Theta,$ the $\sL_\Theta$-representation space $\fu_{\beta_j}$ has an $\sL_\Theta^0$-invariant acute convex cone $c_{\beta_j}^\Theta$, where $\sL_\Theta^0$ denotes the identity component of $\sL_\Theta$. 
\end{Definition}
\begin{Remark}
	When $\sG$ is a split real form and $\Theta=\Delta$, the corresponding parabolic is a Borel subgroup of $\sG$. In this case, the connected component of the identity of the Levi factor is $\sL_\Delta^0\cong(\R^+)^{rk(\sG)}$ and each simple root space $\fu_{\beta_i}$ is one dimensional. The $\sL_\Delta^0$-invariant acute convex cone in each simple root space $\fu_{\beta_i}$ is isomorphic to $\R^+.$ 
	When $\sG$ is a group of Hermitian type and $\sP$ is the maximal parabolic associated to the Shilov boundary of the Riemannian symmetric space of $\sG$, the pair $(\sG,\sP)$ also admits a notion of positivity \cite{BIWmaximalToledoAnnals}. 
\end{Remark}
Recall that the Weyl group $\Ww$ of a root system is generated by reflections $s_\alpha$ associated to the simple roots $\alpha\in\Delta.$ 
In \cite{PosRepsGWPROCEEDINGS}, it is shown that, if $(\sG,\sP_\Theta)$ admits a notion of positivity, then there is at most one simple root $\beta_\Theta\subset\Theta$ which, in the Dynkin diagram of $\Delta,$ is connected to $\Delta\setminus \Theta.$ 
Denote the longest word in the Weyl group of $\Delta\setminus \Theta$ by $\omega^0(\Delta\setminus \Theta).$

 \begin{Definition}
	\label{DEF: W(theta) group} If $(\sG,\sP_\Theta)$ admits a positive structure, define $\Ww(\Theta)\subset\Ww$ as the subgroup generated by $\{\sigma_\beta\ |\ \beta\in\Theta\}$ where 
	\[\sigma_{\beta}=\xymatrix{
		s_\beta\ \  \text{if}\ \beta\in\Theta\setminus\{\beta_\Theta\}&\text{and}&
		\sigma_{\beta_\Theta}= \omega_0(\Delta\setminus\Theta)}~.\]
\end{Definition}
If $(\sG,\sP_\Theta)$ admits a positive structure, then exponentiating certain combinations of elements in the $\sL_\Theta^0$-invariant acute convex cones give rise to a semigroup $\sU^{>0}\subset\sU$. 
\begin{Definition}\label{DEF: positive semi group}
Suppose $(\sG,\sP_\Theta)$ admits a positive structure and, for each $\beta\in\Theta$, let $c^\Theta_\beta$ be the corresponding $\sL_{\Theta}^0$-invariant acute convex cone in $\fu_\beta$ given in Definition \ref{DEF: Positive Structure}. 
Denote the longest word in the group $\Ww(\Theta)$ by $\omega^0_\Theta$, and suppose $\omega^0_\Theta=\sigma_{\beta_{j_1}}\cdots\sigma_{\beta_{j_\ell}}$ is a reduced expression. Define the semisubgroup $\sU_\Theta^{>0}\subset\sU_\Theta$ to be the image of  
	\begin{equation}
		\label{EQ: semigroup Def}F_{\sigma_{j_1}\cdots\sigma_{j_\ell}}:
		\xymatrix@R=.5em@C=3em{c^\Theta_{\beta_{j_1}}\times\cdots c^\Theta_{\beta_{j_\ell}}\ar[r]&\sU_\Theta\\
		(v_{j_1},\cdots,v_{j_\ell})\ar@{|->}[r]&\exp(v_{j_1})\cdots \exp(v_{j_\ell})}~.
	\end{equation}
The semigroup $\sU_\Theta^{>0}$ will be called the {\em positive semigroup}.
\end{Definition}
\begin{Theorem}(Theorem 4.5 \cite{PosRepsGWPROCEEDINGS}) The positive subsemigroup $\sU_\Theta^{>0}$ from Definition \ref{DEF: positive semi group} is independent of the reduced expression of the longest word $\omega^0_\Theta$ of $\Ww(\Theta).$
\end{Theorem}
The positive semigroup $\sU^{>0}_\Theta$ allows one to define a notion of positively ordered triples in the generalized flag variety $\sG/\sP_\Theta.$ Since the group $\sG$ acts transitively on the space of transverse generalized flags, any two generalized flags $x,y\in\sG/\sP_{\Theta}$ can be mapped to the generalized flags $(x_+,x_-)$ associated to $\sP_\Theta$ and $\sP_\Theta^{opp}$ respectively.   
\begin{Definition}(\cite[Definition 4.6]{PosRepsGWPROCEEDINGS})
 	\label{DEF: Positive triples in G/P}
 	Let $x_+\in\sG/\sP_\Theta$ be the generalized flag associated to $\sP_\Theta$ and $x_-\in\sG/\sP_\Theta$ be the generalized flag associated to $\sP_\Theta^{opp}.$ 
Any flag $x_0$ which is transverse to $x_+$ is the image of $x_-$ under a unique element $u_0\in\sU_\Theta$. The triple $(x_+,x_0,x_-)$ is {\em positive} if $u_0$ is in the positive subsemigroup $\sU^{>0}_\Theta.$
 \end{Definition} 
 With respect to the orientation on $\p_\infty\Gamma$, we say that a triple of pairwise distinct points $(a,b,c)$ is a {\em positive triple} if the points appear in this order.  
\begin{Definition}(\cite[Definition 5.3]{PosRepsGWPROCEEDINGS})\label{DEF: Positive rep}
	If the pair $(\sG,\sP_\Theta)$ admits a positive structure, then a $\sP_\Theta$ Anosov representation $\rho:\Gamma\ra\sG$ is called a positive if the Anosov boundary curve $\xi:\p_\infty\Gamma\ra\sG/\sP_\Theta$ sends positive triples in $\p_\infty\Gamma$ to positive triples in $\sG/\sP_\Theta.$
\end{Definition}


\subsection{Positive structures for $\sSO(n,n)$ and $\sSO(n,n+1)$}
We now discuss the positive structures for the groups $\sSO(n,n)$ and $\sSO(n,n+1)$ and discuss how the embeddings 
\[\sSO(n,n-1)\subset\sSO(n,n)\subset\sSO(n,n+1)\]
preserve these notions of positivity. 
For $j\in\{2n-1,2n,2n+1\}$ and $x=(x_1,\cdots,x_{j})\in\R^{j}$, the inner product
\begin{equation}
	\label{EQ: n,n-1 inner product}
	\ip{x}{x}_{n,n-1}=2x_1x_{2n-1}+\cdots+2x_{n-1}x_{n+1}+x_n^2
\end{equation}
has signature $(n,n-1),$ the inner product
\begin{equation}
	\label{EQ: n,n inner product}
	\ip{x}{x}_{n,n}=2x_1x_{2n}+\cdots+2x_{n}x_{n+1}
\end{equation}
has signature $(n,n)$, and the inner product 
\begin{equation}
	\label{EQ: n,n+1 inner product}
	\ip{x}{x}_{n,n+1}=2x_1x_{2n+1}+\cdots+2x_{n}x_{n+2}-x_{n+1}^2
\end{equation}
has signature $(n,n+1).$

Consider the following isometric embeddings 
\begin{equation}
	\label{EQ: isometric embedings} 
	\xymatrix@R=.3em@C=1.5em{
	(\R^{2n-1},\ipd_{n,n-1})\ar[r]^{\iota_{n,n-1}}&(\R^{2n},\ipd_{n,n})\\
	(x_1,\cdots,x_{2n-1})\ar@{|->}[r]&(x_1,\cdots,x_{n-1},\frac{x_n}{\sqrt{2}},\frac{x_n}{\sqrt 2},x_{n+1},\cdots,x_{2n-1})&\\\\(\R^{2n},\ipd_{n,n})\ar[r]^{\iota_{n,n}}&(\R^{2n+1},\ipd_{n,n+1})\\
	(x_1,\cdots,x_{2n})\ar@{|->}[r]&(x_1,\cdots,x_n,0,x_{n+1},\cdots,x_{2n})
	}
\end{equation}
Let $\iota_{n,n-1}:\sSO(n,n-1)\ra\sSO(n,n)$ and $\iota_{n,n}:\sSO(n,n)\ra\sSO(n,n+1)$ be the embeddings induced by the isometric embeddings \eqref{EQ: isometric embedings}.

\smallskip 

\noindent$\mathbf{\sG=\sSO(n,n-1), \ \Theta=\Delta:}$
The group $\sSO(n,n-1)$ consists of $(2n-1)\times(2n-1)$-matrices $A\in\sSL(2n-1,\R)$ which preserve the inner product \eqref{EQ: n,n-1 inner product}. The set of diagonal matrices 
 \begin{equation}
 		\label{EQ torus SO(n,n-1)}
 	\sT=\{ diag(t_1,t_2,\cdots,t_{n-1},1,t_{n-1}^{-1},\cdots,t^{-1}_1)\ | t_i\in\R^*\}
 	\end{equation}
 	is a maximal split torus of $\sSO(n,n-1)$. 
 	The Lie algebra $\ft$ of $\sT$ is given by 
 	\[\ft=\{diag(x_1,\cdots,x_{n-1},0,-x_{n-1},\cdots,-x_1)\ |\ x_i\in\R\}~.\]
 	Consider the simple root system $\Delta=\{\beta_1,\cdots,\beta_{n-2},\beta_{n-1}\}$ with 
\[\beta_j(diag(x_1,\cdots,x_{n-1},0,-x_{n-1},\cdots,-x_1))=\begin{dcases}
	x_j-x_{j+1}&1\leq j\leq n-2\\
	x_{n-1}&j=n-1
\end{dcases}~.\]
 The parabolic $\sP_\Delta$ associated to $\Delta$ has Levi factor $L_\Delta=\sT$. 
 The decomposition of the unipotent Lie algebra $\fu_\Delta$ into irreducible $\sL_\Delta$ representations is the same as the decomposition into positive root spaces. 
 	Let $E_{ij}$ be the elementary matrix with a $1$ in the $(i,j)$ entry and zero elsewhere. The root spaces of the simple roots are  
 	\[\fg_{\beta_i}=\langle E_{i,i+1}-E_{2n-1-i,2n-i}\rangle~.\]

 	The identity component of the Levi factor $L_\Delta^0$ consists of diagonal matrices of the form \eqref{EQ torus SO(n,n-1)} with positive entries. 
 	An element $(t_1,\cdots,t_{n-1},1,t_{n-1}^{-1},\cdots,t_1^{-1})$ acts on the simple root space $\fu_{\beta_i}$ by $t_it_{i+1}^{-1}$ for $1\leq i\leq n-2$ and by $t_{n-1}$ on $\fg_{\beta_{n-1}}.$ 
 	The invariant acute cone in the simple root space $\fg_{\beta_i}$ is  
 	\[c^\Delta_{\beta_i}=\R^+\cdot (E_{i+1,i}-E_{2n-1-i,2n-i})~.\]

The group $\Ww(\Theta)$ from Definition \ref{DEF: W(theta) group} is the whole Weyl group $\Ww$, and is generated by reflections $s_{\beta_j}$. A reduced expression for the longest word $\omega^0_{\Delta}(\sSO(n,n-1))$ in the Weyl group for $\sSO(n,n-1)$ is given by 
\begin{equation}\label{EQ: SO(n,n-1) longest word} 
	\omega^0_{\Delta}(\sSO(n,n-1))=b_{1}b_{2}\cdots b_{n-1}~.
\end{equation}
where 
\[b_j=s_{\beta_{n-j}}\cdot s_{\beta_{n-j+1}}\cdots s_{\beta_{n-2}}\cdot s_{\beta_{n-1}}\cdot s_{\beta_{n-2}}\cdots s_{\beta_{n-j+1}}\cdot s_{\beta_{n-j}}~.\]
Define $B_j^\Delta:c^\Delta_{\beta_{n-j}}\times\cdots\times c^\Delta_{\beta_{n-2}}\times c^\Delta_{\beta_{n-1}}\times c^\Delta_{\beta_{n-2}}\cdots \times c^\Delta_{\beta_{n-j}}\longrightarrow \sU_\Delta$
by 
\begin{equation}
	\label{EQ: Bj Delta}
	\xymatrix@R=.2em{B^\Delta_j(u_{n-j},\cdots,u_{n-2},v_{n-1},w_{n-2},\cdots,w_{n-j})=\\\exp(u_{n-j})\cdots \exp(u_{n-2})\cdot \exp(v_{n-1})\cdot \exp(w_{n-2})\cdots \exp(w_{n-j})}~.
\end{equation}
The positive semisubgroup $U_\Delta^{>0}\subset\sU_\Delta$ from Definition \ref{DEF: positive semi group} is given by the image of \[B_1^\Delta \cdots B^\Delta_{n-1}~.\]

\smallskip

\noindent$\mathbf{\sG=\sSO(n,n),\ \Theta=\Delta}$: The group $\sSO(n,n)$ consists of $2n\times2n$-matrices $A\in\sSL(2n,\R)$ which preserve the inner product \eqref{EQ: n,n inner product}. The set of diagonal matrices 
\begin{equation}
 		\label{EQ torus SO(n,n)}
 	\sT=\{ A=diag(t_1,t_2,\cdots,t_{n},t_{n}^{-1},\cdots,t^{-1}_1)\ | t_i\in\R^*\}
 	\end{equation}
 	is a maximal split torus of $\sSO(n,n)$. 
 	The Lie algebra $\ft$ of $\sT$ is given by 
 	\[\ft=\{X=diag(x_1,\cdots,x_{n},-x_{n},\cdots,-x_1)\ |\ x_i\in\R\}~.\]
 	Consider the simple root system $\Delta=\{\delta_1,\cdots,\delta_{n}\}$ with 
 	\[\delta_j(X)=\begin{dcases}
 		x_j-x_{j+1}&1\leq j\leq n-1\\x_{n-1}+x_n
 	\end{dcases}~.\] 
 	The parabolic $\sP_\Delta$ associated to $\Delta$ has Levi factor $L_\Delta=\sT$. 
 	The decomposition of the unipotent Lie algebra $\fu_\Delta$ into irreducible $\sL_\Delta$ representations is the same as the decomposition into positive root spaces $\fu_\Delta=\bigoplus\limits_{\delta\in R^+}\fu_{\delta}$. 
 	The root spaces of the simple roots are given by 
 	\[\fg_{\delta_i}=\begin{dcases}
 		\langle E_{i,i+1}-E_{2n+1-i,2n-i}\rangle&1\leq i\leq n-1\\\langle E_{1+n,n-1}-E_{n+2,n}\rangle &i=n
 	\end{dcases}~.\]  
 	The identity component of the Levi factor $L_\Delta^0$ consists of diagonal matrices of the form \eqref{EQ torus SO(n,n)} with positive entries. 
 	An element $(t_1,\cdots,t_{n},t_{n}^{-1},\cdots,t_1^{-1})$ acts on the simple root space $\fu_{\delta_i}$ by $t_it_{i+1}^{-1}$ for $1\leq i\leq n-1$ and by $t_{n-1}t_n$ on $\fg_{\delta_{n}}.$ 
 	The invariant acute cone $c^\Delta_{\delta_i}$ in the simple root space $\fg_{\delta_i}$ is given by 
 	\[c^\Delta_{\delta_i}=\begin{dcases}
 		\R^+\cdot (E_{i+1,i}-E_{2n-i,2n-i-1})& 1\leq i\leq n-1\\\R^+\cdot (E_{n-1,n+1}-E_{n,n+2})&i=n
 	\end{dcases}~.\]

Since $\Theta=\Delta,$ the group $\Ww(\Theta)$ from Definition \ref{DEF: W(theta) group} is the whole Weyl group $\Ww$; it is generated by reflections $s_{\delta_j}$. A reduced expression for the longest word $\omega^0_{\Delta}(\sSO(n,n))$ in the Weyl group for $\sSO(n,n)$ is given by 
\begin{equation}\label{EQ: SO(n,n) longest word} 
	\omega^0_{\Delta}(\sSO(n,n))=d_{1}d_{2}\cdots d_n~,
\end{equation}
where 
\[d_j=\begin{dcases}
	s_{\delta_{n+1-j}}&j\leq 2\\
	s_{\delta_{n+1-j}}\cdot s_{\delta_{n-j}}\cdots s_{\delta_{n-2}}\cdot s_{\delta_{n-1}}\cdot s_{\delta_n}\cdot s_{\delta_{n-2}}\cdots s_{\delta_{n-j}}\cdot s_{\delta_{n+1-j}}& 3\leq j\leq n 
\end{dcases}~.\]
For $j=1,2$ define 
\[D_j:\xymatrix@R=.2em{c^\Delta_{\delta_{n+1-j}}\ar[r]&\sU_\Delta\\v\ar@{|->}[r]&\exp(v)}\]
 and, for $3\leq j\leq n$ define 
\[D_j:\xymatrix{c^\Delta_{\delta_{n+1-j}}\times\cdots\times c^\Delta_{\delta_{n-2}}\times c^\Delta_{\delta_{n-1}}\times c^\Delta_{\delta_n}\times c^\Delta_{\delta_{n-2}}\times \cdots \times c^\Delta_{\delta_{n+1-j}}\ar[rr]&&\sU_\Delta}\]
by 
\[D_j(u_{n+1-j},\cdots,u_{n-2},v_{n-1},v_n,w_{n-2},\cdots,w_{n+1-j})=\]
\begin{equation}
	\label{EQ: Dj}
\exp(u_{n+1-j})\cdots \exp(u_{n-2})\cdot \exp(v_{n-1})\cdot \exp(v_n) \cdot \exp(w_{n-2})\cdots \exp(w_{n+1-j})~.\end{equation}
The positive semisubgroup $U_\Delta^{>0}\subset\sU_\Delta$ from Definition \ref{DEF: positive semi group} is given by the image of \[D_1\cdot D_2\cdot D_3\cdots D_n~.\]

\begin{Proposition}\label{Prop: Delta to Delta Positivity}
The embedding $\iota_{n,n-1}:\sSO(n,n-1)\ra\sSO(n,n)$ induced by the isometric embedding \eqref{EQ: isometric embedings} maps the positive semigroup of $(\sSO(n,n-1),\Delta)$ into the positive semigroup of $(\sSO(n,n),\Delta)$.
\end{Proposition}
\begin{proof} 
Let $E_{ij}$ be the elementary matrix with a $1$ in the $(i,j)$ entry and zero elsewhere. 
The isometric embedding $\iota_{n,n-1}$ from \eqref{EQ: isometric embedings} induces a map $\iota:\fgl(2n-1,\R)\ra\fgl(2n,\R)$ given by 
	\[\iota(E_{ij})=\begin{dcases}
		E_{ij}& 1\leq i< n\ \text{and\ }1\leq j< n\\
		E_{in}+E_{i,n+1} & j=n\\
		E_{i+1,j}& n+1\leq i\leq 2n-1\ \text{and\ }1\leq j< n\\
		E_{i,j+1}& 1\leq i< n\ \text{and\ }n+1\leq j\leq 2n-1\\
		E_{i+1,j+1}& n+1\leq i\leq2n-1\ \text{and\ }n+1\leq j\leq 2n-1\\
	\end{dcases}~.\] 
In particular, for $1\leq j\leq n-2$, the restriction of $\iota$ to the simple root space $\fu_{\beta_j}$ is the identity, $\iota|_{\fu_{\beta_j}}=Id:\fu_{\beta_j}\to\fg_{\delta_j}$. 
The restriction of $\iota$ to the simple root space $\fu_{\beta_{n-1}}$ maps $\fu_{\beta_{n-1}}$ into the direct sum of the simple root spaces $\fu_{\delta_{n-1}}\oplus\fu_{\delta_n}$ via the diagonal map $x\to (x,x).$ 
Thus, the inclusion $\iota:\fso(n,n-1)\to\fso(n,n)$ maps the product of positive cones $c_{\beta_j}^\Delta$ into the product of positive cones $c_{\delta_j}^\Delta.$

Let $\sU^{>0}_\Delta(n,n-1)$ be the positive subsemigroup for $\sSO(n,n-1)$. If $g\in\sU^{>0}_\Delta(n,n-1)$, then 
\[g=g_1\cdot g_2\cdots g_{n-1}~,\]
where $g_j=B^{\Delta}_j(u_{n-j},\cdots,u_{n-2},v_{n-1},w_{n-2},\cdots,w_{n-j})$ is defined by equation \eqref{EQ: Bj Delta}. Recall that $u_{n-i},w_{n-i}\in\fu_{\beta_{n-i}}$ and $v_{n-1}\in\fu_{\beta_{n-1}}.$
The image $\iota(g_j)$ is given by 
\[\iota(g_j)=\exp(\iota(u_{n-j}))\cdots \exp(\iota(u_{n-2}))\cdot \exp(\iota(v_{n-1}))\cdot \exp(\iota(w_{n-2}))\cdots \exp(\iota(w_{n-j}))~.\]
Recall the definition of $D_j$ from \eqref{EQ: Dj}. By the definition of $\iota$ we have
\[\iota(g_j)=\begin{dcases}
D_1\big(\frac{v_{n-1}}{\sqrt 2}\big)\cdot D_2\big(\frac{v_{n-1}}{\sqrt 2} \big)&j=1\\
	D_{j+1}(u_{n-j},\cdots,u_{n-2},\frac{v_{n-1}}{\sqrt 2},\frac{v_{n-1}}{\sqrt 2},w_{n-2},\cdots,w_{n-j})&2\leq j
\end{dcases}~.\]
Hence, $\iota(g)=\iota(g_1)\cdots\iota(g_{n-1})$ is in the positive semigroup $\sU_\Delta^{>0}(n,n)$ of $\sSO(n,n).$
\end{proof}

\noindent$\mathbf{\sG=\sSO(n,n+1),\ \Theta=\{\beta_1,\cdots,\beta_{n-1}\}}:$ 
The group $\sSO(n,n+1)$ consists of $(2n+1)\times(2n+1)$-matrices $A\in\sSL(2n+1,\R)$ which preserve the inner product \eqref{EQ: n,n+1 inner product}. The maximal torus $\sT$ and the simple root system $\Delta$ analogous to $\sSO(n,n-1).$ Namely, 
 \[\sT=\{ A=diag(t_1,t_2,\cdots,t_{n},1,t_{n}^{-1},\cdots,t^{-1}_1)\ | t_i\in\R^*\}~,\]
 	\[\ft=\{X=diag(x_1,\cdots,x_{n},0,-x_{n},\cdots,-x_1)\ |\ x_i\in\R\}~,\]
and the simple root system $\Delta=\{\beta_1,\cdots,\beta_{n-2},\beta_{n}\}$ is given by 
\[\beta_j(X)=\begin{dcases}
	x_j-x_{j+1}&1\leq j\leq n-1\\
	x_{n}&j=n
\end{dcases}~.\]
 Consider the subset $\Theta=\{\beta_1,\cdots,\beta_{n-1}\}\subset \Delta$, the parabolic $\sP_\Theta$ has Levi factor $\sL_\Theta$ consisting of matrices of the form
 \begin{equation}
 	\label{EQ Levi of PTheta}
\smtrx{t_1&&&&&& \\ &\ddots&&&&& \\ &&t_{n-1}&&&& \\ &&&A&&& \\&&&&t_{n-1}^{-1}&& \\ &&&&&\ddots&\\ &&&&&&t_1^{-1}}
\ \ \ \ \ \ \text{where}\ \ t_i\in\R^*\ \ \text{and}\ \ A\in\sSO(1,2)~.
 \end{equation}
 For $1\leq j\leq n-2,$ the irreducible $\sL_\Theta$ representations $\fu_{\beta_j}$ associated to the simple roots $\beta_j\in\Theta$ are one dimensional. The irreducible $\L_\Theta$ representation $\fu_{\beta_{n-1}}$ is the three dimension vector space spanned by the root spaces for the positive roots 
 \[\xymatrix{\beta_{n-1}~,&\beta_{n-1}+\beta_n& \text{and} &\beta_{n-1}+2\beta_n}~.\] 
 Thus, the vector space $\fu_{\beta_{n-1}}$ is given by  
 \begin{equation}
 	\label{EQ span light cone}\fu_{\beta_{n-1}}=\langle E_{n-1,n}-E_{n+2,n+3}\ ,\  E_{n-1,n+1}-E_{n+1,n+3}\ ,\  E_{n-1,n+2}-E_{n,n+3}\rangle~.
 \end{equation}

For $(t_1,\cdots,t_{n-1},A)$ in the identity component of the Levi factor $\sL_\Theta^0,$ the action on $\fu_{\beta_i}$ is by 
\[(t_1,\cdots,t_{n-1},A)\cdot x=\begin{dcases}
	t_it^{-1}_{i+1}x&\ \text{for}\ x\in\fu_{\beta_i}\ \text{and}\ 1\leq i\leq n-2\\
	t_{n-1}xA^{-1}& \ \text{for}\ x\in\fu_{\beta_{n-1}}\ \text{and}\ A\in \sSO_0(1,2)
\end{dcases}~.\] 
As before, for $1\leq i\leq n-2$, the positive reals define an invariant acute convex cone $c^\Theta_{\beta_i}\subset\fu_{\beta_i}.$ For $i=n-1,$ the interior of the light cone in $\R^{1,2}$ is the invariant acute convex cone, namely in the basis \eqref{EQ span light cone},
\begin{equation}
	\label{EQ: light cone}c^\Theta_{\beta_{n-1}}=\{(x,y,z)\ |\ 2xz-y^2>0 \}~.
\end{equation}

The element $\beta_\Theta=\beta_{n-1}\in \Theta$ is the unique element of $\Theta$ which, in the Dynkin diagram of $\Delta,$ is connected to $\Delta\setminus\Theta.$ 
The group $\Ww(\Theta)$ from Definition \ref{DEF: W(theta) group} is generated by $\{\sigma_1,\cdots,\sigma_{n-1}\}$, where 
\begin{itemize}
	\item $\sigma_{j}=s_{\beta_j}$ for $1\leq j\leq n-2$, 
	\item $\sigma_{n-1}$ is the longest word in the Weyl group of $\beta_\Theta\cup\Delta\setminus\Theta=\{\beta_{n-1},\beta_{n}\}$. 
\end{itemize}

The group $\Ww(\Theta)$ is isomorphic to the Weyl group of type $B_{n-1}$ with its standard generators. Thus, the longest word $\omega^0_\Theta(\sSO(n,n+1))$ has reduced expression
\[\omega^0_\Theta(\sSO(n,n+1))=b_1\cdots b_{n-1}~,\]
where \[b_j=s_{\beta_{n-j}}\cdot \sigma_{\beta_{n-j-1}}\cdots \sigma_{\beta_{n-2}}\cdot \sigma_{\beta_{n-1}}\cdot \sigma_{\beta_{n-2}}\cdots \sigma_{\beta_{n-j-1}}\cdot \sigma_{\beta_{n-j}}~.\]
Define $B_j^\Theta:c^\Theta_{\beta_{n-j}}\times\cdots\times c^\Theta_{\beta_{n-2}}\times c^\Theta_{\beta_{n-1}}\times c^\Theta_{\beta_{n-2}}\cdots \times c^\Theta_{\beta_{n-j}}\longrightarrow \sU_\Theta$ as in \eqref{EQ: Bj Delta}:
\begin{equation}
	\label{EQ: Bj Theta}
	\xymatrix@R=.2em{B^\Theta_j(u_{n-j},\cdots,u_{n-2},v_{n-1},w_{n-2},\cdots,w_{n-j})=\\\exp(u_{n-j})\cdots \exp(u_{n-2})\cdot \exp(v_{n-1})\cdot \exp(w_{n-2})\cdots \exp(w_{n-j})}.
\end{equation}
The positive semisubgroup $U_\Theta^{>0}\subset\sU_\Theta$ from Definition \ref{DEF: positive semi group} is given by the image of \[B_1^\Theta \cdots B^\Theta_{n-1}~.\]

\begin{Proposition}\label{Prop: Delta to Theta Positivity}
The embedding $\iota_{n,n+1}:\sSO(n,n)\ra\sSO(n,n+1)$ induced by the isometric embedding \eqref{EQ: isometric embedings} maps the positive semigroup of $(\sSO(n,n),\Delta)$ into the positive semigroup of $(\sSO(n,n+1),\Theta)$.
\end{Proposition}
\begin{proof}
Let $E_{ij}$ be the elementary matrix with a $1$ in the $(i,j)$ entry and zero elsewhere. 
The isometric embedding $\iota_{n,n}$ from \eqref{EQ: isometric embedings} induces a map $\iota:\fgl(2n,\R)\ra\fgl(2n+1,\R)$ given by  
	\[\iota(E_{ij})=\begin{dcases}
		E_{ij}& 1\leq i\leq n\ \text{and\ }1\leq j\leq n\\
		E_{i+1,j}& n+1\leq i\leq 2n\ \text{and\ } 1\leq j< n\\
		E_{i,j+1}& 1\leq i< n\ \text{and\ }n+1\leq j\leq 2n\\
		E_{i+1,j+1}& n+1\leq i\leq2n\ \text{and\ }n+1\leq j\leq 2n\\
	\end{dcases}~.\] 
In particular, for $1\leq j\leq n-2$, the restriction of the $\iota$ to the simple $\fu_{\delta_j}$ is the identity  $\iota|_{\fu_{\delta_j}}=Id:\fu_{\delta_j}\ra\fu^\Theta_{\beta_j}$. Thus, for $1\leq j\leq n-2$, $\iota$ maps the positive cone $c^\Delta_{j}$ identically onto the positive cone $c^\Theta_j.$

For $j=n-1$ and $j=n,$ the restriction of $\iota$ to $\fg_{\delta_{j}}$ is the identity: 
\[\xymatrix@C=.5em{\iota|_{\fu_{\delta_{n-1}}}=Id:\fu^\Theta_{\delta_{n-1}}\to\fg_{\beta_{n-1}}\subset\fu^\Theta_{\beta_{n-1}}&\text{and}&\iota|_{\fu_{\delta_{n}}}=Id:\fu_{\delta_{n}}\ra\fg_{\beta_{n-1}+2\beta_n}\subset\fu_{\beta_{n-1}} }~.\]
Recall that the space $\fu_{\beta_{n-1}}^\Theta$ is a direct sum of root spaces $\fu_{\beta_{n-1}}^\Theta=\fg_{\beta_{n-1}}\oplus\fg_{\beta_{n-1}+\beta_{n}}\oplus\fg_{\beta_{n-1}+2\beta_n}.$ The map $\iota$ is given by
\[\iota:\xymatrix@R=.3em{\fu_{\delta_{n-1}}\oplus\fu_{\delta_n}\ar[r]&\fg_{\beta_{n-1}}\oplus\fg_{\beta_{n-1}+\beta_{n}}\oplus\fg_{\beta_{n-1}+2\beta_n}\\(a,b)\ar@{|->}[r]&(a,0,b)}~.\]
Moreover, since $a\in c^\Delta_{\delta_{n-1}}$ and $b\in c^\Delta_{\delta_n}$ implies $2ab-0>0$, $\iota(c^\Delta_{\delta_{n-1}}\times c^\Delta_{\delta_n})\subset c^\Theta_{\beta_{n-1}}$.

Let $\sU^{>0}_\Delta(n,n)$ be the positive subsemigroup for $\sSO(n,n)$. If $g\in\sU^{>0}_\Delta(n,n)$, then 
\[g=g_1\cdot g_2\cdots g_{n}~,\]
where
\[g_j=\begin{dcases}
D_1(v_{n}) & j=1\\
D_2(v_{n-1})&j=2\\
	D_j(u_{n+1-j},\cdots,u_{n-2},v_{n-1},v^j_n,w_{n-2},\cdots,w_{n+1-j})& \text{if} \ 3\leq j\leq n 
\end{dcases}\] are defined by equation \eqref{EQ: Dj}. 
Recall that $u_{n-i},w_{n-i},v_{n-i}\in\fu_{\delta_{n-i}}$.
The image $\iota(g_j)$ is given by 
$\exp(\iota(u_{n+1-j}))\cdots \exp(\iota(u_{n-2}))\cdot \exp(\iota(v_{n-1}))\cdot \exp(\iota(v_{n}))\cdot \exp(\iota(w_{n-2}))\cdots \exp(\iota(w_{n+1-j})).$

Recall the definition of $B^\Theta_j$ from \eqref{EQ: Bj Theta}. By the definition of $\iota$ we have 
\[\iota(g_1)\cdot\iota(g_2)=B_1^\Theta(\iota(v_{n}+v_{n-1}))~,\] and  for $3\leq j\leq n,$
\[\iota_(g_j)=B^\Theta_{j-1}(u_{n+1-j},\cdots,u_{n-2},\iota(v_{n-1}+v_n),w_{n-2},\cdots,w_{n+1-j})~.\]
 Hence, $\iota(g)=\iota(g_1)\cdots\iota(g_{n})$ is in the positive semigroup $\sU_\Theta^{>0}$ of $\sSO(n,n+1).$
\end{proof}

\subsection{Positive $\sSO(n,n+1)$-representations}
We are now ready to prove the reducible representations in the components $\Xx_0(\sSO(n,n+1))$ and $\Xx_{sw_1}^{sw_2}(\sSO(n,n+1))$ are $\sP_\Theta$ positive Anosov representations. For $\Theta=\{\beta_1,\cdots,\beta_{n-1}\},$ the generalized flag variety $\sSO(n,n+1)/\sP_\Theta$ consists of the set flags 
\[V_1\subset\cdots\subset V_{n-1}\subset V_{n-1}^\perp\subset\cdots\subset V_1^\perp\subset\R^{2n+1}~,\] where $V_j\subset\R^{2n+1}$ is an isotropic $j$-plane. 
We start with the following proposition.

\begin{Proposition}\label{Prop: singular Theta Pos reps}
Let $\rho:\Gamma\ra\sSO(n,n+1)$ be a representation. If there is a finite index subgroup $\widehat\Gamma\subset\Gamma$ such that the restriction of $\rho$ to $\widehat\Gamma$ factors though an $\sSO(n,n-1)\times\sSO(2)$-representation with $\sSO(n,n-1)$ factor a Hitchin representation or $\rho$ factors through an $\sSO(n,n)$-Hitchin representation, then $\rho$ is a positive  $\sP_\Theta$ Anosov representation. 
\end{Proposition}
\begin{proof}
Let $\rho_0:\Gamma\ra\sSO(n,n-1)$ and $\rho_0':\Gamma\ra\sSO(n,n)$ be Hitchin representations. By \cite{fock_goncharov_2006,AnosovFlowsLabourie}, there are $\rho_0$ and $\rho_0'$ equivariant positive Anosov boundary curves
\[\xymatrix{\xi_{\rho_0}:\p_\infty\Gamma\ra\sSO(n,n-1)/B_{n,n-1}&\text{and}&\xi_{\rho_0'}:\p_\infty\Gamma\ra\sSO(n,n)/B_{n,n}}~,\]
where $B_{n,n-1}\subset\sSO(n,n-1)$ and $B_{n,n}\subset\sSO(n,n)$ are Borel subgroups. 
The embeddings 
\[\xymatrix{\sSO(n,n-1)\ar[r]^{\ \ \iota_{n,n-1}}&\sSO(n,n)\ar[r]^{\iota_{n,n}\ \ }&\sSO(n,n+1)}\] induced by \eqref{EQ: isometric embedings} induce maps 
\[\xymatrix{\sSO(n,n-1)/B_{n,n-1}\ar[r]^{\ \ \iota_{n,n-1}}&\sSO(n,n)/B_{n,n}\ar[r]^{\iota_{n,n}\ \ }&\sSO(n,n+1)/\sP_\Theta}~.\]  
By Propositions \ref{Prop: Delta to Delta Positivity} and \ref{Prop: Delta to Theta Positivity} the Anosov boundary curves 
\[\xymatrix@R=.2em{\iota_{n,n}\circ\iota_{n,n-1}\circ\xi_{\rho_0}:\p_\infty\Gamma\longrightarrow\sSO(n,n+1)/\sP_\Theta&\text{and}\\\iota_{n,n}\circ\xi_{\rho_0'}:\p_\infty\Gamma\longrightarrow\sSO(n,n+1)/\sP_\Theta}\]
are $\Theta$-positive. 
The centralizer of $\iota_{n,n}(\iota_{n,n-1}(\sSO(n,n-1)))$ in $\sSO(n,n+1)$ contains $\sSO(n,n-1)\times\sSO(2).$ 
Thus, by Remark \ref{REM: FACTS about ANOSOV}, if $\rho:\Gamma\ra\sSO(n,n+1)$ is a representation and there exists a finite order subgroup $\widehat\Gamma\subset\Gamma$ such that the restriction of $\rho$ to $\widehat\Gamma$ factors though an $\sSO(n,n-1)\times\sSO(2)$-representation with $\sSO(n,n-1)$ factor a Hitchin representation or $\rho$ factors through an $\sSO(n,n)$-Hitchin representation, then $\rho$ is a positive  $\sP_\Theta$ Anosov representation.
\end{proof}
\begin{Theorem}\label{THM4}
	Let $\sSO(n,n+1)/P_\Theta$ be the generalized flag variety of flags 
	\[V_1\subset\cdots\subset V_{n-1}\subset V_{n-1}^\perp\subset\cdots\subset V_1^\perp\subset\R^{2n+1}~,\] where $V_j\subset\R^{2n+1}$ is an isotropic $j$-plane. If $n\geq2,$ then the set of representations in $\Xx_0(\sSO(n,n+1))$ or $\Xx_{sw_1}^{sw_2}(\sSO(n,n+1))$ which are not irreducible is a nonempty set which consists of positive $\sP_\Theta$ Anosov representations. 
\end{Theorem}
\begin{proof}
By Theorem \ref{THM: Zariski closures}, every reducible representation in the connected components $\Xx_0(\sSO(n,n+1))$ and $\Xx_{sw_1}^{sw_2}(\sSO(n,n+1))$ satisfy the hypothesis of Proposition \ref{Prop: singular Theta Pos reps}. Thus, all reducible representations in $\Xx_0(\sSO(n,n+1))$ or $\Xx_{sw_1}^{sw_2}(\sSO(n,n+1))$ is a positive $\sP_\Theta$ Anosov representation.\end{proof}
\begin{Remark}
	\label{REM: Facts about Pos reps}In \cite{PosRepsGLW}, it is shown that the set of positive representations is open, and it is conjectured to be closed. In fact, it can be shown that the set of positive Anosov representations is closed in the set of irreducible representations \cite{AnnaPrivateCommunication}. Namely, let $\rho_j:\Gamma\to \sSO(n,n+1)$ is a sequence of positive $\sP_\Theta$ Anosov representation which converge to $\rho_\infty:\Gamma\to\sSO(n,n+1).$ If each $\rho_j$ is irreducible (see Definition \ref{DEF irreducible REP}) and $\rho_{\infty}$ is also irreducible, then $\rho_{\infty}$ is a positive $\sP_\Theta$ Anosov representation.

Assuming these results, Theorem \ref{THM4} can be significantly strengthened to the statement that the components $\Xx_0(\sSO(n,n+1))$ and $\Xx_{sw_1}^{sw_2}(\sSO(n,n+1))$ consist {\em entirely of Anosov representations}. 

The argument is as follows: Let $\rho$ be a reducible representation in $\Xx_0(\sSO(n,n+1))$ or $\Xx_{sw_1}^{sw_2}(\sSO(n,n+1)).$ 
Since positive representations define an open set in the character variety, there is an open neighborhood $U_{\rho}$ of $\rho$ consisting of $\Theta$-positive representations. 
In particular, there exists $\rho\in U_{\rho_0}$ which is irreducible. Since positivity is closed in the set of irreducible representations, all irreducible representations $\rho\in\Xx_0(\sSO(n,n+1))$ are $\Theta$-positive. By Theorem \ref{THM4} all representations in $\Xx_0(\sSO(n,n+1))$ and $\Xx_{sw_1}^{sw_2}(\sSO(n,n+1))$ are positive $\sP_\Theta$ Anosov representations.
 \end{Remark}

\bibliography{../../mybib}{}
\bibliographystyle{plain}

\end{document}